\documentclass[12pt,a4paper]{ouparticle}
\usepackage{geometry}
\usepackage{amssymb}
\usepackage{euscript}
\RequirePackage{setspace}
\usepackage[linkcolor=green!40!blue,colorlinks=green!40!blue]{hyperref} 
\hypersetup{colorlinks=true, citecolor=blue, filecolor=Maroon,
 linkcolor=blue, urlcolor=black!65!purple}

\usepackage{xcolor,colortbl}

\definecolor{Gray}{gray}{0.35}
\definecolor{LightCyan}{rgb}{0.88,1,1}

 \usepackage{amsthm}
\usepackage[OT2,T1]{fontenc}

\newcommand{\Sym}{\mathrm{Sym}_{\A_X}}

\newcommand{\AD}{\mathcal{A}[\mathcal{D}]}

\newcommand{\I}{\mathcal{I}}

\newcommand{\MD}{\scalemath{1.02}{\mathfrak{M}}_{\mathcal{D}_X}}
\newcommand{\D}{\mathcal{D}}

\newcommand{\A}{\mathcal{A}}

\newcommand{\M}{\mathcal{M}}

\newcommand{\EQ}{\mathcal{Y}}
\newcommand{\IAB}{\mathcal{I}\rightarrow \mathcal{A}\rightarrow \mathcal{B}}

\newcommand{\DG}
{{\color{white!10!black}\mathbf{dg}_{\mathcal{D}_X}}}

\usepackage{tikz-cd}
\tikzset{%
    symbol/.style={%
        draw=none,
        every to/.append style={%
            edge node={node [sloped, allow upside down, auto=false]{$#1$}}}
    }
}
\usepackage{adjustbox}
\numberwithin{equation}{section}

\usepackage{aliascnt}
\usepackage{hyperref, cleveref}
\newtheorem{thm}{Theorem}[section]
\newaliascnt{lem}{thm}

\aliascntresetthe{lem}

\newaliascnt{sublem}{thm}

\aliascntresetthe{sublem}

\newaliascnt{warn}{thm}

\aliascntresetthe{warn}

\newaliascnt{ass}{thm}

\aliascntresetthe{ass}

\newaliascnt{prop}{thm}
\newtheorem{prop}[prop]{Proposition}
\aliascntresetthe{prop}

\newaliascnt{cor}{thm}
\newtheorem{cor}[cor]{Corollary}
\aliascntresetthe{cor}

\newaliascnt{defn}{thm}
\newtheorem{defn}[defn]{Definition}
\aliascntresetthe{defn}

\newaliascnt{ex}{thm}
\newtheorem{ex}[ex]{Example}
\aliascntresetthe{ex}
\usepackage{physics}
\newaliascnt{obs}{thm}
\newtheorem{obs}[obs]{Observation}
\aliascntresetthe{obs}
\usepackage{tikz-cd}
\tikzset{%
    symbol/.style={%
        draw=none,
        every to/.append style={%
            edge node={node [sloped, allow upside down, auto=false]{$#1$}}}
    }
}
\usepackage{adjustbox}
\newcommand\scalemath[2]{\scalebox{#1}{\mbox{\ensuremath{\displaystyle #2}}}}

\newaliascnt{rmk}{thm}
\newtheorem{rmk}[rmk]{Remark}
\aliascntresetthe{rmk}

\newaliascnt{notate}{thm}
\newtheorem{notate}[notate]{Notation}
\aliascntresetthe{notate}

\newaliascnt{rem}{thm}
\newtheorem{rem}[rem]{Reminder}
\aliascntresetthe{rem}

\newaliascnt{coords}{thm}

\aliascntresetthe{coords}

\newaliascnt{interpret}{thm}

\aliascntresetthe{interpret}

\aliascntresetthe{ex}

\newaliascnt{term}{thm}

\aliascntresetthe{term}

\newaliascnt{infmdefn}{thm}

\aliascntresetthe{infmdefn}

\newaliascnt{cons}{thm}
\newtheorem{cons}[cons]{Construction}
\aliascntresetthe{cons}

\usepackage{authblk}

\AtEndDocument{%
  \par
  \medskip
 \flushleft \begin{tabular}{@{}l@{}}%
    \textsc{\small{Jacob Kryczka}}\\
    \textsc{\footnotesize{Beijing Institute of Mathematical Sciences and Applications (BIMSA)}}
    \\
    \scriptsize{\textit{E-mail}: \texttt{jkryczka@bimsa.cn}}
    \\
    \\
    \textsc{\small{Artan Sheshmani}}\\
    \textsc{\footnotesize{Beijing Institute of Mathematical Sciences and Applications (BIMSA)}}
    \\
    \textsc{\footnotesize{Massachusetts Institute of Technology (MIT), IAiFi Institute} }
\\
\textsc{\footnotesize{Laboratory of Mirror Symmetry, Higher School of Economics (NRU HSE)}}
\\
   \scriptsize{\textit{E-mail}: \texttt{artan@mit.edu}}
  \end{tabular}}

\begin{document}
  
\title{The $\D$-Geometric Hilbert Scheme -- Part I:\\
\Large{Involutivity and Stability}}

\author{%
\name{\normalsize{Jacob Kryczka and Artan Sheshmani}}}

\abstract{We construct a moduli space of formally integrable and involutive ideal sheaves arising from systems of partial differential equations (PDEs), by introducing the $\mathcal{D}$-Hilbert and $\mathcal{D}$-Quot functors in the sense of Grothendieck and establishing their representability. Central to this construction is the notion of Spencer (semi)stability, which presents an extension of classical stability conditions from gauge theory and complex geometry, and which provides the boundedness needed for our moduli problem.

As an application, we show that for flat connections on compact Kähler manifolds, Spencer-polystability of the associated PDE ideal is equivalent to the existence of a Hermitian–Yang–Mills metric. This result provides a refinement of the classical Donaldson–Uhlenbeck–Yau correspondence, and identifies Spencer cohomology and stability as a unifying framework for geometric PDEs.} 


\date{\today}


\maketitle
\tableofcontents

\section{Introduction}
\label{sec: Introduction}
This paper develops a moduli-theoretic framework for studying partial differential equations (PDEs) as geometric objects, using tools from Geometric Invariant Theory (GIT) and derived algebraic geometry. Our approach is grounded in the geometry of jet spaces and the theory of formal integrability \cite{G,G2,Q,KLV,KL}, and draws inspiration from classical moduli problems in algebraic geometry \cite{Gro2,Gro3,HL,Ta,M}.

We adopt the perspective of Vinogradov's school \cite{Vin,Vin2,KV}, in which differential complexes, rather than functional-analytic methods, encode the fundamental invariants of PDEs, such as symmetries, conservation laws and variational structures, among others. Building on this viewpoint, we formulate a moduli-theoretic approach to PDEs wherein Spencer cohomology acts as a unifying cohomological invariant \cite{Sp,Sp2}. This framework specializes to classical theories including de Rham, Dolbeault, flat connection, and foliated cohomologies when applied appropriately to their corresponding PDEs, while also providing new tools for the analysis of nonlinear systems.

This cohomological structure leads naturally to a notion of \emph{Spencer stability}, a refinement of classical stability concepts in geometric analysis. We propose that two seemingly distinct stability theories, namely the Donaldson–Uhlenbeck–Yau (DUY) correspondence 
\cite{D,UY} and $K$-stability, arise naturally as instances of a broader Spencer-theoretic paradigm.

As a first instance of this principle, we prove the following result concerning flat connections on compact Kähler manifolds:
\\

\noindent\textbf{Theorem.}\hspace{1mm} (Theorem \ref{thm: DSUY implies DUY})
\emph{Let $(X,\omega)$ be a compact Kähler manifold. The differential ideal encoding a flat connection on a holomorphic vector bundle is Spencer-polystable if and only if the bundle admits a Hermitian–Yang–Mills metric.}
\\

This result indicates that Spencer stability is linked to the existence of canonical geometric structures. We conjecture an analogous relationship for Kähler–Einstein metrics, where the Monge–Ampère equation on a Fano manifold may admit a Spencer-semistable formulation linked to $K$-stability. While this connection remains open, it offers a new perspective on the stability and existence problem of canonical metrics.

In forthcoming work \cite{KSh2}, we construct derived enhancements of the moduli spaces of involutive PDE systems. These appear as $\D$-geometric Hilbert and Quot \emph{dg-schemes}, following modern methods in derived algebraic geometry \cite{CFK,CFK2,BKS,BKSY2,T,TV}. 

These enhancements capture higher obstructions and homotopical invariants of PDE systems, generalizing classical moduli constructions while enabling new analytic techniques. They also connect with recent parallel developments, such as Borisov’s derived moduli space of Hermitian–Einstein connections \cite{Bo}, and the Haiden–Katzarkov–Kontsevich–Pandit program on categorical Kähler geometry. 

In this context, derived moduli provide new analytic tools for studying nonlinear PDEs, extending beyond classical energy methods and toward homotopical and categorical stability conditions (e.g. of Bridgeland-type \cite{Br}).

These perspectives point toward a broader unification of geometric, analytic, and categorical approaches to PDEs. The present paper lays the foundation by introducing the moduli-theoretic formalism and demonstrating its relevance to classical stability problems in geometric analysis and algebraic geometry.

\subsection{Background}
Let $X$ be a smooth algebraic variety or complex analytic manifold of dimension $n$, equipped with a holomorphic vector bundle $E$ of rank $m$. We study systems of nonlinear partial differential equations (PDEs) for sections $u=u(x)$ of the form
\begin{equation}
    \label{eqn: PDE}
F(x,u(x),\partial_iu,\partial_{ij}u,\ldots,\partial_x^{\sigma}u)=0,|\sigma|\leq k,
\end{equation}
where $F$ is analytic in the base variables $x\in X$ but polynomial in the derivatives of $u$. Combining the formal theory of PDEs \cite{G,G2,Sp,Sp2,Q}, for analyzing integrability conditions and prolongation structures with the geometry of jet spaces \cite{KLV,KL,KV}, which provides the natural setting for studying differential equations geometrically as subvarieties of jet bundles, we treat (\ref{eqn: PDE})  as an algebraic subvariety of the $k$-jet bundle $J_X^kE$, where the differential structure is encoded in the Cartan distribution $\mathcal{C}.$

The Cartan-Kuranishi prolongation theory provides fundamental tools for analyzing formal integrability of differential systems, characterizing obstructions to solvability through successive differentiation \cite{Ku}. For PDE systems of order $\leq k$ modeled as subvarieties $Z=\{F(x,u,u_{\sigma})=0\}\subset J_X^kE,$ where $J_X^kE$ denotes the $k$-jet bundle of $E\rightarrow X$, the infinite prolongation $Z^{\infty}$ encodes the complete differential consistency conditions. This construction, governed by the Cartan-Kähler theorem, yields a pro-algebraic variety with finite-dimensional distribution $\mathcal{C}$.

Building on Vinogradov's fundamental insights in the context of differential geometry \cite{KLV,Vin,Vin2}, we interpret $Z^{\infty}$ more algebro-geometrically as a $\D$-scheme, as introduced by Beilinson-Drinfeld \cite{BD}. This is a geometric object whose coordinate ring $\mathcal{O}(Z^{\infty})$ carries a natural $\D_X$-algebra structure induced by the total derivative operators which span $\mathcal{C}.$ This perspective, elaborated further in \cite{P,P2,KSY,KSY2}, establishes an algebraic $\D$-geometry framework where PDEs are studied through their associated $\D$-ideal sheaves $\mathcal{I}\subset \mathcal{O}(Z^{\infty})$, closed under the natural $\D_X$-action.

Motivated by Grothendieck's Hilbert scheme \cite{Gro,Gro2,Gro3} and its derived enhancements \cite{BKS,BKSY2,CFK,CFK2}, we construct the $\D$-Hilbert functor $\underline{\mathcal{H}ilb}_{\D_X}^P(J_X^{\infty}E)$ classifying $\D$-ideal sheaves with fixed numerical invariant $P_{\D}$ (a PDE-theoretic Hilbert polynomial) and prove its representability by a finite-type $\D$-Hilbert scheme $\mathrm{Hilb}_{\D}^{P}.$
\\

\noindent\textbf{Theorem.} (Theorem \ref{MainTheorem}) \emph{Fix a numerical polynomial $P\in \mathbb{Q}[t].$ There exists a representable moduli functor classifying formally integrable (non-singular) algebraic differential systems with fixed (Spencer) regular symbolic behavior and whose numerical polynomial $P_{\D}$ is equal to $P.$}
\\

This result provides a natural generalization of classical Hilbert schemes to differential-algebraic geometry and is the foundation for derived moduli spaces of PDEs (developed in \cite{KSh2}).
The $\D$-Hilbert functor shares conceptual foundations with classical Hilbert schemes but differs in several important respects. Most notably, its associated Hilbert polynomials are intrinsically affine invariants that require no polarization or weighted projective spaces, reflecting the natural affine structure of jet spaces. 

A key technical challenge lies in controlling the homological complexity of $\D$-ideal sheaves to obtain bounded families. This problem is resolved through involutivity conditions on the symbol, following ideas of Cartan \cite{C,C3}.

The notion of \emph{Spencer regularity} (Definition \ref{defn: D-Geometric Regularity}) and the existence of finite free resolutions are central to our construction, and encode a cohomological vanishing requirement for the solution sheaf associated to the linearized PDE system. More precisely, it governs the acyclicity of the $\D$-geometric characteristic module, which is the algebraic counterpart to the linearized equation's solution space. This condition emerges naturally when examining the Spencer complex resolution of the symbol module, where regularity ensures the appropriate exactness properties needed for our moduli-theoretic arguments. The criterion serves as a differential-algebraic analogue of Castelnuovo-Mumford regularity in classical algebraic geometry \cite{M}, adapted to account for the additional structure imposed by the differential operators acting on the system \cite{Ma3} (see also \cite{Sei}).

Our approach synthesizes insights from geometric PDE theory (Cartan-Kähler formalism \cite{KLV}), differential algebra (Ritt-Kolchin theory \cite{Ri}), microlocal analysis \cite{K,Sch}, and sheaf theory \cite{KS}, revealing unexpected naturality when viewed through modern derived geometric lenses \cite{T,TV,STV,L}.

The combination of these perspectives establishes a new moduli-theoretic framework for geometric PDEs that provides new approaches to classical problems through derived methods while opening new research directions at the intersection of analytic moduli problems, higher structures, and classical existence questions in differential geometry and geometric analysis.

\subsection{Overview and statement of results}
\label{ssec: Overview and statement of results}

Our main results, which were informally stated above, classify admissible $\mathcal{D}_X$-ideal sheaves arising from formally integrable and involutive PDE systems within the universal commutative $\mathcal{D}_X$-algebra of algebraic jets. Geometrically, these correspond to sub-crystals (or sub-local systems) of schemes \cite{Si,KP}, capturing various classes of algebraic differential equations. The admissibility criteria are rooted in the foundational work of Goldschmidt \cite{G, G2}, Quillen \cite{Q}, and Spencer \cite{Sp,Sp2} on formal integrability, combined with Cartan’s involutivity conditions \cite{C,C3}.

We begin by rephrasing the theory of analytic integrability in algebro-geometric terms. This is summarized in Table \ref{Table: AlgGeom v.s. DGeom}, borrowed from \cite{P3}:

\begin{table}[ht]
\caption{\small{Comparing $\D$-Geometry I}}
\label{Table: AlgGeom v.s. DGeom}
\centering
\scalebox{.76}{
\begin{tabular}{c|c|c}
\hline
Concept & \textbf{Algebraic Geometry}& \textbf{$\mathcal{D}$-Geometry}\\
\hline
 \textit{Formula} & $P(x)=0$ & $F(x,u,\partial^{\sigma}u)=0$  \\
\textit{Algebraic structure} & Commutative $\mathbf{k}$-algebra $A$ & Commutative $\mathcal{D}$-algebra $\mathcal{A}$ \\

\textit{Free structure} & $P\in\mathbf{k}[x]$ & $F\in \mathcal{A}:=\mathcal{O}\big(J^{\infty}(\mathcal{O}_E)\big)$ \\
 \textit{Solution space} & $\{x\in A|P(x)=0\}$ & $\{a\in \mathcal{A}|F(a)=0\}$  \\ 
 \textit{Affine object} & $\mathrm{Spec}_k(A)$ & $\mathrm{Spec}_{\mathcal{D}}(\mathcal{A})$
\\
\textit{Representability} & $\mathrm{Sol}_k(P=0)\simeq \mathrm{Spec}(A/P)$ & $\mathrm{Sol}_{\mathcal{D}}(F=0)\simeq\mathrm{Spec}_{\mathcal{D}}(Jets(\mathcal{O}_C)/F)$
\\
\hline
\end{tabular}}
\end{table}

We take this analogy further, enabling the numerical classification of $\mathcal{D}$-finitely generated ideals associated with differentially consistent systems and establishing a connection between differential-algebraic constraints and moduli-theoretic invariants. See Table \footnote{In this table, $M_*=\bigoplus_{i}H^0(X,\mathcal{F}(i))$ is the graded $R_*$-module with $\mathcal{F}\in \mathrm{Coh}(X)$ for a projective variety $(X,\mathcal{O}_X(1))$. The PDE analog is the associated graded $\D$-module $\mathrm{gr}^F(\Omega_{\EQ}^1)$ associated to the symbol (see $\mathcal{C}h^{\EQ}(*)$ (\ref{sssec: D-Geometric Microcharacteristic Varieties})), viewed as a graded $\mathcal{O}_{T^*(X;\EQ)}$-module.} \ref{Table: AlgGeom v.s. DGeom II}:

\begin{table}[ht]
\caption{\small{Comparing $\D$-Geometry II}}
\label{Table: AlgGeom v.s. DGeom II}
\centering
\scalebox{.78}{
\begin{tabular}{c|c|c}
\hline
Concept & \textbf{Algebraic Geometry}& \textbf{$\mathcal{D}$-Geometry}\\
\hline
\textit{Geometric} & Ideal sheaf $\EuScript{I}_Z\subset \mathcal{O}_X$ & $\D$-Ideal sheaf $\mathcal{I} \subset \mathcal{A}^{\ell}$ 
 \\
\textit{Numerical} & Hilbert polynomial $P_{\mathcal{O}_Z}$ &  $\D$-Hilbert polynomial $P_{\mathcal{A}/\mathcal{I}}$ (\ref{ssec: D-Hilb Polynomials}) \\
\textit{Graded ring} & $R_*=\oplus_t \Gamma(X,\mathcal{O}_X(t))$ & $\mathcal{O}_{T^*(X;\EQ)}(*)=\oplus_t\mathcal{O}_{T^*(X;\EQ)}(t)$ \\
\textit{Sheaves} & $\mathcal{F}\in \mathrm{Coh}(X)$ & $\Omega_{\EQ}^1\in\mathrm{Mod}(\mathcal{O}_{\EQ}[\mathcal{D}_X])$
\\
\textit{Graded module} & $M_*\in \mathrm{Mod}^{gr}(R_*)$ & $\mathrm{Gr}(\Omega_{\EQ}^1)\in \mathrm{Mod}^{gr}(\mathcal{O}_{T^*(X;\EQ)})$ (\ref{defn: Strict Char modules})
\\
\textit{Sheaf theoretic} & Support & Microsupport (characteristic variety)
\\
\textit{Boundedness} & Castelnuovo-Mumford regularity & Degree of involutivity (\ref{defn: D-geom m-involutive})
\\
\textit{Moduli} & Hilbert scheme $Hilb_X$ &  $\D$-Hilbert scheme $\mathbf{Hilb}_{\mathcal{D}}^{P}$ (\ref{eqn: D-Hilbert}) \\ 
\hline
\end{tabular}}
\end{table}

Following \cite{Ma3}, the fundamental observation underpinning this work is that the geometric symbol of an involutive PDE induces, at each point, a certain polynomial comodule structure. This insight allows us to employ Koszul homology—or equivalently, its $\mathbb{R}$-dual formulation via (nonlinear) Spencer cohomology, as is more standard in PDE geometry—to characterize the Castelnuovo-Mumford regularity of these modules \cite{M}.

We introduce the notions of \emph{Spencer regularity} and \emph{degree of involution} in Subsection \ref{ssec: Regularity criterion for D-schemes}, providing a homological criterion for the involutivity of differential systems, and after recalling foundational aspects of Spencer cohomology, in Section \ref{sec: Moduli Space of D-Ideal Sheaves} we introduce $\mathcal{D}$-geometric numerical polynomials ($\mathcal{D}$-Hilbert polynomials) associated with $\mathcal{D}$-ideal sheaves $\mathcal{I}$ and their symbolic systems $\mathrm{gr}(\mathcal{I})$. In Subsection \ref{ssec: Spencer stability}, we define a reduced $\mathcal{D}$-Hilbert function (see (\ref{eqn: Reduced D-Hilb})), leading to a natural stability condition (Definition \ref{defn: Spencer ss}) and the concept of Spencer slopes $\mu_{Sp}(\mathcal{I})$ (given by (\ref{eqn: Spencer slope})).

Following the arguments of Grothendieck \cite{Gro3}, in Subsection \ref{ssec: D-Hilb and Quot} we construct the $\mathcal{D}$-Hilbert and $\mathcal{D}$-Quot moduli functors. It should be emphasized though due to the pro-finite nature of $J_X^{\infty}E$, our approach diverges from classical Hilbert scheme constructions, instead following methods present in Douady’s treatment of complex analytic subspaces \cite{Dou}.

In Section \ref{sec: Main Result Proof} we present the main theorem of this paper which establishes the existence, representability, and finiteness of the $\mathcal{D}$-Hilbert moduli functor. The key technicality here is proving the existence of uniform bounds on Spencer regularities across families. 
We prove the three results in this direction.

Let a group $G$ act on $E$ by diffeomorphisms. Its action lifts to jet-spaces $J_X^kE$ and restricts to systems $Z_k\subset J_X^kE.$ We are interested in the compatibility of this action with prolongation and with description of the quotient variety in the projective limit. In this setting, one speaks of \emph{(Lie) pseudogroups}, which are certain infinite-dimensional Lie groups obtained by integrating Lie equations \cite{C2,E,Ku2,KuSp,SiSt}.

Pesudogroups act on PDE variables $u_{\sigma},|\sigma|>1$, so because we are in the (almost) algebraic category, it is sufficient to consider \emph{algebraic pseudgroups} (recalled in Subsect. \ref{ssec: Differential algebra quotients}). 
\\

\noindent\textbf{Theorem} (GIT for PDEs, Proposition \ref{prop: GeometricQuotientEquation}) \emph{Given an algebraic action by a (formally integrable) Lie pseudogroup $G:=\{G^k\}_{k\geq 0},$ on a $\D$-scheme with ideal $\mathcal{I}$, the quotient equation exists and is again an algebraic $\D$-scheme, with ideal $\mathcal{I}^{G}$.}
\\

In particular, the (categorical) quotient $\D$-scheme $[Z/G]:=\{[Z^k/G^k]\}_{k\geq 0}$, corresponds to a geometric quotient. 

The next result establishes invariance for regularity of $\D$-ideals under restriction.
To motivate it, recall that in constructing the usual Hilbert scheme for coherent sheaves $F$ on projective space $\mathbb{P}^n$, one proves that if $F$ is $m$-regular, so is $F|_{H}$, for any hyperplane $H\simeq \mathbb{P}^{n-1}$ (not containing associated points of $F$, see e.g. \cite{Gro3}).

For PDEs, the analogous property is related to a sub-scheme being non-characteristic. By applying an analog of Noether normalization lemma to characteristic varieties (see, (\ref{eqn: Filtered Char ideal})), the set of all non microcharacteristic subspaces of $T^*X$ of fixed dimension define an open dense subset of the corresponding Grassmannian. Consequently, we prove the following.
\\

\noindent\textbf{Theorem} 
 (PDE Lefschetz Principle, Proposition \ref{prop: Inv Res}). \emph{Let $H\subset X$ be a non-characteristic sub-variety. Then the Spencer regularity of $\mathcal{I}$ and $\mathcal{I}|_{H}$ agree.}
\\

We then prove the following important result.
\\

\noindent\textbf{Theorem} (Boundedness, Proposition \ref{prop: Boundedness 1}) \emph{The moduli problem for differential ideals $\mathcal{I},\mathcal{I}^{G}$ is bounded, with finite-dimensional parameter spaces.}
\\

With these results, we prove the first main result of the paper.
\\

\noindent\textbf{Theorem.} (Theorem \ref{MainTheorem})\emph{
Let $X$ be a smooth proper $k$-scheme of dimension $n,$ and let $m\in \mathbb{N}$ represent the number of dependent variables. Let $P$ be a numerical polynomial. There is a natural functor 
$$\underline{\MD}(P;n,m)^{\mathrm{inv}}:\big(\mathcal{D}_X-\mathrm{Alg}_X^{op}\big)\rightarrow \mathrm{Sets},$$
which parameterizes $\mathcal{D}$-flat families of involutive ideal sheaves corresponding to formally integrable (almost) algebraic differential equations with Hilbert polynomial equal to $P.$
There exists a moduli sub-functor $\mathfrak{M}_{\D_X}^{ss}(P;n,m)^{\mathrm{inv}}$ whose points are semi-stable differential ideals which is isomorphic to the $\D$-Hilbert scheme of involutive $\D$-ideal sheaves $\mathcal{H}\mathrm{ilb}_{\mathcal{D}_X}^{P}\big(J_X^{\infty}E\big)$ with fixed $\D$-Hilbert polynomial \emph{(\ref{ssec: D-Hilb Polynomials})} equal to $P.$ It is representable by an ind-scheme of finite-type with compatible $\D$-action.}
\\


We then initiate a study of the deformation-obstruction theory for $\D$-ideal sheaves, which will be elaborated further using derived deformation theory and the cotangent complex \cite{I}, in the sequel \cite{KSh2}.
Our second main result computes the Zariski tangent space to the $\D$-Hilbert scheme, and identifies the corresponding space of obstructions.
\\

\noindent\textbf{Theorem.} (Theorem \ref{MainTheorem2})\emph{The tangent space at a point $[\mathcal{I}]$ controlling first-order deformations of $\mathcal{I}$ as a $\D$-ideal, is given by
    $$T_{[\mathcal{I}]}\mathbf{Hilb}_{\D_X}(J_X^{\infty}E)\simeq Hom_{\D_X}(\mathcal{I},\mathcal{O}(J_X^{\infty}E)/\mathcal{I}).$$
    There is a space of obstructions $\mathcal{O}bs_{[\mathcal{I}]}$ which is a filtered $\D$-module over the $\D$-scheme defined by $\mathcal{I}$ whose associated graded algebra is isomorphic to a sub-quotient of the $Ext^1$ group giving the (truncated) Spencer cohomology determined by the symbol of $\mathcal{I}.$}
\\

Following (e.g. \cite{Gi,HL,M,Ta}), we introduce a notion of (semi)-stability for $\D$-ideal sheaves that we call \emph{Spencer (semi)-stability} (see Definition \ref{defn: Spencer ss}). It has close relations to the notions of (semi)-stability appearing in related moduli problems: 
\begin{itemize}
    \item[-] Nitsure and Sabbah's (semi)-stability for moduli of (pre-)$\D$-modules \cite{NS} and more general regular holonomic $\D$-modules with normal crossing singularities \cite{Ni}. In \emph{loc.cit} it was proven that there exists a moduli space of semistable regular holonomic $\D$-modules which is moreover constructed via GIT quotient of a Quot scheme. It is a quasi-projective scheme.

    \item[-] Simpson's stability for flat bundles $(E,\nabla)$ \cite{Si} also elaborated later by Katzarkov-Pantev \cite{KP}. The moduli space $\M_{DR}^{ss}(X/S)$ of semistable flat bundles is proven to be a quasi-projective variety (via non-abelian Hodge theory). A similar problem studied by Pantev-To\"en using derived algebraic geometry appears in \cite{PT}, using methods which do not require regularity assumptions.
    \end{itemize}

In addition to possessing close relations with stability conditions appearing in the works above, due to the universality of Spencer cohomology $\mathcal{H}_{\mathrm{Sp}}^*(Z_P)$ associated to the PDE manifold $Z_P$ defined by a differential operator $P$, we may study the implication of Spencer stability in cases of interest in algebraic and complex geometry and in connection with moduli spaces of interest in gauge theory:
\begin{enumerate}
    \item When $P$ is a determined system of linear PDEs, then $\mathcal{H}_{\mathrm{Sp}}^0(X;Z_P)=\mathrm{ker}(P)=\mathrm{Sol}(P)$ and $\mathcal{H}_{\mathrm{Sp}}^1(X;Z_P)=\mathrm{coker}(P)$.
    
    \item When $P$ is the de Rham differential, $\mathcal{H}_{\mathrm{Sp}}^*(X;Z_{d_{DR}})\simeq \mathcal{H}_{DR}^*(X;\mathbb{C})$ are isomorphic to the Hodge bundles.
    
    \item For flat bundles $(E,\nabla)$, we have $\mathcal{H}_{\mathrm{Sp}}^*(X;Z_{d_{\nabla}})\simeq \mathcal{H}_{\nabla}^*(X;E),$ where $d_{\nabla}$ is the covariant derivative.
    
    \item When $X$ is a complex manifold with $E$ a holomorphic bundle, denote by $\Omega_X^{p,q}(E)$ the sheaf of $(p,q)$ forms on $X$ with values in $E.$ Then Spencer cohomology of the Cauchy-Riemann equation $\overline{\partial}:\Omega_{X}^{p,0}(E)\rightarrow \Omega_X^{p,q}(E)$ is isomorphic to the Dolbeault cohomology, $\mathcal{H}_{\mathrm{Sp}}^*(X;Z_{\overline{\partial}})\simeq \mathcal{H}_{\overline{\partial}}^*(X;\Omega_X^p(E)).$ 
    
    \item When $Z=\{Z_k\subset J_X^kE\}$ is an integrable system of finite type, meaning prolongation stabilized e.g. $\mathrm{Pr}_{1}(Z_{\ell})\simeq Z_{\ell}$ for large $\ell,$ then there exists a canonical flat bundle $(Z_{\ell},\nabla^{\ell})$ over $X$ and Spencer cohomology of $Z$ is isomorphic to flat connection cohomology, $\mathcal{H}_{\mathrm{Sp}}^*(X;Z_{ft})\simeq \mathcal{H}_{\nabla^{\ell}}^*(X;Z_{\ell}).$
\end{enumerate}
These examples, in particular (2),(3) and (4), suggest to study the relationship of Spencer-based stability conditions with classical stability conditions and their implications with classical correspondences between moduli spaces of stable objects \cite{D,UY}. We prove the following result in this direction.
\\

\noindent\textbf{Theorem.} (Theorem \ref{thm: DSUY implies DUY})\emph{ Let $(X,\omega)$ be a compact Kähler manifold, $(E,\nabla)$ a holomorphic flat bundle. Then there exists a canonical differentially generated $\D$-ideal $\mathcal{I}_{\nabla}$ such that if $\mathcal{I}_{\nabla}$ is Spencer-polystable, then there exists a Hermitian-Yang-Mills metric on $E$. The conserve is also true, if $E$ admits a HYM metric, then $\mathcal{I}_{\nabla}$ is Spencer-polystable.}
\\

We also propose a conjectural extension connecting the Monge–Ampère equation on Fano manifolds to $K$-stability: if the associated PDE ideal is Spencer-semistable, we anticipate the existence of a Kähler–Einstein metric. This would suggest a PDE-theoretic refinement of the Yau–Tian–Donaldson program, linking analytic and algebro-geometric notions of stability within a broader moduli-theoretic perspective.
This problem requires extra care concerning questions of integrability and involutivity, thus the conjectural link of such a Spencer-semistable formulation to $K$-stability remains open for now.
\\

\noindent\textbf{Acknowledgments.}
J.K is supported by the Postdoctoral International Exchange Program of
Beijing Municipal Human Resources and Social Security Bureau. He would like to thank V.Rubtsov for helpful discussions. A.S. is supported by grants from Beijing Institute of Mathematical Sciences and Applications (BIMSA), the Beijing NSF BJNSF-IS24005, and the China National Science Foundation (NSFC) NSFC-RFIS program W2432008. He would like to thank China's National Program of Overseas High Level Talent for generous support.

\subsection{Conventions and notations}
\label{Notations and Conventions}

Let $k$ be a field of characteristic zero and consider a proper $k$-scheme or smooth algebraic variety $X$ with sheaf of differential operators $\mathcal{D}_X.$ The category of quasi-coherent $\mathcal{O}_X$-modules is denoted $\mathrm{QCoh}(X).$
The abelian category of $\mathcal{O}_X$-quasi-coherent left (resp. right) $\mathcal{D}_X$-modules is $\mathrm{Mod}_{qc}(\D_X)$ 
(resp. $\mathrm{Mod}_{qc}(\D_X^{op})$).

A smooth algebraic variety $X$ is \emph{D-affine} if: (i) the functor
\begin{equation}
\label{eqn: Global sect}
\Gamma(X,-):\mathrm{Mod}_{qc}(\D_X)\rightarrow \mathrm{Mod}\big(\Gamma(X,\D_X)\big),
\end{equation}
is exact, and (ii) if $\Gamma(X,M)=0$ for $M\in \mathrm{Mod}_{qc}(\D_X)$, then $M=0.$

A smooth affine algebraic variety is $D$-affine. Some non-affine varieties are $D$-affine for example, projective spaces (see e.g. \cite[Theorem 1.6.5]{HT}).

\begin{rmk}
    There is an obvious notion of being $D^{op}$-affine (though we do not use it). If $X$ is $D$-affine, it is not necessarily $D^{op}$-affine. For example, since $\Gamma(\mathbb{P}^1,K_{\mathbb{P}})=0,$ it follows $\mathbb{P}^1$ is not $D^{op}$-affine. 
\end{rmk}
Consider $M\in \mathrm{Mod}_{qc}(\D_X)$ and let $N\subset M$ be the submodule generated by global sections.
Since $X$ is $D$-affine, there is a short-exact sequence
$$0\rightarrow \Gamma(X,N)\rightarrow \Gamma(X,M)\rightarrow \Gamma(X,M/N)\rightarrow 0.$$
The first map is an isomorphism (by definition of $N$) so $\Gamma(X,M/N)=0.$ By $D$-affinity (ii), this implies $M/N=0$ so $M=N.$ In other words, if $X$ is $D$-affine, any $M\in \mathrm{Mod}_{qc}(\D_X)$ is $\D$-generated by its global sections.
The following is clear.
\begin{prop}
Suppose $X$ is $D$-affine. Then (\ref{eqn: Global sect}) is an equivalence of categories.
\end{prop}
\begin{proof}
    See \cite[Proposition 1.4.4 (ii)]{HT}.
\end{proof}

Non-linear PDEs modulo their symmetries are treated in a coordinate free way using $\D$-modules and $\D$-algebras, following closely \cite{BD,KSY,KSY2,P,P2} and references therein. The category of commutative unital left $\D$-algebras, is denoted by $\mathrm{CAlg}_X(\mathcal{D}_X)$. Given a $\D$-algebra $\A$, its category of left $\A$-modules in $\D_X$-modules is denoted $\mathrm{Mod}(\A[\D])$ where $\A[\D]:=\A\otimes_{\mathcal{O}_X}\D_X.$

We fix some numerical data: $(n,m,N,k),$ with $n$ the number of independent variables (e.g. $dim_X$), $m$ the number of dependent variables (e.g. $rank(E)$), with $N$ a natural number (the number of equations), and with $k:=(k_1,\ldots,k_N)\in \mathbb{N}^N$ the orders of the operators. This will correspond to a system,
\begin{equation}
    \label{eqn: System}
F_1\big(x,u,u_{\sigma}^{\alpha}\big)=0,\ldots,F_A\big(x,u,u_{\sigma}^{\alpha}\big)=0,\ldots,F_N(x,u,u_{\sigma}^{\alpha})=0,
    \end{equation}
with $m$ dependent variables $u^{\alpha}=u^{\alpha}(x_1,\ldots,x_n)$ and $n$ independent variables $x=(x_1,\ldots,x_n).$
Let $f:X\rightarrow S$ be a projective morphism of smooth algebraic varieties of relative dimension $d.$
Within the purview of $\D$-module theory, we study families of PDE systems algebraically via ideal sheaf sequences,
\begin{equation}
    \label{eqn: SES}
    \mathcal{I}_{X/S}\rightarrow \mathcal{A}_{X/S}\rightarrow\mathcal{B}_{X/S},
\end{equation}
consisting of a commutative unital left $\mathcal{D}$-algebra $\mathcal{A}:=\mathcal{O}(J_{X/S}^{\infty}E)$ of (relative) jets of dependent functions, a $\D$-ideal $\mathcal{I}$ of relations, and a quotient $\D$-algebra $\mathcal{B}$. One recovers the absolute case when $S$ is a point.
\begin{defn}
\normalfont
 Consider the system (\ref{eqn: System}). Call $N$ the \emph{$\mathcal{D}_X$-codimension} of $\mathcal{B}_X:=\mathcal{A}_X^{\ell}/\mathcal{I}_X$ (sometimes called \emph{formal codimension}), denoted by $\mathrm{codim}_{\mathcal{D}}(\mathcal{I}_X)=N.$ The \emph{order} of system is 
    $\mathrm{Ord}_{\mathcal{D}}(\mathcal{I}_X):=max\big\{\mathrm{ord}(F_i)=k_i:i=1,\ldots,N\big\},$
    and a system is said to be of \emph{pure order $k$} if $\mathrm{ord}(F_i)=k$ for all $i=1,\ldots,N.$ 
    \end{defn}
    We always assume pure order\footnote{Otherwise, we are lead to additional unnecessary complexities. Namely, the corresponding symbol maps of our operators will not be homogeneous of degree zero and we should impose certain weights e.g. $k_{i}^{-1}$ to render it so. This adds a layer of technicality (e.g. additional multi-gradings in cohomology appear) that is unimportant for our purposes so we omit it.}.

\begin{rmk}
    Note the use of $\D_X$-codimension \emph{not} the algebraic codimension. If one restricts to finite jets, the latter does not have an invariant meaning as both the number of equations and number of variables defining our system changes at each prolongation level. Furthermore, a generic $\D$-ideal will be of infinite algebraic codimension for which additional complexities arise.
\end{rmk}
The differentially stable ideals we consider are obtained from finite-order systems via prolongation. The following notations and terminology are to be fixed throughout the article:

A $k$-th order PDE is a closed sub-scheme of the scheme of $k$-jets of sections of a bundle $E$ over $X,$ denoted $Z_k\subset J_X^kE.$ For $0\leq \ell\leq k$, denote by $Z_{\ell}$ the projection of $Z_k$ in $J_X^{\ell}E.$ Let $\mathcal{I}_k\subset\mathcal{O}(J_X^kE)$ be the ideal of $Z_k$ i.e. $\mathcal{O}_{Z_k}=\mathcal{O}_{J_X^kE}/\mathcal{I}_k.$ Then $Z_{\ell}$ is the closed sub-scheme in $J_X^{\ell}E$ defined the ideal $\mathcal{I}_k\cap \mathcal{O}_{J^{\ell}_XE}.$
Locally, let $\mathrm{pr}_1\mathcal{I}_k\subset \mathcal{O}_{J_X^{k+1}E}$ be the ideal generated by $\mathcal{I}_k$ and $D_i\mathcal{I}_k,$ with $D_i$ the total derivatives.
We denote by $\mathrm{Pr}_1Z_k$ the corresponding closed sub-scheme of $J_X^{k+1}E$. Continuing this way, for all $\ell\geq 0,$ there exists $\mathrm{Pr}_{\ell}Z_k\subset Z_{k+\ell}.$
For $f\in \mathcal{O}_{J_X^kE},$ put $d^vf\in \oplus_j\mathcal{O}_{J_X^kE}[\xi_1,\ldots,\xi_n](k)d^vu^{\alpha},$ where $[-](k)$ indicates homogeneous degree $k$ elements, and $\alpha=1,\ldots,m.$ 

\begin{defn}
\label{geom saturation}
\normalfont 
Consider a closed subscheme $Z_q\subset J_X^{q}E$. Let $Z_{q-1}$ denote the projection of $Z_q$ in $J_X^{q-1}E$. It is \emph{saturated} if $Z_q\subset \mathrm{pr}_1(Z_{q-1}).$ 
\end{defn}
If $Z^q$ is not saturated, there is a canonical way to obtain from it a saturated closed subscheme.
\begin{prop}
Suppose that $Z^q$ is not saturated. Then there exists a subscheme $(Z^q)_{Sat}$ which the largest saturated sub-scheme of $J_X^q(E)$ containing $Z^q.$
\end{prop}
\begin{proof}
Take the projection $Z_{q-1}$ of $Z_q$ in $J_X^{q-1}E$. Then, form the fiber product $Z_q\times_{J_X^qE}\mathrm{pr}_1Z_{q-1}.$ Iterating this if necessary, one obtains a saturated closed subscheme.
\end{proof}
The $\D$-scheme obtained from $Z_k$, denoted by $Z=\{Z_k\subset J_X^kE\},$ (or $Z^{\infty}$), is defined as follows. Let $\mathcal{I}$ be the union of $\mathrm{pr}_{\ell}(\mathcal{I}_k),$ defined via interated prolongation, $\mathrm{pr}_{\ell}:=\mathrm{pr}_1\circ \mathrm{pr}_{\ell-1},\ell\geq 1,$ where $\mathcal{O}(Z_k):=\mathcal{O}(J_X^kE)/\mathcal{I}_k.$ Then, $\mathcal{I}$ is a sub-sheaf of the inductive limit $\mathcal{O}(J_X^{\infty}E):=\varinjlim \mathcal{O}(J_X^kE).$ The standard filtration (c.f. Definition \ref{defn: Standard filt}), is $F^{\ell}\mathcal{I}:=\mathcal{I}\cap \mathcal{O}(J_X^{\ell}E),\ell\geq 0.$ This defines a closed subscheme $F^{\ell}Z^{\infty}$ of $J_X^{\ell}E.$


Introduce the (naive) parameterizing space of isomorphism classes of exact sequences,
\begin{equation}
    \label{eqn: MD(m,n,K)}
\MD(m,N,k):=\big\{\text{Sequences } (\ref{eqn: SES})\big\}/\sim \simeq \bigg\{\adjustbox{scale=.85}{\begin{tikzcd}
    \mathcal{I}_X\arrow[d,"\alpha"] \arrow[r] & \mathcal{A}_X\arrow[d] \arrow[r, two heads, "q"] & \mathcal{B}_X\arrow[d,"\beta"]
    \\
    \mathcal{I}_X'\arrow[r] & \mathcal{A}_X\arrow[r, two heads, "q'"] & \mathcal{B}_X',
\end{tikzcd}}\bigg\},
\end{equation}
where $\alpha,\beta$ are isomorphisms of $\D_X$-algebras (algebro-geometric Lie-Bäcklund equivalence).

We also impose that the system $\{F_A=0:A=1,\ldots,N\}$, or equivalently the ideal $\mathcal{I}_X$ is non-trivial and differentially consistent i.e. $\mathrm{Sol}_{\D}(\I)\neq \mathrm{Spec}_{\D}(\A_X)$, and $\mathrm{Sol}_{\D}(\I)\neq \emptyset,$ respectively. Here $\mathrm{Sol}_{\D}(\I)$ is the $\D$-space of solutions, defined using the functor of points, recalled in Proposition \ref{prop: Sol}. See also \cite[Sect 2]{KSY}.


Restrict our attention to the corresponding sub-space of (\ref{eqn: MD(m,n,K)}) given by
\begin{equation}
    \label{eqn: D-Inv Moduli}
\MD^{inv}(m,N,k):=\big\{[\mathcal{I}_X]\in \MD(m,N,k)| \text{ involutive, formally integrable }\big\}.
\end{equation}
The goal of this paper is thus to prove (\ref{eqn: D-Inv Moduli}) defines a moduli functor.

\begin{rmk}
    We work in the almost algebraic category. Here, the Cartan-Kähler
theorem states that for involutive systems there exists a unique analytic
solution for a certain initial value problem prescribing finitely many analytic functions
of independent variables as Cauchy data. As any analytic system can be completed to
an involutive one by the Cartan-Kuranishi theorem \cite{Ku},
it poses no restriction to only look at 
$\MD^{\mathrm{inv}}(n,m)$ inside $\MD(n,m)$ and deal solely with involutive systems.
\end{rmk}
Throughout, $C_k^n$ denotes the binomial coefficient and $T_k^n$ will denote the number of Taylor coefficients of order $k$ in a function of $n$-variables i.e. $(n+k-1)!/k!(n-1)!.$ 
Denote the projectivized characteristic variety by $\mathbb{P}\mathrm{Char}(\mathcal{E}),$ and denote by $\mathbb{P}^*X$ the projectivized cotangent bundle.
\section{Ideal sheaves in $\D_X$-geometry}
\label{sec: Ideal Sheaves in D-Geometry}
The notion of a $\D$-module generalizes to the non-linear setting and to families of schemes or stacks.
In this section we collect the necessary background material for the scheme-theoretic approach to $\D$-geometry. We follow \cite{KSY}.

\subsection{Crystals and local systems of schemes}
Fix a scheme $S$. The category of $D_S$-schemes, $Sch_{S}(\D)$ may be defined as the category of local-systems of schemes over $S$ i.e. crystals of schemes over $S_{DR},$ the de Rham space of $S$ \cite{Si,Si2}.
Let $(S\times S)^{\wedge}$ and $(S\times S\times S)^{\wedge},$ denote formal neighbourhoods.
\begin{defn}
\normalfont 
A \emph{crystal of schemes on $S$} is a scheme $Z\rightarrow S$ together with an isomorphism $\psi:(Z\times S)|_{(S\times S)^{\wedge}}\simeq (S\times Z)|_{(S\times S)^{\wedge}},$ satisfying the cocycle condition: $p_{23}^*(\psi)p_{12}^*(\psi)=p_{13}^*(\psi)$, for the resulting isomorphisms between restrictions of $Z\times S\times S$ and $S\times S\times Z$ over $(S\times S\times S)^{\wedge}.$
\end{defn}
This notion generalizes in an obvious way to the relative situation and we
can talk about about crystals on $X/S$. If in addition $Z\rightarrow S$ is a vector bundle and $\psi$ is a morphism of vector bundles, we will call $Z$ a crystal of vector bundles on $S$. By reducing to the case when $X$ is affine over $S$ a crystal of vector
bundles on $X/S$ is the same as a vector bundle on $S$ with a relative integrable connection
over $S$ (see e.g. \cite[Lemma 8.1]{Si}).

Forgetting the $\D_S$-action has a left-adjoint functor given by taking jets, $J_S^{\infty}$, which both commute with Zariski and \'etale localization, giving the adjunction:
$$Sch_{S}(\D)\rightleftarrows (Sch_k/S).$$
More generally, for any $X/S$-scheme $E$, there is a $\D_{X/S}$-scheme $J_{X/S}^{\infty}(E).$ The canonical $\D$-scheme structure morphism is denoted $p_{\infty}:J_{X/S}^{\infty}(E)\rightarrow X/S.$
The coherent cohomology $H^q(X/S,E)$ of a coherent sheaf $E$ on $X$ is known to be isomorphic to the relative de Rham cohomology of relative jet bundles $J_{X/S}^{\infty}(E)$. 
\begin{defn}
\label{defn: Algebraic solution}
\normalfont
Let $Z\rightarrow S$ be a crystal of schemes. An \emph{algebraic flat/horizontal section} $a_{\nabla}:S\rightarrow Z$ is an algebraic section $a:S\rightarrow Z$ such that the subscheme $a(S)\subset Z$ is a sub-crystal (sub-local system).
\end{defn}
The sheaf of flat sections identifies with sections of $E$ (Proposition \ref{prop: Sol} below) and de Rham cohomology (hypercohomology) $H_{DR}^*(X/S,J_{X/S}^{\infty}E)$ is canonically isomorphic to $H^*(X/S,E).$
For our applications to PDEs, $Z$ can taken as the universal $\D_S$-scheme of algebraic infinite jets $Z=J_{S}^{\infty}(E),$ on sections of some holomorphic bundle $p:E\rightarrow S$, with canonical structure map $\pi_{\infty}:Z\rightarrow S.$

\begin{rmk} 
\label{rmk: Universal pair}
We view the pair $(J_{S}^{\infty}(E)/S,\mathcal{C}),$ as a universal receptacle for sub-$\D$-schemes, where $\mathcal{C}$ is the canonical flat connection given by the Cartan distribution \cite{KLV}.
\end{rmk}

Thus given a local system of schemes $(Z/S,\nabla)$, a sub local-system is a variety $\EQ\hookrightarrow Z$ smooth over $S$ such that $\nabla$ lifts infinitesimal symmetries of $S$ to infinitesimal symmetries of $Z$, which at all geometric points $y\in \EQ$, preserve $\EQ.$ This means, the composition 
$$(p_{\infty}^*T_S)|_{\EQ}\xrightarrow{\nabla}(T_Z)|_{\EQ}\rightarrow N_{\EQ/Z},$$
is identically zero. Thus, an algebraic section $a:S\rightarrow Z$ of $p_{\infty}:Z\rightarrow S$ is flat if and only if 
\[
\begin{tikzcd}
    a^*p_{\infty}^*T_S\arrow[d,"="]\arrow[r,"a^*\nabla"]& a^*T_Z\\
    T_S\arrow[ur,"da"]
    & 
\end{tikzcd}
\]
commutes.

In this work we consider sub-local systems of varieties in the fixed ambient jet-scheme $J_X^{\infty}E$, as suggested by Remark \ref{rmk: Universal pair}.
Given such a $\D$-scheme $Z\subset J_X^{\infty}E,$ one defines a functor (pre-stack) of solutions $\underline{\mathrm{Sol}}_{\D}(Z),$ constructed (in greater generality) in \cite[Construction 6.13]{KSY}. We give one example.

\begin{ex}Consider a differential polynomial algebra in $m$-variables, $\mathbb{C}[(u_{\sigma}^{\alpha}):|\sigma|\in \mathbb{N},\alpha=1,\ldots,m]$. The action by derivations is $\partial_i\bullet u_{\sigma}^{\alpha}:=u_{\sigma+1_i}^{\alpha}$ for $\alpha=1,\ldots,m.$ Then for $P_1,\ldots,P_N\in \mathbb{C}[u_{\sigma}^{\alpha}]$ and every $m$-tuple of elements $(r):=(r_1,\ldots,r_m)\in \mathcal{R}_X^{m}$ in a given test $\D$-algebra $\mathcal{R}$, there exists an evaluation map 
$ev_{(r)}:\mathbb{C}[u_{\sigma}^{\alpha}]\rightarrow \mathcal{R},$ with $ev_{(r)}(u^i)=r_i,$ for each $i$. The $\mathcal{R}$-valued solution functor is given by
$$\mathrm{Sol}_{\mathcal{D}}(P_i)(\mathcal{R})=\{(r)=(r_1,\ldots,r_m)\in \mathcal{R}^m|P_i(r)=ev_{(r)}(P_i)=0,i=1,\ldots,N\}.$$
\end{ex}

The following well-known, but is recalled for later use.
\begin{prop}
 \label{prop: Sol}
Let $Z\subset J_X^{\infty}E$ be a sub-crystal of varieties. Then, there is closed-embedding of functors $\underline{\mathrm{Sol}}_{\D}(Z)\hookrightarrow \underline{\mathrm{Sol}}_{\D}(J_X^{\infty}E),$ and if $\underline{\Gamma}(X,E)(Spec(R))=Hom_{Sch}(X\times Spec(R),E)$ denotes the sheaf of parameterized sections of $E$, there is an isomorphism $\underline{\Gamma}(X,E)\simeq \underline{\mathrm{Sol}}_{\D}(J_X^{\infty}E),$ induced by infinite-jet prolongation.
\end{prop}
\begin{proof}
See e.g. \cite[Proposition 6.11,6-12 and 6.15]{KSY}.
\end{proof}

Tangent spaces $T_Z$ to $\D$-schemes $Z$ play the role of universal linearization (see \cite{KLV}). In particular, for each (algebraic) solution $a:S\rightarrow Z$, the pull-back $a^*T_Z$ is a $\D$-module on $S$. 

Furthermore, let $q:S\rightarrow S_{DR}$ be the universal projection. Then, there exists an $(q^*,q_*)$-adjunction and a (non-linear) differential operator between (sections) of coherent sheaves of $\mathcal{O}_X$-modules
$F_P:E\rightarrow E'$ is nothing but a morphism of $\D$-schemes between jet constructions, $J_S^{\infty}(E)\rightarrow J_{S}^{\infty}(E').$ 

Denoting the structure maps by $p_{\infty},p_{\infty}'$, then $F_P$ is morphism of $S_{DR}$-schemes $F:q_*E \rightarrow q_*E'.$ The Weil-restrictions $q_*E$ are universal $S_{DR}$-spaces, denoted $J_{S_{DR}}^{\infty}(E).$ They are to be thought of as the $S$-scheme $J_{S}^{\infty}E$ together with its canonical flat connection (Cartan). We often identify the two. 

By adjunction, there is an $S$-scheme morphism $J_S^{\infty}(E):=q^*q_*E\rightarrow E.$
The tangent map $TF$ is therefore the $\mathcal{O}_{q_*E}\otimes_{\mathcal{O}_X}\mathcal{D}_X$-module morphism of linearization
\begin{equation}
    \label{eqn: Universal linearization}
\ell_{F}:=TF:\Theta_{q_*E}\rightarrow F^*\Theta_{q_*E'}.
\end{equation}
Its sheaf of zero-th de Rham cohomology $h(TF)$ coincides with the morphism of relative tangent sheaves
$$p_{\infty}^*T_{E/X}\rightarrow F^*(p_{\infty}')^*T_{E'/X}\simeq p_{\infty}^*T_{E'/X}.$$
Let $\EQ$ denote the $\D$-scheme corresponding to $F_P$ and put $\mathcal{M}_1:=\Gamma(p_{\infty}^*E')|_{\mathcal{Y}}.$ If $\EQ$ is formally integrable (resp. involutive) (\ref{eqn: Universal linearization}) gives a map $\ell_F:\Theta_{q_*E}|_{\EQ}\rightarrow \mathcal{M}_1,$
which is a formally integrable (resp. involutive) \emph{linear} horizontal PDE (a PDE in total derivatives). Cohomology of the $\D$-module de Rham complex with coefficients in linearized solution sheaf $\mathrm{ker}(\ell_{F})$ is isomorphic to the cohomology of a compatibility complex \cite{G,KV}.

We now briefly recall a description via Grothendieck-style algebraic geometry which allows us to formulate scheme theory for PDEs in terms of differential ideals.

\subsubsection{$\D$-schemes} The global theory of algebraic $\D$-geometry is centered around the notion of an affine non-linear PDE in terms of sheaves of $\D_X$-ideals. 

\begin{defn}
\label{defn: RelAlgNLPDE}
\normalfont An (affine) $\D$-\emph{algebraic PDE} is a sequence of commutative $\mathcal{D}_X$-algebras (\ref{eqn: SES}). When $\mathcal{A}$ is of the form $\mathcal{O}(J_X^{\infty}E)$ for some bundle (or sheaf) $E$ over $X$, we say the $\mathcal{D}_X$-PDE is imposed on sections $\Gamma(X,E).$ The category of affine $\D$-schemes is denoted $\mathrm{Aff}_X(\D_X).$ 
\end{defn}
More precisely, $\mathrm{Aff}_X(\D_X)$ is the opposite to $\mathrm{CAlg}_X(\D_X)$ and this is reflected by writing $Z\in \mathrm{Aff}_X(\D_X)$ simply by $Z=\mathrm{Spec}_{\D}(\mathcal{B}),$ for a commutative $\D$-algebra $\mathcal{B}.$

The category of $\mathcal{D}_X$-schemes is denoted by $\mathrm{Sch}_X(\D_X)$, and is obtained by gluing. It was introduced in \cite[Sect. 4.2]{KSY}, as the $X$-local Zariski sheaves of sets on the category $\mathrm{CAlg}_{X}(\mathcal{D}_X)^{op}$ with the induced pre-topology $\tau_{Zar}$, denoted 
$$\mathrm{Sch}_{X}(\mathcal{D}_X)\subset \mathrm{Spaces}_{\mathcal{D}_X}^{Zar}:=\mathrm{Shv}(\mathrm{CAlg}_X(\mathcal{D}_X)^{op},\tau_{Zar}).$$
Specifically, via a well-defined notion of \emph{$\D$-open immersion}, which are stable by base-change \cite[Proposition 4.17-4.18]{KSY}, one may glue affine $\D$-schemes along families of $\D$-open immersions.
\begin{defn}
\normalfont
An algebraic $\D$-space $\mathcal{X}:\mathrm{CAlg}_X(\mathcal{D}_X)^{op}\rightarrow \mathrm{Sets}$ is a \emph{$\D_X$-scheme} if there exists a finite family $(\mathcal{A}_i)_{i\in I}$ of commutative $\mathcal{D}_X$-algebras such that 
$\coprod_{i\in I}h_{\mathcal{A}_i}^{\mathcal{D}}\rightarrow \mathcal{X},$
is a $\D$-epimorphism, where $h_{\mathcal{A}}^{\mathcal{D}}$ is the functor sending a $\D$-algebra $\mathcal{R}$ say,
the set $\mathrm{Hom}_{\mathrm{CAlg}(\mathcal{D})}(\mathcal{A},\mathcal{R})$ (also denoted $\underline{\mathrm{Spec}}_{\mathcal{D}}(\mathcal{A})$).
\end{defn}
A $\D$-scheme is thus a functor ($\D$-space) together with a collection
$\{h_{\mathcal{A}_i}^{\mathcal{D}}\rightarrow \mathcal{X}\}_{i\in I}$, which is a $\D$-Zariski covering family. 

They may also be understood as a type of differentially structured ringed space $(Z,\mathcal{O}_{Z})$ together with an open covering $\cup_i V_i$ with each $(V_i,\mathcal{O}_{Z}|_{V_i})$ isomorphic to an affine $\D$-scheme. A morphism $F:(Z_1,\mathcal{O}_{Z_1})\rightarrow (Z_2,\mathcal{O}_{Z_2})$ is given by a morphism of underlying topological spaces $|F|:|Z_1|\rightarrow |Z_2|$ and a morphism of sheaves of $\D$-algebras $F^{\sharp}:F^*\mathcal{O}_{Z_2}\rightarrow \mathcal{O}_{Z_1}.$

Closed $\D$-Zariski subsets can be described using $\D$-ideals. In what follows we often study $\D$-schemes via algebraic techniques based on properties characterizing prime ideal sheaves.
\begin{defn}
\label{defn: Closed D-subscheme}
\normalfont 
A \emph{closed} $\D$-subscheme is a morphism $i:(Z_0,\mathcal{O}_{Z_0})\rightarrow (Z,\mathcal{O}_{Z})$ such that $|i|:|Z_0|\subseteq |Z|$ is an inclusion of topological spaces whose image is a closed subset with $i^{\sharp}:i^*\mathcal{O}_{Z}\rightarrow \mathcal{O}_{Z_0}$ a surjective morphism of sheaves of $\D$-algebras (on $X$).
\end{defn}
Any $\D$-algebraic PDE, as in the sequences (\ref{eqn: SES}) define a closed $\D$-subscheme, denoted $\mathrm{Spec}_{\D}(\mathcal{B})\hookrightarrow \mathrm{Spec}_{\D}(\mathcal{A}).$
Open $\D$-subschemes are defined similarly. Moreover, one may characterize $\D$-algebraic varieties \cite[4.22]{KSY}; 
a $\D$-variety is a reduced $\D$-scheme of the form $(\mathcal{V},\mathcal{O}_{\mathcal{V}})$ where $\mathcal{V}$ is representable by some reduced $\D$-algebra $\mathcal{A}$ of finite type that is a quotient by a radical $\D$-ideal $\mathcal{I}.$

Consider an (analytic) ringed space $\pi:E\rightarrow X.$ If $X_0\hookrightarrow X$ is an open sub-variety and $E_0\hookrightarrow E$ is such that $\pi(|E_0|)\subset |X_0|,$ given an $X$-local $\D$-space $Z$ imposed on sections of $E$ (written $Z_X$ when precision is needed), by evident restrictions, there is an $X_0$-local $\D$-space $Z_{X_0}$ on sections of $E_0.$ 

Affine $\D$-space is $\mathbb{A}_{\mathcal{D}}^n:=\mathrm{Spec}(\mathrm{Sym}^*(\mathcal{D}_X^n)),$ for $n\geq 0.$ Given an étale map of $X$-schemes $\psi:E\rightarrow \mathbb{A}_X^n,$ we have an induced $\D$-scheme map $J^{\infty}(\psi):J^{\infty}(E)\rightarrow J^{\infty}(\mathbb{A}^n)\simeq \mathbb{A}_{\mathcal{D}}^n.$

We always consider closed $\D$-sub-schemes in an ambient $\D$-smooth space. In particular, algebraic jet $\D$-schemes are $\D$-smooth.
\begin{defn}
\label{defn: D-Finite-type}
\normalfont 
An affine $\D$-scheme $Z$ is said to be \emph{$\D$-finite type} if it is of the form $\mathrm{Spec}(\mathrm{Sym}(\mathcal{D}_X\otimes\mathcal{F})/\mathcal{I})$ for some coherent sheaf $\mathcal{F}$ of $\mathcal{O}_X$-algebras and some $\D$-ideal sheaf $\mathcal{I}.$  A $\mathcal{D}$-algebra is \emph{$\mathcal{D}$-smooth} if $\Omega_{\mathcal{A}}^1$ is a projective $\mathcal{A}$-module of finite $\mathcal{A}[\mathcal{D}]$-presentation and there exists a finite-type $\mathcal{O}_X$-module $\mathcal{M}$ and an ideal $\mathcal{K}\subset \mathrm{Sym}_{\mathcal{O}_X}(\mathcal{M})$ such that there exists a surjection $J^{\infty}\big(\mathrm{Sym}_{\mathcal{O}_X}(\mathcal{M})/\mathcal{K}\big)\rightarrow \mathcal{A}$ whose kernel is $\mathcal{D}_X$-finitely generated.
\end{defn}
The $\mathcal{A}[\mathcal{D}]$-module of Kähler differentials $\Omega_{\mathcal{A}}^1$ is defined by the universal property 
$$\mathcal{D}er_{\mathcal{D}}(\mathcal{A},\mathcal{M})\simeq\mathcal{H}om_{\mathcal{A}[\mathcal{D}]}(\Omega_{\mathcal{A}}^1,\mathcal{M}\otimes \mathcal{A}[\mathcal{D}]),$$
and we put $\Theta_{\mathcal{A}}:=\mathcal{D}er_{\mathcal{D}}(\mathcal{A},\mathcal{A}).$ The main example is a $\D$-algebra $\mathcal{A}=\mathrm{Sym}_{\mathcal{O}_X}(\mathcal{M})$ for a finite-generated (e.g. coherent) $\D$-module of rank $r$. Clearly, $\Omega_{\mathcal{A}}^1\simeq\mathcal{A}\otimes\mathcal{M},$ while $\Omega_{\mathcal{A}}^r\simeq \mathrm{Sym}(\mathcal{M})\otimes\wedge^r\mathcal{M}.$

\begin{ex}
    \normalfont 
Consider $X=\mathrm{Spec}(k[x]/I)\subset \mathbb{A}_k^n$ as a usual affine $k$-scheme. The jet space of $X$ is affine and given by $J^k(X)\subset J^k(\mathbb{A}^n)\simeq\mathrm{Spec}(k[x,u_{\sigma}^{\alpha}])$. It is cut out by the system $P(x)=0$ and $\sum_{\sigma} (\partial^{\sigma}P)(x)u^{\alpha}=0,$ for each $P\in I.$
\end{ex}

The notion of a $\D$-group scheme $\mathcal{G}$ was introduced in \cite{BD}. We will work with a related notion of \emph{algebraic pseudogroup}, following S. Lie \cite{Li,Li2}. We recall their definition in Subsect. \ref{ssec: Differential algebra quotients}.

Finally, note the important example of a representable $\D$-space given by the \emph{general linear} $\D_X$-group of order $n$, denoted $\mathrm{GL}_{\mathcal{D}}^n$. It is a representable functor by a $\D$-algebra $\mathbb{C}[u_{\sigma}^{\alpha}]$ of differential polynomials localized with respect to the multiplicative system of invertible elements denoted $\{\mathrm{det}(u_{\sigma}^{\alpha})\}.$


\subsection{Properties of $\D_X$-ideal sheaves}
\label{ssec: Properties of D-Ideal Sheaves}
In this subsection we follow \cite{Ri}.
A prime (resp. radical) $\mathcal{D}_X$-ideal is an ordinary ideal in a $\mathcal{D}_X$-algebra $\mathcal{A}$ which is prime (resp. radical) as usual but is further stable under the action by differential operators.
Given a $\mathcal{D}_X$-ideal $\mathcal{I},$ then its radical $\sqrt{\mathcal{I}}$ is again a $\mathcal{D}_X$-ideal and every radical $\mathcal{D}_X$ ideal in a polynomial $\mathcal{D}_X$-algebra is an intersection of prime $\mathcal{D}_X$-ideals. 

A polynomial $\mathcal{D}_X$-algebra is a $\mathcal{D}_X$-algebra whose elements are differential polynomials in some collection of variables e.g. $\mathcal{A}=\mathbb{C}[x_i,u_{\sigma}^{\alpha}],$ or $\mathrm{Sym}(\mathcal{D}_X^r)$ for $r\geq 0.$ Given $\mathcal{I}_k\subset \mathcal{A}_k$ and denote its prolongation of order $\ell$ by $Pr^{\ell}\mathcal{I}_k\subset \mathcal{A}_{k+\ell}.$ 
An ideal $\mathcal{J}\subset \mathcal{A}_k$ is sometimes said to be \emph{saturated} if for all $\ell\geq 1,$ one has $Pr^{\ell}\mathcal{J}\cap \mathcal{A}_k=\mathcal{J}$ c.f. (\ref{geom saturation}).

\begin{rmk}
We do not actively make use of the so-called Kolchin topology that arises in this setting, but stress that care must be taken in some places by replacing Noetherian $k$-algebras with appropriate analogs, the so-called Ritt-Noetherian differential algebras \cite{Ri}.
\end{rmk}

\begin{prop}
Let $\mathcal{I}$ be a reduced differential ideal. Then it is a finite intersection of primary differential ideals.
\end{prop}
\begin{proof} See \cite[Cor. 4.4]{Ma}.
\end{proof}

\begin{prop}
Fix an affine $\D$-smooth $\D$-scheme $W$ and let $V$ be a separated $\D$-scheme over $W$. Let $\{V_i\}$ be an open covering by affine $\D$-schemes. Then for each $i,j$ the intersection $V_i\cap V_j$ is homeomorphic to an affine $\D$-scheme.
\end{prop}

We can identify $\D$-algebraic varieties over a sufficiently saturated differential ring $K$ with the set of its $K$-points.
 \begin{prop}
Consder two radical $\D$-ideals of a $\D$-smooth algebra $\mathcal{A}$, satisfying $\mathcal{I}\subseteq \mathcal{J}$. Then $\mathcal{I}=\mathcal{J}$ if and only if $V_{\mathcal{D}}(\mathcal{I})=V_{\mathcal{D}}(\mathcal{J}).$
\end{prop}

The central result used to obtain generic local models is the following and is essentially due to B. Malgrange, which roughly speaking, expresses local generic involutivity of reduced differential ideals.
\begin{prop}
\label{prop: Generic local analytic models}
Consider $U\subset \mathbb{C}^n,V\subset \mathbb{C}^n,$ with $U_0\subset U,V_0\subset V$ relatively compact. Then the following hold:
\begin{enumerate}
    \item There exists $k\geq 1$ and a section $F\in F^k\mathcal{A}(U_0\times V_0),$ such that for every $\D$-finitely presentation algebra $\mathcal{B}=\mathcal{A}/\mathcal{I},$ we have $F$ is injective on $\mathcal{B}$ and outside of $\{F=0\},$ we have $F^k\mathcal{I}$ is involutive with each $F^{\ell}\mathcal{I}$ for $\ell\geq k+1,$ its natural prolongations;

    \item Any increasing sequence of $\mathcal{D}$-coherent ideals of $\mathcal{A}$ is stationary over $U_0\times V_0\subset U\times V.$
\end{enumerate}
\end{prop}

Given some collection of differential polynomial $\{F_A\}$ in variables $u_1,\ldots,u_m$, the largest number of derivatives $k$ such that $u_{(k)}^{\alpha}$ appears is said to be the order with respect to $u^{\alpha}$. Denote it by $\mathrm{ord}(\{F\};u^{\alpha}),$ for each $\alpha=1,\ldots,m.$ Put
$\mathrm{ord}(\{F
_A\}):=\underset{\alpha=1,\ldots,m}{\mathrm{max}}\mathrm{ord}(\{F_A\},u^{\alpha}),$
and call it the \emph{order} of $\{F_A\}.$  

Computational (e.g symbolic) algebra in this setting is facilitated by introducing a notion of ranking of the differential variables, in the sense of
 Janet-Riquier theory. Roughly, it is an order on derivative variables $u_{\sigma}^{\alpha}$ compatible with differentiations.
 \begin{defn}
\label{term: Rankings} 
The \emph{total ordering} is the ranking defined by:
$u_{\sigma}^{\alpha}<u_{\tau}^{\alpha},\hspace{1mm}\text{ when } \hspace{1mm} |\sigma|<|\tau|,\hspace{1mm}\forall \alpha=1,\ldots,m,$
where $\sigma,\tau$ are multi-indices.
\end{defn}
 In other words, $p\in\{u_{\sigma}^{\alpha}|\sigma|\geq 0,\alpha=1,\ldots m\}$ satisfies: $p<D_i(p)$ for all total derivatives $D_i$ and the ordering $p<q$ implies $D_i(p)<D_i(q)$ for each $i=1,\ldots, \mathrm{dim}(X).$

 A differentially generated $\mathcal{D}_X$-ideal $\mathcal{I}=\{F_A\}$ 
automatically contains all differential consequences of its generators. It is possible to restrict to finite dimensional jets of a certain order, say $k$, but at the cost of that our ideal is now only algebraic, defined for all $k\geq 0$ by
\begin{equation}
    \label{eqn: FiniteIdeal}
F^k\mathcal{I}:=\mathcal{I}\cap F^k\mathcal{A}^{\ell}\simeq \mathrm{Span}\big\{D^{\sigma_i}(F_i)|\mathrm{ord}(F_i)+|\sigma_i|\leq k\big\}\subseteq F^k\mathcal{A}^{\ell}.
\end{equation}
Algebraic ideal (\ref{eqn: FiniteIdeal}) contains all
integrability conditions up to order $k$. We assume formal integrability, referring to \cite[Definition 4.12]{KSY} for an algebraic formulation, which is equivalent to the standard one \cite{G,G2}.

\begin{rmk}
A system $Z_k\subset J_X^kE$ is: (1) formally integrable to order $\leq 1$ if $Z_{k+1}\rightarrow Z_{k}$ is surjective, (2) formally integrable to order $\ell$ if its prolongations $Z_m$ are formally integrable to order $1$ for all $k\leq m\leq \ell+k-1$ and (3) formally integrable if $Z_m$ is formally integrable to order $1$ for all $m\geq k.$
\end{rmk}
For applications, special attention is given to $\D$-ideals which are differentially generated i.e. they are free as $\D$-modules whose defining relations are algebraically generated by elements $D_{\sigma}(F_i),|\sigma|\geq 0,i=1,\ldots,N.$

\subsection{Characteristic modules}
\label{sssec: Strict Chars Locally}
We give the analog of the characteristic variety, characteristic ideal etc. for a $\D$-module adapted to $\mathcal{D}$-algebras.

\begin{defn}
\label{defn: Standard filt}
\normalfont
    Given a formally integrable $\D$-subscheme $Z=\mathrm{Spec}_{\D}(\mathcal{B})\hookrightarrow J_X^{\infty}E$, with $\D$-ideal sheaf $\mathcal{I},$ the canonical filtration by jet order, denoted $F^k\mathcal{A}:=\mathcal{O}(J_X^kE)$, induces the \emph{standard filtration,} given (by $\mathcal{O}_X$-coherent sheaves), as $F^k\mathcal{I}:=F^k\mathcal{A}\cap \mathcal{I}$ with $F^k\mathcal{B}:=F^k\mathcal{A}/F^k\mathcal{I},$ for each $k.$
\end{defn}
Locally, if $f=f(x,u^{\alpha},u_{\sigma}^{\alpha})\in F^k\mathcal{A}$ is a germ of a section, the relative differential, when restricted to $k$-jets, gives
$$d^vf=\sum_{\alpha=1,|\sigma|\leq k-1}^{m}\frac{\partial f}{\partial u_{\sigma}^{\alpha}}d^vu_{\sigma}^{\alpha}.$$

Since the formation of Kähler differentials in $\mathcal{A}\otimes\mathcal{D}_X$-modules is $\mathcal{D}_X$-linear, this means that  
$d^vD_i(f)=\xi_id^v(f)$ and $\xi_i(gd^vu_{\sigma}^{\alpha})=gd^v u_{\sigma+1_j}^{\alpha}, i=1,\ldots,n.$


The $m$-th order symbol of some $F\in \mathcal{I}$ may be understood as the class of the vertical differential $dF\in \Omega_{J^{m}/J^{m-1}}^1\simeq \mathcal{O}(J^m)[\xi_1,\ldots,\xi_n](m).$ In particular, for $F\in \mathfrak{J}_k=F^k\mathcal{I},$ its symbol is just as above, modulo $F^k\mathcal{I}$ i.e. an element of $\mathcal{O}(Z_m)[\xi_1,\ldots,\xi_n](m),$ that is, the ring of homogeneous polynomials of order $m$ in variables $(\xi)$ with coefficients in $F^k\mathcal{B}.$

\subsubsection{Strict characteristic modules}
\label{sssec: Strict char modules}
Geometrically, the objects just described in Subsect. \ref{sssec: Strict Chars Locally} are given by considering the conormal sheaf of $Z_k$,
$$\mathcal{C}_{Z_k/J_X^k(E)}:=\mathfrak{J}_k/\mathfrak{J}_{k}^2=F^k\mathcal{I}/(F^k\mathcal{I})^2,$$ which is a coherent $\mathcal{O}(Z_k)$-submodule 
\begin{equation}
    \label{eqn: Conormal complex}
\mathcal{C}_{Z_k/J^k_XE}\subset \Omega_{J^k_XE}^1\otimes\mathcal{O}(Z_k).
\end{equation}
Moreover, there are natural subsheaves 
\begin{equation}
    \label{eqn: k-th symbol module}
\mathcal{N}_k\subset \mathcal{O}(Z_k)[\xi_1,\ldots,\xi_n](k),\hspace{2mm} k\geq 0,
\end{equation}
generated by the classes corresponding to $F\in \mathfrak{J}_k.$
Spaces (\ref{eqn: k-th symbol module}) canonically define  algebraic ideals $\{J_k\subset \mathcal{O}(Z_k)[\xi_1,\ldots,\xi_n](*)\}_{k\geq 0},$ whose corresponding quotient module 
\begin{equation}
    \label{eqn: k-characteristic module}
\mathcal{C}h_k:=\mathcal{O}(Z_k)[\xi_1,\ldots,\xi_n](*)/J_k,\hspace{2mm} k\geq 0,
\end{equation}
is called the $k$-\emph{characteristic module}.
The full characteristic module is the graded quotient module of the symbol module by ideal spanned $\{J_k\}.$ Explicitly, the \emph{symbol module} is the graded module $\mathcal{N}_{\bullet}:=\bigoplus_k\mathcal{N}_k$ corresponding to (\ref{eqn: k-th symbol module}). Locally,
$$\mathcal{N}_k\subset \bigoplus_{1\leq \alpha\leq m}\mathcal{O}(Z_k)[\xi]d^vu^{\alpha}\simeq \mathcal{O}(Z_k)[\xi](p),$$
generated by $d^vF^0\mathcal{I},\cdots, d^vF^k\mathcal{I}$ modulo the sub-module $F^k\mathcal{I}$. This is exactly $J_k$ and in what follows we will often confuse the two.

\begin{obs}
    Locally, one may take $F^k\mathcal{A}[\xi_1,\ldots,\xi_n]$ as the ideal generated by $F^k\mathcal{I}$ and all minors of the order $n$ i.e. 
$det(d^v F_1,\ldots, d^vF_n)$ with $F_i\in F^k\mathcal{I},0\leq r\leq k,$
and $d^vF_i$ the symbol of order $r.$
\end{obs}

The following is clear.

\begin{prop}
Let $\mathcal{A},\mathcal{B}=\mathcal{A}/\mathcal{I}$ be as above.
Then, if $f=f(x,u,u_{\sigma}^{\alpha})\in F^k\mathcal{A},$ put $\overline{f}\in F^k\mathcal{B}=F^k\mathcal{A}/F^k\mathcal{I}.$ The class of $d^v f$ in $\mathcal{C}h_k$ is independent of $\overline{f}$ and is given by  $d_{J^k_XE/J^{k-1}_XE}(\overline{f})\in \Omega_{J^k_XE/J^{k-1}_XE}^1.$ 
\end{prop}

Given a $\D$-PDE defined by $\mathcal{I},$ the relative differential is
$$d_{F^k\mathcal{B}/F^{k-1}\mathcal{B}}^{rel}(\overline{f})\in \Omega_{F^k\mathcal{B}/F^{k-1}\mathcal{B}}^1,$$
and set $\mathcal{C}h_k(\ell),$ to be the homogeneous degree $\ell$ elements of (\ref{eqn: k-characteristic module}). There are isomorphism of $\mathcal{O}_X$-modules,
\begin{equation}
\label{eqn: ChKahler1}
\mathcal{C}h_k(k)\simeq \Omega_{J^k_XE/J^{k-1}_XE}^1, \hspace{1mm} k\geq 1.
\end{equation}
The identification (\ref{eqn: ChKahler1}) will be discussed in Proposition \ref{prop: Omega and Ch isom} below. Put
\begin{equation}
    \label{eqn: Char Module of A}
\mathcal{C}h^{\mathcal{A}}(*)=\bigoplus_k\mathcal{C}h_k(*)=\bigoplus_{k}\mathcal{A}\otimes_{F^k\mathcal{A}}\Omega_{J^k_XE/J^{k-1}_XE}^1.
\end{equation}
The characteristic module of a $\D$-subscheme $Z\subseteq J_X^{\infty}E=\mathrm{Spec}_{\D}(\mathcal{A})$ is defined following (\ref{eqn: Char Module of A}).
\begin{defn}
\label{defn: Strict Char modules}
\normalfont 
Let $Z\hookrightarrow \mathrm{Spec}_{\mathcal{D}}(\mathcal{A})$ be a sub $\D$-scheme with commutative $\D$-algebra of functions $\mathcal{B}:=\mathcal{O}_{Z}.$ The \emph{characteristic module (of $Z$)} is given by  
$\mathcal{C}h^{Z}(*):=\bigoplus_k\mathcal{B}\otimes_{F^k\mathcal{B}}\Omega_{F^k\mathcal{B}/F^{k-1}\mathcal{B}}^1.$
\end{defn}
It is a graded module and we set  $\mathcal{C}h_k(\geq p)$ to be the sub-module in degrees at least $p$ and for all $p\leq q,$ set 
\begin{equation}
    \label{eqn: Truncation in degrees}
\mathcal{C}h_k([p,q]):=\mathcal{C}h_k(\geq p)/\mathcal{C}h_k(\geq q),
\end{equation}
to be the truncation in degrees $[p,q].$
\begin{prop}
\label{prop: Mult by microlocal coordinates}
Fix a coordinate system $(x_1,\ldots,x_n)$ on $X$, and its induced coordinate system $(x;\xi)\in T^*X.$ There are well-defined maps 
$\Omega_{J^{k}/J^{k-1}}^1\rightarrow \Omega_{J^{k+1}/J^k}^1$ as well as 
$\Omega_{Z_k/Z_{k-1}}^1\rightarrow \Omega_{Z_{k+1}/Z_k}^1,$ induced from the multiplication by $\xi_i$ for $i=1,\ldots,n$. 
\end{prop}
\begin{proof}
Via the isomorphism (\ref{eqn: ChKahler1}), the result follows by standard arguments \cite{Gro,Gro2}.
\end{proof}

To express the prolongation of symbols in terms of prolongations of ideals and equations, set 
$$\mathrm{Pr}_1(Z_k):=(Z_{k+1},\mathcal{I}'),\hspace{1mm}\text{ where } \mathcal{I}':=\mathcal{I}_{k+1}.$$
There is a map of topological spaces
$\pi_{k+1,k}:=|p^{k+1}_k|:|Z_{k+1}|\rightarrow |Z_k|,$ where $p^{k+1}_k$ is the restriction of the projection $J_X^{k+1}E\rightarrow J_X^kE$ to $Z_{k+1}$. Thus, there is a homomorphism of $\mathbb{C}$-algebras $\pi_{k+1,k}^*:\pi_{k+1,k}^{-1}\mathcal{O}_{Z_k}\rightarrow \mathcal{O}_{Z_{k+1}}.$
If $\mathcal{N}',\mathcal{M}'$ are the characteristic modules of the order $1$ prolongations $\mathrm{Pr}_1(Z_k)$, then
$$\pi_{k+1,k}^*(\pi_{k+1,k}^{-1}\mathcal{N}_{k+1})\subset \pi_{k+1,k}^*\pi_{k+1,k}^{-1}\mathcal{O}_{Z_k}[\xi]_{k+1}^p,$$
where $\mathcal{O}[\xi]_{k+1}^p$ is shorthand for homogeneous degree $p$ elements. Similarly for all $\mathcal{N}_{m}',m\geq k+1,$ so it follows that
$$\mathcal{M}'=\mathcal{O}_{Z_{k+1}}\otimes_{\pi_{k+1,k}^{-1}\mathcal{O}_{Z_k}}\pi_{k+1,k}^{-1}\mathcal{M}_{>0}.$$

\subsubsection{$\D$-geometric microcharacteristics}
\label{sssec: D-Geometric Microcharacteristic Varieties}
Let $Z=\mathrm{Spec}_{\mathcal{D}}(\mathcal{B})$ be $\D$-subscheme of $J_X^{\infty}E=\mathrm{Spec}_{\D}(\A)$, induced by a $\D$-ideal sequence (\ref{eqn: SES}) with canonical projection $p_{\infty}:Z\rightarrow X.$ Let $\pi_X:T^*X\rightarrow X$ be the cotangent bundle, and consider the pull-back,
\begin{equation}
\label{eqn: pb}
\begin{tikzcd}
    T^*(X,Z)\arrow[d,"\pi_{\infty}^{Z}"]\arrow[dr,"\pi_{\infty}"]\arrow[r] & T^*X\arrow[d,"\pi_X"]
    \\
    Z\arrow[r,"p_{\infty}"] & X,
\end{tikzcd}
\end{equation}
Letting $(x;\xi)$ be the standard coordinates on $T^*X$, then $\mathcal{O}_{J^{\infty}_X(E)}[\xi]$ is the sheaf of functions on $J^{\infty}_XE\times_X T^*X,$ inducing a filtration by homogeneous degrees in $\xi$, on the pull-back (\ref{eqn: pb}), defined by 
\begin{equation}
    \label{eqn: Filtration on D coordinate ring}
\mathrm{Fl}^k\mathcal{O}_{T^*(X,Z)}=\mathcal{O}_{Z}\otimes_{\mathcal{O}_X}\mathrm{Fl}^k\mathcal{O}_{T^*X}, k\geq 0.
\end{equation}
Via the $\D$-spectrum, write
$F^k\mathrm{Spec}_{\mathcal{D}}(\mathcal{A}):=\mathrm{Spec}_X(F^k\mathcal{A}),$ so that 
$\mathrm{Spec}_{\D,T^*}(\A):=\mathrm{Spec}_{\mathcal{D}}(\mathcal{A})\times_X T^*X,$
is filtered with $F^k\mathrm{Spec}_{\D,T^*}(\A)\simeq J_X^k(E)\times_X T^*X.$ Thus, $\mathrm{Spec}_{\D,T^*}(\A)$ is just the (filtered) relative $\D$-scheme $J_X^{\infty}(E)\times_X T^*X.$ It is an affine ind-scheme.

Following \cite{K}, the characteristic variety is defined to be the support of the graded module $\mathcal{C}h(*);$ it is the sub ind-scheme defined by the radical of the annihilator of $\mathcal{C}h(*)$: 
$$\mathrm{supp}(\mathcal{C}h)=\varprojlim \mathrm{supp}(\mathcal{C}h_k), \hspace{2mm} \mathcal{O}(\mathrm{supp}(\mathcal{C}h))=\varinjlim \mathcal{O}(\mathrm{supp}(\mathcal{C}h_k)).$$

\begin{defn}
\label{annideal}
\normalfont
Given a $\D$-algebraic PDE $Z$, its \emph{strict characteristic variety} $\mathrm{char}(Z)$ is  the sub-variety of $T^*(X,Z)$ defined by the \emph{characteristic ideal} $\mathcal{I}_{\mathrm{char}}^{Z}:=\sqrt{\mathrm{Ann}(\mathcal{C}h^{Z})}.$
\end{defn}
The characteristic ideal inherits a filtration from (\ref{eqn: Filtration on D coordinate ring}),
$$\mathrm{Fl}(\mathcal{I}_{\mathrm{char}}^{Z}):=\big\{\mathrm{Fl}^k\big(\mathcal{I}_{\mathrm{char}}^{Z}\big):=\mathcal{I}_{\mathrm{char}}^{Z}\cap F^k(\mathcal{O}_{T^*(X,Z)}),k\geq 0\big\},$$
and if $\iota_k:\mathcal{C}h_k(*)\hookrightarrow \mathcal{C}h^{Z}$ is the canonical inclusion, considering the image $\iota_k(\mathcal{C}h_k^{Z})$ of  $\mathcal{C}h_k^{Z}$ in $\mathcal{C}h^{Z}$ gives an isomorphism 
\begin{equation}
    \label{eqn: Filtered Char ideal}
\mathrm{Fl}^k\big(\mathcal{I}_{\mathrm{char}}^{Z}\big)\simeq \sqrt{\mathrm{Ann}(\iota_k\mathcal{C}h_k^{Z})},\hspace{2mm} k\geq 0.
\end{equation}
Algebraic ideals (\ref{eqn: Filtered Char ideal}) define an analytic subvariety $\mathrm{Char}_k(Z)\subset T^*(X,F^kZ).$ 

Setting
$\mathcal{D}_{Z}:=\mathcal{O}_{Z}\otimes_{p^{-1}\mathcal{O}_X}p^{-1}\mathcal{D}_X,$ it is clear that for $Z=\mathrm{Spec}_{\mathcal{D}_X}(\mathcal{A})$ one has $\mathcal{D}_{\mathrm{Spec}_{\mathcal{D}}(\mathcal{A})}\simeq \mathcal{A}\otimes_{\mathcal{O}_X}\mathcal{D}_X.$ 
It is endowed with its obvious order filtration.
\begin{prop}
\label{prop: Char inclusions}
There is an isomorphism of graded algebras 
$Gr^{F}(\mathcal{D}_{Z})\simeq\mathcal{O}_{T^*(X,Z)},$
and $\Omega_{Z/X}^1$ is filtered by 
$\mathrm{Fl}^k\Omega_{Z/X}^1=\big\{\mathcal{O}_{Z}\otimes_{\mathcal{O}(F^kZ)}\Omega_{F^kZ/X}^1\big\}_{k\geq 0},$
and generated by $\mathcal{D}_{Z}$ and $\mathcal{O}_{Z}\otimes_{\mathcal{O}_{E}^{alg}}\Omega_{E/X}^1.$
The support  
$\mathcal{C}\mathrm{har}_{\mathcal{D}}^1(Z):=\mathrm{supp}\big(Gr(\Omega_{Z/X}^1)\big),$ of this $Gr(\mathcal{D}_{Z})$-module for a $\D$-involutive PDE agrees with the strict characteristic variety. 
\end{prop}
\begin{proof}
From Proposition \ref{prop: Mult by microlocal coordinates} we have for each $\ell \geq 0,$ canonical exact sequences associated to the restrictions of $\pi_{\ell-1}^{\ell}:J_{X}^{\ell}(E)\rightarrow J_{X}^{\ell-1}(E)$ to $Z$ i.e. $F^{\ell}Z\subset J_X^{\ell}(E),$
\begin{equation}
\label{eqn: Filtered Kahler sequences}
\mathcal{O}_{F^{\ell}Z}\otimes_{\mathcal{O}_{F^{\ell-1}Z}}\Omega_{F^{\ell-1}Z/X}^1\rightarrow \Omega_{F^{\ell}Z/X}^1\rightarrow \Omega_{F^{\ell}Z/F^{\ell-1}Z}^1\rightarrow 0.
\end{equation}
Applying $\mathcal{O}(Z)\otimes_{\mathcal{O}(F^{\ell}Z)}(-)$ to (\ref{eqn: Filtered Kahler sequences}), gives
$$\mathcal{O}_{Z}\otimes_{\mathcal{O}_{F^{\ell-1}Z}}\Omega_{F^{\ell-1}Z/X}^1\rightarrow \mathcal{O}_{Z}\otimes_{\mathcal{O}_{F^{\ell}Z}}\Omega_{F^{\ell}Z/X}^1\rightarrow \mathcal{O}_{Z}\otimes_{\mathcal{O}_{F^{\ell}Z}}\Omega_{F^{\ell}Z/F^{\ell-1}Z}^1\rightarrow 0.$$
Then, the desired filtration on $\Omega_{Z/X}^1$ is defined via
$$\mathrm{Fl}^{\ell}\Omega_{Z/X}^1:=Im(\mathcal{O}_{Z}\otimes_{\mathcal{O}_{F^{\ell}Z}}\Omega_{F^{\ell}Z/X}^1\hookrightarrow \Omega_{Z/X}^1),\ell\geq 0.$$
Since $\Omega_{Z/X}^1\simeq lim_{\ell} \Omega_{F^{\ell}Z/X}^1,$ the desired result follows by noting that one has
$\mathcal{D}_{Z}:=\mathcal{O}_{Z}\otimes_{p^{-1}\mathcal{O}_X}p^{-1}\mathcal{D}_X,$ and thus for $Z=\mathrm{Spec}_{\mathcal{D}_X}(\mathcal{A})$ one has $\mathcal{D}_{\mathrm{Spec}_{\mathcal{D}}(\mathcal{A})}\simeq \mathcal{A}\otimes_{\mathcal{O}_X}\mathcal{D}_X,$ as an $\A-\D_X$-bimodule. With its obvious order filtration, setting $p_{\infty}:\mathrm{Spec}(\A)\rightarrow X,$ the sheaf of algebras $\A[\D_X]$ on $X$ is isomorphic to $p_{\infty*}C_{\D_{\mathrm{Spec}(\A)}}(\D_{\mathrm{Spec}(\A)/X}),$ where $C_{\D_{\mathrm{Spec}(\A)}}$ denotes the centralizer of the sheaf of vertical differential operators in $\D_{\mathrm{Spec}(\A)}.$
Comparing the associated graded objects $\mathrm{gr}_{F_{ord}}(\D_{Z})$ is isomorphic to the quotient filtration associated with $\mathrm{Fl}^{\ell}\Omega_{Z/X}^1.$
In order to prove the independence of the filtration we must invoke Proposition \ref{prop: Omega and Ch isom}, given below. To this end, note that for each $k\in \mathbb{Z}$ there is a shift endo-functor $[k]:\mathrm{Sch}_{\mathcal{D}_X}^{inv}\rightarrow \mathrm{Sch}_{\mathcal{D}_X}^{inv},$ defined by shifted filtration $F^{\ell}\mathrm{Spec}_{\mathcal{D}_X}(\mathcal{B})[k]:=\mathrm{Spec}_{X}(F^{\ell}\mathcal{B}[k]),$ with $F^{\ell}\mathcal{B}[k]:=F^{\ell+k}\mathcal{B}.$
Then, there is a canonical map $i_k:\mathrm{Spec}_{\mathcal{D}}(\mathcal{B})[k]\rightarrow \mathrm{Spec}_{\mathcal{D}}(\mathcal{B}),$ and by Proposition \ref{prop: Omega and Ch isom}  $i_k^*\mathcal{C}h=0$ if and only if $\mathcal{C}h\equiv 0.$
\end{proof}

The characteristic variety $\mathcal{C}\mathrm{har}_{\mathcal{D}}(Z)$ appearing in Proposition \ref{prop: Char inclusions}, is the \emph{$\D$-geometric characteristic variety}. 
\begin{rem}
    \label{rem: s-microchars}
    \normalfont 
Given an involutive analytic subset $V\subset T^*X,$ in \emph{loc.cit.} we defined for each $s\geq 1,$ the $s$-microcharacteristic varieties $\mathcal{C}\mathrm{har}_{\mathcal{D},V}^{(s)}(Z)\subset Z\times_X T_VT^*X.$ For our purposes we need only the $\D$-geometric characteristic variety which is given by the $1$-microcharacteristic variety along the zero-section Lagrangian subvariety $V=T_X^*X$ i.e.
$\mathcal{C}\mathrm{har}_{\mathcal{D}}(Z)\simeq \mathcal{C}\mathrm{har}_{\mathcal{D},T_X^*X}^{(1)}(Z).$
\end{rem}

$\D$-characteristic varieties are invariant under $\D_X$-isomorphisms.

\begin{prop}
\label{prop: D-Isom of Chars}
    Let $f:\EQ\rightarrow Z$ be a $\D_X$-isomorphism between $\D_X$-smooth schemes. Then there is a $\D$-isomorphism of $f^*\mathcal{O}_{Z}[\mathcal{D}_X]$-modules, 
    $f^*\Omega_{Z/X}^1\simeq\Omega_{\EQ[k]/X}^1,$ which induces an isomorphism of $\D$-geometric characteristic varieties
    $\mathcal{C}\mathrm{har}_{\mathcal{D}}(Z)\simeq \mathcal{C}\mathrm{har}_{\mathcal{D}}(\EQ).$ Moreover, there is an induced isomorphism of complexes of $\mathcal{O}_{\EQ}[\mathcal{D}_X]$-modules $\Omega_{\EQ}^{*}\simeq f^*\Omega_{Z}^{*}.$
\end{prop}
As in the linear setting \cite{Ga}, the characteristic ideal is involutive/coisotropic.
\begin{prop}
\label{prop: D-Char is coisotropic}
$\mathrm{Char}(Z)$ is stable under the Poisson bracket, compatible with the symplectic embedding of $(T^*X,\{-,-\})$ into $\big(T^*(X,Z),\{-,-\}\big).$
\end{prop}
\begin{proof}
Suppose $f,g\in \mathcal{I}_{\mathrm{char}}.$ We want to show for every solution or formal solution $\varphi$, that $\varphi^*f,\varphi^*g$ vanish on $\varphi^{-1}\mathrm{Char}(Z)$ and that $\varphi^*\{f,g\}=\{\varphi^*f,\varphi^*g\}$.
Choose a solution of $Z$ i.e. a morphism of $\mathcal{D}$-schemes $(U,\mathcal{O}_U)\xrightarrow{\varphi}Z,$ with $U$ an open subset of $X$, one has that 
$\varphi^{-1}\mathrm{Char}(Z)=\mathrm{supp}(\varphi^*\mathcal{M}^{Z}),$
in $T^*X.$ Then, the $\mathcal{D}$-module $\varphi^*\Omega_{Z/X}^1$ is coherent and the characteristic variety of it, $\varphi^{-1}\mathrm{Char}(Z)$, is stable by the Poisson bracket by Gabber's theorem in the linear setting.
It suffices to then remark that the Poisson bracket $\{-,-\}$ on $\mathcal{O}_{T^*(X,Z)}$ has the following property: for every solution $\varphi$ or formal solution $\varphi_{\infty},$ one has that if $\varphi^*\{f,g\}$ vanishes on $\varphi^{-1}\mathrm{Char}(Z),$ then $\{f,g\}$ vanishes on $\mathrm{Char}(Z)$. In particular, $\{f,g\}\in \mathcal{I}_{\mathrm{char}}.$ 
\end{proof}

\subsection{Non-characteristic subspaces}
\label{ssec: NC subspaces}
Let $f:Y\rightarrow X$ be a morphism of complex analytic manifolds and consider $p_X\in \mathring{T}^*X:=T^*X\backslash T_X^*X$ with $\pi_X(p_X)\in Y.$ 
Following \cite{KS}, consider the diagram
$$T^*Y\xleftarrow{f_d}Y\times_X T^*X\xrightarrow{f_{\pi}}T^*X,$$
and put $p:=f_{\pi}^{-1}(p_X),p_Y:=f_{d}f_{\pi}^{-1}(p_X)$ so $p_X=f_{\pi}(p).$

Fix an $\mathcal{A}[\mathcal{E}_X]$-module $\mathcal{M}.$ It is \emph{non-characteristic} for $f:Y\hookrightarrow X$ at the point $p\in Y\times_X T^*X$ if and only if: there exists an open neighbourhood $U$ of $p$ such that $f_d$ is finite on $U\cap f_{\pi}^{-1}(\mathrm{supp}(\mathcal{M})).$
In other words, for a $\D$-algebra $\mathcal{B}_X$, it reads as
$$f_{d}^{-1}(p_Y)\cap f_{\pi}^{-1}\big(\mathrm{Char}_{\mathcal{D}}(\mathcal{B})\big)\subset p_{\infty}^{-1}\{ p \}.$$

Making use of the notion of a $1$-non-microcharacteristic subset $V\subset T^*X$ in the $\D$-algebraic setting \cite[Sec. 7]{KSY}, $f:Y\subset X$ is non-characteristic morphism if it is a $1$-micro-
noncharacteristic morphism for $\mathcal{B}$ via the zero section Lagrangian sub-manifold $V=T_Y^*Y.$ 

We allow for $S$-families of $\D_X$-algebraic PDEs, by imposing that $f:X\rightarrow S$ is a \emph{non-characteristic fibration}, following \cite{LZ}. In particular, the fibers $X_s:=f^{-1}(s)$ are non-characteristic for each $s\in S$. Then, by results of \cite{K}, for any coherent $\D_X$-module $\M_X$, the restricted systems $\M_{X_s}$ satisfy $H^k(\M_{X_s})=0,k\neq 0$ and $H^0(\M_{X_s})$ is a coherent $\D_{X_s}$-module.

\subsection{Regularity criterion for $\D$-schemes}
\label{ssec: Regularity criterion for D-schemes}
We propose an analog of regularity, in the sense of Castelnuovo-Mumford \cite{M}, for $\D_X$-ideal sheaves. It is given by  Definition \ref{defn: D-Geometric Regularity} below, and plays an essential role in the construction of a Quot functor for smooth, reduced $\D$-schemes. Together with Serre's Theorem B \cite{Se}, boundedness for families of $\D$-ideal sheaves is established using this differential regularity. We refer to \cite{Ma2,Ma3} for detailed explanations.

\begin{rmk}
Prolonging a system gives rise to sub-varieties $Z_{k+\ell}\subset J^{k+\ell}_XE$, but one should check the tower $\{Z_{k+\ell}\}_{\ell\geq 0}$ stabilizes. A priori it is unclear, as one is adding new equations and variables. An analog of Hilbert's basis theorem is therefore required in this setting which states the ascending chain of ideals constructed this way does indeed stabilize and if the ideal is not the entire ring, it defines a non-empty variety.
\end{rmk}

In the $\D_X$-geometric setting we employ Ritt's analog of the Hilbert basis theorem for differential algebras \cite{Ri}.

\begin{thm}[Ritt-Radenbush]
Every finitely generated differential algebra of characteristic zero is radically Noetherian, i.e., satisfies the ascending
chain condition for radically closed differential ideals. 
\end{thm}
This result has an important corollary, namely that every finitely generated differential domain over a field of characteristic zero is equationally Noetherian, which in this setting states for all systems of differential equations $\{F_i=0\}_{i\in I}$ there exists a finite-subsystem $\{F_{i_1},\ldots,F_{i_m}\}\subset \{F_i\}$ such that $\mathrm{Sol}_{\D}(\{F_i=0\}_{i\in I})\simeq \mathrm{Sol}_{\D}(\{F_{i_j}=0\}_{j=1}^m).$

\begin{prop}
\label{prop: DiffEqnNotherian}
Let $R$ be arbitrary associative
and commutative ring with identity containing $\mathbb{Q}$ and consider
 a differential $R$-algebra $\A$ without nilpotent elements satisfying the
ascending chain condition for radically closed differential ideals. Then $\A$ is equationally
Noetherian.
\end{prop}

The next result is clear.
\begin{prop}
    Let $R$ be an associative and commutative radically Noetherian domain. Then
every differential polynomial $R$-algebra is equationally Noetherian.
\end{prop}
\begin{proof}
    Follows from the Ritt-Radenbush theorem and Proposition \ref{prop: DiffEqnNotherian}.
\end{proof}

Let $Z=\{Z_k\}$ be a relative-algebraic $\D$-scheme. Then it suffices to require only that
Zariski locally, the map $Z_{m+1}\rightarrow Z_m$ is surjective for large $m\gg0.$ This is guaranteed for $\D_X$-integrable ideals.

\begin{prop}
\label{prop: Omega and Ch isom}
    Suppose that $\mathcal{I}$ is a $\D_X$-formally integrable $\D$-ideal sheaf. Then, after possibly shrinking to a Zariski-open dense subset, there is an isomorphism of $\mathcal{O}_X$-modules,
    $\Omega_{Z_{m+1}/Z_m}^1\simeq \mathcal{C}h_{m+1}.$
\end{prop}

In particular, if $\mathcal{C}h_{m+1}$ is a locally free sheaf and $Z_m$ is smooth, then $Z_{m+1}$ is also smooth. Moreover, if $\mathcal{C}h_{\ell}$ is locally free on $\mathcal{O}_{Z_{\ell}}$, then $\mathrm{dim}(\mathcal{C}h_{\ell,z})$ is locally constant on $|Z_{\ell}|$ for point $z\in Z_{\ell}.$ Conversely, if $\mathrm{dim}\mathcal{C}h_{\ell}:|Z_{\ell}|\rightarrow \mathbb{N}$ is a locally constant function, then $\mathcal{C}h_{\ell}$ is locally of the form  $\mathcal{O}_{Z_{\ell}}^{\oplus r}\rightarrow \mathcal{O}_{Z_{\ell}}^{\oplus s}\rightarrow \mathcal{C}h_{\ell}\rightarrow 0.$ These facts suggest introducing a PDE analog of regularity.
\begin{defn}
\label{defn: D-Geometric Regularity}
\normalfont 
A $k$-th order system equations $Z_k\subset J_X^kE$, imposed on sections of $\pi:(E,\mathcal{O}_E)\rightarrow (X,\mathcal{O}_X)$ is \emph{$m$-regular} if
\begin{enumerate}
    \item The $m$-fold prolongation defines a smooth algebraic variety $Z_m\subset J_X^m(E)$ and $Z_{m+1}\rightarrow Z_m$ is smooth and surjective;

    \item The modules $\mathcal{C}h_m$ and $\mathcal{C}h_{m+1}$ for $m$ and $(m+1)$-st prolongations are locally free sheaves of $\mathcal{O}_X$-modules;

    \item For every point $\theta_m\in Z_m$ the (fiber of) the module of relative Kähler forms $\Omega_{Z_{m+1}/Z_{m}}^1$ is $m$-regular in the sense of Castelnuovo-Mumford. 
\end{enumerate}
\end{defn}
An easy observation which is an important structural fact for the theory is the following.

\begin{prop}
    Let $Z_k$ be $k$-th order system of non-linear PDEs. Suppose that it is $m$-regular. Then the symbol ideal $\EuScript{J}_m$ of its $m$-th prolongation is $(m+1)$-regular and if $E$ is reduced both $\mathcal{C}h_m$ and $\mathcal{C}h_{m+1}$ are locally free sheaves.
\end{prop}
\begin{proof}
By assumptions, one can see that condition (3) in Definition \ref{defn: D-Geometric Regularity} equivalently means that the symbol ideal $J_m$ of the $m$-th prolongation is $(m+1)$-regular. For $Z_k\subset J^k(E),$ as above with $\mathcal{C}h_k=\Omega_{Z_k/J^{k-1}}^1.$ Since $E$ is reduced the fact that $\mathcal{C}h$ is $m$-regular follows by noting that
$\mathcal{C}h^{An}:=\mathcal{C}h\otimes_{\mathcal{O}(Z_k)}\mathcal{O}_{Z_k^{An}},$ contains direct summands
$\mathcal{C}h_m^{An}$ and $\mathcal{C}h_{m+1}^{An},$ as locally free sheaves on $Z_k^{An}$.
\end{proof}
A differentially generated $\D$-ideal $\mathcal{I}:=\{F_A\}_{\D}$ implicitly corresponds to some infinite prolongations of systems of relatively-algebraic differential equations. Thus, assuming they are additionally $\D$-involutive, by the Cartan-Kuranishi theorem we can always find some $m$ large enough (i.e. always exists) for which conditions of Definition \ref{defn: D-Geometric Regularity} are satisfied.
It is then convenient to introduce another notion packaging this fact.
\begin{defn}
\label{defn: D-geom m-involutive}
\normalfont
Let $\mathcal{A}$ be a $\D_X$-smooth commutative $\D$-algebra. Let $\mathcal{I}_X$ be a $\D$-ideal which is formally integrable, and denote by $Z:=\mathrm{Spec}_{\D}(\mathcal{A}/\mathcal{I})\hookrightarrow \mathrm{Spec}_{\D}(\A)$ the closed $\D$-subscheme. Then $Z$ (equiv. $\mathcal{I}$) is said to be \emph{Spencer regular of degree $r$} if: (i) $F^r(Z)$ is $r$-regular in the sense of Definition \ref{defn: D-Geometric Regularity} i.e. $Z_r,\mathrm{pr}_1(Z_r)$ are smooth algebraic varieties and $\mathrm{pr}_1(Z_r)\rightarrow Z_r$ is surjective with $\mathcal{C}h_r$ of $Z_r$ is CM-$r$-regular, and (ii) for all $\ell\geq 0$ one has $Z_{\ell+r}=\mathrm{Pr}_\ell(Z_r).$ Write simply $\mathrm{Reg}_{\D}(Z)=r.$ 

A $\D$-scheme is \emph{Spencer regular} if it is Spencer $r$-regular for some $r\geq 0.$
\end{defn}

Some comments on this definition are in order. First, condition (i) equivalently states $Fl^m(\mathcal{I})$ is \emph{maximal rank}, in the sense that the matrix of size
$\mathrm{codim}_{\mathcal{D}}(\mathcal{I})\times \big(n+m\cdot C_{\mathrm{ord}(F)}^{n+\mathrm{ord}(F)}\big),$
given by the 
Jacobian with respect to $(x,u_{\sigma}^{\alpha})$-variables i.e.
$\mathrm{Jac}_{\mathcal{I}}(x,u_{\sigma}^{\alpha}):=\big(\frac{\partial F_A}{\partial x^i},\frac{\partial F_A}{\partial u_{\sigma}^{\alpha}}\big),$
satisfies $
\mathrm{rank}(\mathrm{Jac}_{\mathcal{I}})=\mathrm{codim}_{\mathcal{D}_X}(\mathcal{I}_X).$ In other words, integrability conditions (thus involutivity conditions) are in correspondence with a deficit in the rank of the symbol (of the linearization in non-linear setting). Then, condition (ii) states that the (generically) infinitely generated $\D$-ideal $\mathcal{I}$ is, in fact finitely $\D$-generated. Namely, it is determined by differential relations of orders $\leq r.$ This is emphasized by writing $\mathcal{I}\simeq \mathcal{D}_X\bullet F^r\mathcal{I},$ with $F^r\mathcal{I}=\mathcal{I}\cap F^r\mathcal{O}(J_X^{\infty}E),$ a coherent $\mathcal{O}_X$-module.  

This rank may be determined using the strict characteristic variety.

\begin{rmk}
Indeed, suppose $\{F=0\}$ is a system of $N$-equations of pure order $k.$ Linearizing around a point $v$ gives,
    $\ell(F^A)(v)=\sum_{|\sigma|=k}\sum_{j=1}^mp_{j}^{i,,\sigma}(x;\xi)v_{\sigma}^j,$ with the symbol matrix
    $\sigma(\ell_{F^A})(x;\xi):=||p_j^i(\xi)||_{i=1,\ldots,N,j=1,\ldots,m}.$
    Then, $\mathrm{char}(Z)=\{\xi\in T_x^*X|\mathrm{rank}(\sigma_{\ell(F)}(\xi))<m\}.$
\end{rmk}

The linearization of (resp. prolongation) of an involutive non-linear PDE (resp. involutive and formally integrable non-linear PDE) is again involutive (see e.g. \cite{Ts}).

The results of this section are now studied within a cohomological framework designed for detecting obstructions to integrability \cite{G,Sp,Sp2,Q}, but which suggests a precise criteria for the measure of the complexity of a $\D$-ideal based on homological finiteness of its associated symbol/characteristic module.

\section{Spencer sequences and their cohomologies}
In this section we recall the (first) Spencer complex, and the $\delta$-Spencer complex \cite{Sp,Sp2}. 

Spectral sequences computing Spencer cohomologies for formally integrable systems of noncharacteristic fibered manifolds was constructed in \cite{G,LZ}, providing analogs of the Leray-Serre sequence for de Rham cohomologies (see also \cite{LR}).


These results are now adapted to the case of sheaves over a fixed smooth $\D$-scheme $J_X^{\infty}E.$ To study Spencer cohomology of relatively-algebraic PDEs, we work with sheaves of $\mathcal{A}[\mathcal{D}]:=\mathcal{A}\otimes_{\mathcal{O}_X}\mathcal{D}_X$-modules, e.g. associated to characteristic modules of a $\D$-ideal sheaf over a fixed $\D$-algebra $\mathcal{A}$ of jets. We prove the resulting non-linear analogs of Jet-Spencer and $\delta$-Spencer cohomologies \cite{G,G2}, are compatible with the previously introduced Spencer regularity, just as sheaf-cohomology is with Casetlnuovo-Mumford regularity.

\subsection{Jet-Spencer sequences}
Let $V$ be a finite-dimensional vector space. Consider the $n$'th polynomial de Rham complex
$$0\rightarrow Sym^n(V)\xrightarrow{d} Sym^{n-1}(V)\otimes V\xrightarrow{d} Sym^{n-2}(V)\otimes \wedge^2 V\xrightarrow{d} \cdots,$$ with differential $d(w_1\cdots w_p\otimes v):=\sum_{i=1}^pw_1\cdots w_{i-1}w_{i+1}\cdots w_p\otimes w_i\wedge v.$
We are interested in a polynomial de Rham complex canonically associated to a symbolic system associated with a formally integrable PDE. It is obtained by restricting a certain kernel complex obtained from the related \emph{Jet-Spencer sequences}, that we now describe.

Let $Z:=\{Z_k\subset J_X^kE:k\geq0\}$ be as before. There is a canonical isomorphism of vector bundles, 
\begin{equation}
    \label{eqn: Cosymbol iso}
Sym^k(T_X^*)\otimes E\xrightarrow{\simeq} Ker(J_X^{k}E\rightarrow J_X^{k-1}E),
\end{equation}
defined by $df_1\cdots df_k\otimes e\mapsto \delta_{f_1,\cdots,f_k}(j_k)(e),e\in E,$ where $\delta_{f_1,\cdots,f_k}:=\delta_{f_1}\circ \cdots \delta_{f_k},$ is the nested commutator, where $\delta_{f_i}:=[f_i,-],$ for $f_1,\ldots,f_k\in \mathcal{O}_X.$


The $k$-th Jet-Spencer complex of $E$, denoted by $\mathcal{J}^k\mathcal{S}p_X^{\bullet}(E)$ is given by,
\begin{equation}
    \label{eqn: k-JetSpencer}
    0\rightarrow E\xrightarrow{j_k}J_X^kE\xrightarrow{D}T_X^*\otimes_{\mathcal{O}_X}J_X^{k-1}E\xrightarrow{D} \cdots\rightarrow \wedge^nT^*X\otimes_{\mathcal{O}_X}J_X^{k-n}E\rightarrow 0,
\end{equation}
whose differential is uniquely determined on corresponding sections by:
\begin{itemize}
\item[(i)] $D(\theta\otimes \omega)\otimes \eta):=D(\theta)\wedge\eta+(-1)^k(\pi_{\ell,\ell-1}\otimes Id)(\theta\wedge d_{dR}\eta),$ where $\theta\in J_X^{\ell}(E)\otimes \Omega_X^k$ and $\eta\in \Omega_X^k$;
\item[(ii)] the fact that the sequence determined by the map $j_k:H^0(X,E)\otimes\mathcal{O}_X\rightarrow J_X^k(E),$ sending $s\in H^0(X,E)$ and a $k$-rational point $x_0\in X$ to the $k$-th order Taylor expansion of $s$ around $x_0$, followed by the Spencer operator $D$, is exact.
\end{itemize}
Extending (ii) via the Leibniz rule, one obtains the (first) Spencer sequence (see \cite{G,LZ}). For each $k>0$, there are morphisms $\mathcal{J}^k\mathcal{S}p_X^{\bullet}(E)\rightarrow \mathcal{J}^{k-1}\mathcal{S}p_X^{\bullet}(E),$ of complexes (\ref{eqn: k-JetSpencer}), induced from the jet projections $\pi^k_{k-1}$:  
\[
\begin{tikzcd}
0\arrow[r] & E\arrow[d] \arrow[r] & J_X^kE\arrow[d] \arrow[r] & T_X^*\otimes_{\mathcal{O}_X}J_X^{k-1}E\arrow[d] \arrow[r] & \wedge^2T_X^*\otimes_{\mathcal{O}_X}J_X^{k-2}E\arrow[d]\rightarrow \cdots
\\
0 \arrow[r] & E\arrow[r] & J_X^{k-1}E\arrow[r] & T_X^*\otimes_{\mathcal{O}_X}J_X^{k-2}E\arrow[r] & \wedge^2 T_X^*\otimes_{\mathcal{O}_X}J_X^{k-3}E\rightarrow \cdots
\end{tikzcd}
\]

By (\ref{eqn: Cosymbol iso}), we consider its kernel complex - the so-called \emph{$\delta$-Spencer complex of $E$.} We describe it now, and explain its usefulness for studying symbolic systems in the formal theory of PDEs.
\subsection{$\delta$-Spencer sequences}
Let $\mathcal{E}$ be a finitely generated locally free sheaf of $\mathcal{O}_X$-modules and consider the \emph{$\delta$-Spencer complex} (of order $k$), the polynomial de Rham complex with coefficients in $\mathcal{E}$,
\begin{equation}
\label{eqn: Spencer delta}
0\rightarrow Sym^k(\Omega_X^1)\otimes_{\mathcal{O}_X}\mathcal{E}\xrightarrow{\delta}\cdots \xrightarrow{\delta}\wedge^{\ell}\Omega_X^1\otimes_{\mathcal{O}_X}Sym^{k-\ell}(\Omega_X^1)\otimes_{\mathcal{O}_X}\mathcal{E}\rightarrow\cdots,
\end{equation}
where $Sym^i(\Omega_X^1)=\emptyset,i<0.$ The differential at the $s$-th term,
$$\delta:\Omega_X^s\otimes Sym^{k-s}\Omega_X^1\otimes\mathcal{E}\rightarrow \Omega_X^{s+1}\otimes Sym^{k-s-1}\Omega_X^1\otimes\mathcal{E},$$
is defined by $\delta(\omega\otimes U\otimes e):=(-1)^{s}\omega\wedge i(U)\otimes e,$ for $\omega\in \Omega_X^s,U\in Sym^{k-s}(\Omega_X^1)$ and $e\in \mathcal{E},$ where we used the inclusion
$$i:Sym^{k-s}\Omega_X^1\hookrightarrow \Omega_X^1\otimes Sym^{k-s-1}(\Omega_X^1).$$

Applied to $\D$-schemes, we consider a $k$-th order PDE $Z\subset J_X^k(\mathcal{E})$ defined by an operator $P\in F_{ord}^k\mathcal{D}_X(\mathcal{E},\mathcal{F})$ (equiv. as an $\mathcal{O}_X$-module homomorphism $F_P:J_X^k(\mathcal{E})\rightarrow \mathcal{F}$). It is prolonged to $\psi_{\ell}(P):J_X^{k+\ell}(\mathcal{E})\rightarrow J_X^{\ell}(\mathcal{F})$ for each $\ell\geq 0$ in the standard manner \cite{KV}.

Thus there is a family
\begin{equation}
\label{eqn: Symbol family}
\sigma_{\ell,k}(F_P):Sym_{\mathcal{O}_X}^{k+\ell}(\Omega_X^1)\otimes_{\mathcal{O}_X}\mathcal{E}\rightarrow Sym_{\mathcal{O}_X}^{\ell}(\Omega_X^1)\otimes \mathcal{F},
\end{equation}
of $\mathcal{O}_X$-module homomorphism. Each map $\sigma_{\ell,k}(F_P)$ is the $\ell$-th prolongation of the symbol $\sigma_k(F_P).$
Associated with (\ref{eqn: Symbol family}) are the family of so-called $\mathcal{O}_X$-coherent \emph{(co)symbolic} submodules\footnote{More commonly denoted by $\mathfrak{g}=(\mathfrak{g}_k)_k$.}, 
\begin{equation}
\label{eqn: Symbol}
\big\{\mathcal{N}^{k+\ell}:=\mathrm{ker}\big(\sigma_{\ell,k}(F_P)\big)\subset Sym_{\mathcal{O}_X}^{k+\ell}\Omega_X^1\otimes\mathcal{E}\big\}_{\ell\geq 0}.
\end{equation}
Denote by $\mathrm{Pr}_{r}Z\subset J_X^{r+k}(\mathcal{E})$ the $r$-th prolongation of $Z$ similarly for the symbols. Formally, they are given by
$\mathrm{Pr}_1(\mathcal{N}):=(T^*X\otimes \mathcal{N})\cap \mathrm{Sym}^{k+1}T^*X\otimes E,$
and more generally,
$$\mathrm{Pr}_{r+1}(\mathcal{N})=\mathrm{Pr}_1\big(\otimes_{i=1}^rT^*X\otimes \mathcal{N}_k)\cap (\mathrm{Sym}^{k+r}T^*X\otimes E).$$

Spencer differentials restrict to the equation sub-variety and its prolongations 
$$D|_{\mathrm{Pr}_{r}Z}:\mathrm{Pr}_rZ\otimes \Omega_X^k\rightarrow \mathrm{Pr}_{r-1}Z\otimes \Omega_X^{k+1},$$ for $0\leq k\leq \mathrm{dim} X,$ and descend to a differential on symbols (\ref{eqn: Symbol}).

Therefore, given a geometric symbol $(\mathcal{N}^{(r)})_{r\in \mathbb{N}}$ associated with an algebraic PDE $Z=\{Z_k\subset J^kE\},$ of (max) order $\leq k$, then setting $\mathcal{N}^{(i)}=0,i<0$ and $\mathcal{N}^{(i)}:=\mathrm{Sym}^i(T^*X)\otimes T_{E/X}$ for $0\leq i\leq k,$ one has the restrictions of (\ref{eqn: Spencer delta}), given by 
\begin{equation}
\label{eqn: Spencer delta symbol}
0\rightarrow \mathcal{N}^{(k+r)}\rightarrow T^*X\otimes \mathcal{N}^{(k+r-1)}\rightarrow\cdots\rightarrow \wedge^{dim_X}T^*X\otimes \mathcal{N}^{(k+r-dim_X)}\rightarrow 0,
\end{equation}
whose cohomology at the term $\mathcal{N}^{k+\ell-i}\otimes\Omega_X^i,$ is denoted $\mathcal{H}^{k+\ell,i}.$
\begin{rmk}
\label{SpRmk}
    Spencer cohomologies encode important information about the system e.g. $\mathcal{H}^{*,1}(Z)$ count the number of generators of the (dual) symbol module, $\mathcal{H}^{p,1}(Z)$ count the number of equations of order $p$, and $\mathcal{H}^{*,2}(Z)$ counts compatibility conditions etc. The Euler-characteristic of $Z$ is
    $\chi(Z):=\sum_{i=0}^{n}(-1)^idim\big(\mathcal{H}^i(Z)\big).$
    Finite-dimensionality of $\mathcal{H}^*(Z)$ can be guaranteed, for instance, if the system $Z$ is elliptic $X$ is compact. 
\end{rmk}

 The situation is summarize by the commuting diagram,
\begin{equation}
\label{eqn: Diagram}
\begin{tikzcd}
& 0\arrow[d] & 0\arrow[d] & & 0\arrow[d] 
\\
0\arrow[r] & \mathcal{N}^{(r)} \arrow[d] \arrow[r]& \mathcal{N}^{(r-1)}\otimes\Omega_X^1\arrow[d]\arrow[r] & \cdots \arrow[r]& \arrow[d]\mathcal{N}^{(r-n)}\otimes\Omega_X^n\rightarrow 0
\\
0\arrow[r] & \mathrm{Pr}_rZ\arrow[d]\arrow[r] & \mathrm{Pr}_{r-1}Z\otimes\Omega_X^1\arrow[d]\arrow[r] & \cdots\arrow[r] & \mathrm{Pr}_{r-n}Z\otimes\Omega_X^n \arrow[d] \rightarrow 0
\\
0\arrow[r] & \mathrm{Pr}_{r-1}Z \arrow[d]\arrow[r]& \mathrm{Pr}_{r-2}Z\otimes\Omega_X^1\arrow[d]\arrow[r]& \cdots \arrow[r] & \mathrm{Pr}_{r-n-1}Z\otimes \Omega_X^n\rightarrow 0 \arrow[d]   
\\
& 0 & 0&  & 0
\end{tikzcd}
\end{equation}

Sub-complex (\ref{eqn: Spencer delta symbol}) is called the $\delta$-Spencer complex of the operator $F_P,$ and whose cohomologies are called $\delta$-Spencer cohomology of the operator.
\begin{defn}
    \label{defn: Involutive operator}
    \normalfont 
    The operator $F_P$ of order $\leq k$ is said to be:
    \begin{enumerate}
        \item \emph{Involutive}, if $H^{k+\ell,i}(F_P)=0,\forall i\geq 0;$

        \item \emph{Formally integrable}, if for each $\ell\geq 0,$ the spaces $F^{\ell}\EQ_{F_P}:=\mathrm{ker}(\psi_{\ell}(F_P))\subset J^{k+\ell}(\mathcal{E})$ and $\mathcal{N}^{k+\ell}$ are projective modules and whose induced maps 
        $\mathrm{ker}(\psi_{\ell}(F_P))\rightarrow \mathrm{ker}(\psi_{\ell-1}(F
_P))$ are surjective.
\end{enumerate}
\end{defn}
Involutivity of the operator is not necessary for the formal exactness of the compatibility complex. This result is known as the (formal) $\delta$-Poincare lemma \cite{KLV}. Namely, the columns of (\ref{eqn: Diagram}) are exact and there exists some $r_0$ depending on $n:=\mathrm{dim}(X)$ the rank of $E$ and $k$, such that for $r\geq r_0+n,$ the Spencer cohomologies stabilize.
\begin{prop}
\label{prop: Delta-Poincare}
If the underlying algebra over which our modules on which the operator $F_P$ is acting is Noetherian, then for any operator $F_P$ one can find an integer $m_0$ depending only on the $m:=\mathrm{rank}(\mathcal{E}),$ the base space dimension and the order of the operator such that $H^{k+\ell,q}(F_p)=0,$ for every $\ell \geq m_0$ and $q\geq 0.$
\end{prop}
Given a $\D$-scheme $Z$ corresponding to a relatively-algebraic PDE defined by a non-linear operator, we introduce a certain resolution (the horizontal compatibility complex) associated to its universal linearization operator $\ell_{Z}$ and thus its characteristic $\D$-module (\ref{eqn: Char Module of A}). For this, we must adapt the construction of these complexes to the category of $\mathcal{O}(Z)\otimes_{\mathcal{O}_X}\D_X$-modules.

In this case, the characteristic and symbolic modules e.g. as in Proposition \ref{prop: Omega and Ch isom} are canonically identified with an associated graded sheaf determined by $\ell_{Z}.$

\begin{prop}
\label{prop: Redux}
 Consider a PDE given by the sequence $Z_k\subset J_X^k(E)$, and set $\A_X,\mathcal{B}_X$ as above, with usual filtrations $F^k\A,F^k\mathcal{B}.$
Canonical jet-projection maps restrict to $Z_{k+1}\rightarrow Z_k$ and dually read as $F^k\mathcal{B}_X\hookrightarrow F^{k+1}\mathcal{B}_X.$
Then, for each $i\geq 0$, there exists an $F^i\mathcal{B}_X$-module $\mathrm{Gr}_i^{F}(\M_{\mathsf{F}})$ which is coherent as an $\mathcal{O}_X$-module. Moreover, its support is contained in $Z_i$ and the stalk of this sheaf at $z_i\in Z_i$ is simply
$$(\mathrm{Gr}_i^{F}(\M_{\mathsf{F}}))_{z_i}\simeq ker(Tp_{i-1}^i|_{Z_i}:TZ_i\rightarrow TZ_{i-1})^*.$$
\end{prop}
\begin{proof}
See Appendix \ref{proof of redux}.
\end{proof}

We now briefly mention the extension of Spencer cohomology to treat the operator of universal linearization associated to a $\D$-scheme $Z.$ 
\\

\noindent\textbf{Spencer sequences for $\AD$-modules.}
There exists a horizontal analogue of the Goldschmidt-Spencer formal theory of linear
differential equations \cite{G,KV}. Importantly, if $Z^{\infty}$ is a formally integrable PDE, universal linearization defines a formally integrable horizontal PDE \cite{Ts}.

Consider a $k$-th order operator in total derivatives, $P\in F^k\mathcal{A}[\D](P,P_2).$ There exists a unique $\A$-homomorphism $\varphi_{\mathsf{P}}:\overline{J}^k(P)\rightarrow P_1.$ The $r$-th prolongation is defined
$$\mathrm{Pr}_{r}(\mathsf{P}):=\overline{j}_r\circ \mathsf{P}\in F^{k+r}\mathcal{A}[\D](P,\overline{J}^{r}(P_1)),$$
and the corresponding $\A$-homomorphism is $\varphi_{\mathsf{P}}^{(r)}:=\varphi_{\mathrm{Pr}_r(\mathsf{P})}:\overline{J}^{k+r}(P)\rightarrow \overline{J}^{r}(P_1).$
Continuing in this way, we obtain
\begin{equation}
    \label{eqn: InfiniteAhom}
    \varphi_{\mathsf{P}}^{(\infty)}:=\varphi_{\mathrm{Pr}_{\infty}(\mathsf{P})}:\overline{J}^{\infty}(P)\rightarrow \overline{J}^{\infty}(P_1).
\end{equation}
Morphism (\ref{eqn: InfiniteAhom}) is naturally interpreted as an infinite prolongation of differential operator in total derivatives.

There is a natural way to lift $\D$-linear constructions on the tautological $\D$-scheme $X,$ to the universal $\D$-scheme $J_X^{\infty}(E)$ associated with a bundle. In particular, a section $P\in\Gamma(U,\mathcal{D}_X),$ has a lift $\widehat{P}\in \Gamma(U,\mathcal{A}[\mathcal{D}_X]),$ obtained in coordinates, roughly speaking, by replacing partial derivatives by total derivatives.

\begin{prop}
    \label{prop: Lifts}
Consider two locally free sheaves of $\mathcal{O}_X$-modules $\mathcal{E},\mathcal{F}.$ There is an isomorphism
$\gamma:\mathcal{A}_X^{\ell}\otimes_{\mathcal{O}_X}\mathcal{D}_X(\mathcal{E},\mathcal{F})\simeq \mathcal{A}^{\ell}[\mathcal{D}](p_{\infty}^*\mathcal{E},p_{\infty}^*\mathcal{F}),$
defined by 
$(f\otimes\mathsf{Q})(e):=f\cdot \big(\mathsf{Q}(e)\circ p_{\infty}\big),e\in \mathcal{E}.$
\end{prop}
This fact is well-known to experts, but perhaps not in this context. We only comment on the isomorphism. Namely, 
$$\gamma:\big(f(x,u_{\sigma}^{\alpha})\otimes\mathsf{Q}\big)\mapsto (f(x,u_{\sigma}^{\alpha})\cdot_{\mathcal{A}} \hat{\mathcal{Q}}),$$ where the corresponding lift $\mathsf{Q}\in \mathcal{D}_X(\mathcal{E},\mathcal{F})$ is the operator
$\hat{\mathsf{Q}}:p_{\infty}^*(\mathcal{E})\rightarrow p_{\infty}^*(\mathcal{F}),$ defined by $(j_{\infty}\varphi)^*\big(\hat{\mathsf{Q}}s_{\infty}\big)=\mathsf{Q}\big((j_{\infty}\varphi)^*s_{\infty}\big),$
where $s_{\infty}\in \mathrm{Sect}(J_X^{\infty}E,p_{\infty}^*\mathcal{E})$ and $\varphi\in \mathrm{Sect}(X,E).$
Proposition \ref{prop: Lifts} induces a subcategory $\mathrm{Mod}_{ind}(\A[\D])$ of $\mathrm{Mod}(\A[\D])$, spanned by lifted/induced sheaves. An $\A[\D]$-module is \emph{induced} if isomorphic to $p_{\infty}^*F\otimes \mathcal{A}[\D]$ for a coherent $\mathcal{O}_X$-module $F$. 

For induced modules (\ref{eqn: InfiniteAhom}) is a morphism of vector $\D$-schemes over $\mathrm{Spec}_{\mathcal{D}}(\A)$, where we recall a vector $\mathcal{A}[\D]$-bundle is a locally finitely generated $\mathcal{A}[\D]$-module of finite $\mathcal{A}[\D]$-rank \cite{KSY2}.
\begin{defn}
\normalfont
The \emph{$\mathcal{A}[\D]$-rank} of an induced $\mathcal{A}[\D]$-module $\M\simeq ind_{\mathcal{A}[\D]}(\mathcal{E})$, where $\mathcal{E}\in \mathrm{QCAlg}(X),$ is the ordinary rank of the underlying sheaf 
$rank_{\mathcal{A}[\D]}(\M):=rank(\mathcal{E}).$
\end{defn}
For example, if $E\rightarrow X$ is a vector bundle of rank $m$, considering $\Omega_{\A}^1,$ 
then $rank_{\mathcal{A}[\D]}(\Omega_{\A}^1)=m.$




\subsubsection{The horizontal Spencer-Jet sequence}
Fix a global section $s\in H^0(X,E)$ and let $s^*:H_{dR}(X)\rightarrow H_{dR}(E),$ denote the associated monomorphism. A canonical choice is the zero-section and by composition with pullback along the projection $p_{\infty}$, we obtain a monomorphism into the de Rham cohomology of the $\mathcal{D}_X$-scheme $Z^{\infty}:=\{Z_k\}\rightarrow X.$ 

Horizontal cohomologies of $Z^{\infty}$ are given by the image of the de Rham complex of the base via $p_{\infty}^*,$ thus we write $p_{\infty}^*\Omega_X^{\bullet}$ to indicate the horizontal de Rham algebra of $Z$.

Let $\M\in \mathrm{Mod}(\mathcal{A}[\D])$ and consider the diagram:
\begin{equation}
\label{eqn: Diagram}
\adjustbox{scale=.82}{
\begin{tikzcd}
& 0\arrow[d] & 0\arrow[d] & 0\arrow[d] 
\\
0\arrow[r] & \Sym^k(p_{\infty}^*\Omega_X^1)\otimes_{\A}\M \arrow[d] \arrow[r]& \overline{\mathcal{J}}^k(\M)\arrow[d]\arrow[r] & \overline{\mathcal{J}}^{k-1}(\M)\arrow[d]\rightarrow 0
\\
0\arrow[r] & p_{\infty}^*\Omega_X^1\otimes_{\A}\Sym^{k-1}(p_{\infty}^*\Omega_X^1)\otimes_{\A}\M\arrow[d]\arrow[r] & p_{\infty}^*\Omega_X^1\otimes_{\A}\overline{\mathcal{J}}^{k-1}(\M)\arrow[d]\arrow[r] & p_{\infty}^*\Omega_X^1\otimes_{\A}\overline{\mathcal{J}}^{k-2}(\M)\arrow[d] \rightarrow 0
\\
& \vdots \arrow[d] & \vdots \arrow[d] & \vdots \arrow[d]
\\
0\arrow[r]& \omega_{\A}\otimes_{\A}\Sym^{k-d_X}(p_{\infty}^*\Omega_X^1)\otimes_{\A}\M\arrow[r] & \omega_{\A}\otimes_{\A}\overline{\mathcal{J}}^{k-d_X}(\M)\arrow[r] & \cdots
\end{tikzcd}}
\end{equation}
Diagram (\ref{eqn: Diagram}) gives a 
short-exact sequence of dg-$\mathcal{A}[\D]$-modules
\begin{equation}
\label{eqn: SES SpJet}
0\rightarrow \mathsf{Sp}_{\D}(\M)^{\bullet}\rightarrow \mathrm{DR}_{\D}\big(\A;\overline{\mathcal{J}}^{k-\bullet}(\M)\big)^{\bullet}\rightarrow \mathrm{DR}_{\mathcal{D}}(\A;\overline{\mathcal{J}}^{k-1-\bullet}(\M)\big)^{\bullet}\rightarrow 0.
\end{equation}
Namely, the horizontal $\delta$-Spencer complex in the left-most column is the horizontal \emph{co-symbol} complex (of order $k$).
Explicitly, $\mathsf{Sp}_{\D}^r(\M)^{\bullet}$ is
\begin{eqnarray}
    \label{eqn: k-horizontal delta jet}
0 &\rightarrow& \Sym^r(p_{\infty}^*\Omega_X^1)\otimes_{\A}\M\xrightarrow{\overline{\delta}}\Sym^{r-1}(p_{\infty}^*\Omega_X^1)\otimes_{\A}p_{\infty}^*\Omega_X^1\otimes_{\A}\M \rightarrow \nonumber
\\
&\rightarrow& \Sym^{r-2}(p_{\infty}^*\Omega_X^1)\otimes_{\A}\wedge^2p_{\infty}^*\Omega_X^1\otimes_{\A}\M\rightarrow\cdots\rightarrow \nonumber
\\
&\rightarrow& \Sym^{r-d_X}(p_{\infty}^*\Omega_X^1)\otimes_{\A}\wedge^{d_X}p_{\infty}^*\Omega_X^1\otimes_{\A}\M\rightarrow 0.
\end{eqnarray}

Involutivity for $\D_{\A}$-morphisms e.g. horizontal differential operators, is defined as in Definition \ref{defn: Involutive operator}. In this case, replace the sheaves $\Omega_X^i$ with $p_{\infty}^*\Omega_X^i,i=1,\ldots, dimX$ and geometric symbols (\ref{eqn: Symbol}) with $\overline{\mathcal{N}},$ defined as follows. 
Start with some horizontal operator $P\in \mathcal{A}[\D](\M_0,\M_1),$ of order $\leq k$. There is a canonical way to construct (via induction) a compatibility complex for it. Details are given in Appendix \ref{sec: Compatibility complex}.
Then $\overline{\mathcal{N}}^{k+\ell}(P):=Ker\big(\sigma(P)),$ and  
$$\sigma(P):\Sym^{k+\ell}(p_{\infty}^*\Omega_X^1)\otimes_{\A}\M_0\rightarrow \Sym^{\ell}(p_{\infty}^*\Omega_X^1)\otimes_\A \M_1.$$
Repeating the above construction of geometric symbols (\ref{eqn: Symbol}) gives a collection of finitely-presented $\mathcal{A}[\D]$-submodules,
$\overline{\mathcal{N}}^{k+\ell}(P)\hookrightarrow \Sym^{k+\ell}(p_{\infty}^*\Omega_X^1)\otimes_{\A}\M_0.$
There exists sub-complexes, obtained from (\ref{eqn: k-horizontal delta jet}) via restriction:
\begin{equation}
    \label{eqn: AD-Spencer}
    0\rightarrow \overline{\mathcal{N}}^{k+\ell}(P)\xrightarrow{\overline{\delta}}p_{\infty}^*\Omega_X^1\otimes_{\A}\overline{\mathcal{N}}^{k+\ell-1}(P)\xrightarrow{\overline{\delta}}p_{\infty}^*\Omega_X^2\otimes_{\A}\overline{\mathcal{N}}^{k+\ell-2}(P)\rightarrow \cdots.
\end{equation}
    The cohomology at the $i$-th term of (\ref{eqn: AD-Spencer}), denoted by 
    $\overline{h}^{k+\ell,i}(P):=H\big(p_{\infty}^*\Omega_X^i\otimes_{\A}\overline{\mathcal{N}}^{k+\ell-i}(P);\overline{\delta}\big),$
    is the \emph{$\mathcal{A}[\D]$-Spencer $\overline{\delta}$-cohomology of $P$}.
Following standard terminology, the $\mathcal{A}[\D]$-operator $P$ is \emph{involutive} if 
\begin{equation}
    \label{eqn: Involutive}
    \overline{h}^{k+\ell,i}(P)=0, i\geq 0.
    \end{equation}
The \emph{degree of involution} for $\mathcal{B}$, (non-linear case) is
$$\mathrm{iDeg}_{\D}(\mathcal{B}):=inf_{q_0\in \mathbb{N}}\big\{\overline{h}^{p,q}(\overline{\mathcal{N}})=0,\forall q\geq q_0,0\leq p\leq dim(X)\}.$$
We say $\overline{\mathcal{N}}$ is $p_0$-acyclic at $q_0$ for $0\leq p_0\leq dimX$ if
$$\overline{h}^{p,q}(\mathcal{N})=0,\forall q\geq q_0,0\leq p\leq p_0.$$
The $\overline{\delta}$-Poincare lemma (see e.g. \cite{G,Ts,KV}) states for any $P\in \mathrm{Fil}^k\mathcal{A}[\D](\M_0,\M_1),$ there exists $\ell_0\equiv \ell_0\big(rank_{\mathcal{A}[\D]}(\M_0),dimX,ord(P)\big)$ such that $\overline{h}^{k+\ell,i}(P)=0$ for $\ell\geq \ell_0$ for every $i\geq 0.$
In the case when $\M_0,\M_1$ are induced modules with underlying coherent sheaves $\mathcal{E}_0,\mathcal{E}_1,$ of ranks $m_0,m_1,$ for a $k$-th order operator $P$, the integer $\ell_0$ depends only on $\ell_0(m_0,d_X,k).$ 

We give one illustrative example.

\begin{ex}
    Consider a system in $N$-equations of $m$-unknown functions between coherent sheaves $\mathcal{E},\mathcal{F}$ $\mathcal{F}.$ Compatibility complexes for non-linear PDEs are defined for each linearization on a
jet-solution $j_{\infty}(s).$ Suppose $F=(F_i^1,\ldots,F_i^m)$ is a locally complete intersection (as ordinary maps between finite rank bundles). Then,
\begin{eqnarray*}
0&\rightarrow& \mathrm{ker}(TF)\rightarrow H^0(\EQ,\mathcal{M}_1)\rightarrow H^0(\EQ,\mathcal{M}_2)\rightarrow H^0(\EQ,\wedge^{m+1}p_{\infty}^*\mathcal{F})\rightarrow  
\\
&\rightarrow& H^0(\EQ,p_{\infty}^*\mathcal{E}^*\otimes\wedge^{m+2}p_{\infty}^*\mathcal{F})\rightarrow H^0(\EQ,\mathcal{S}ym_{\mathcal{A}}^2(p_{\infty}^*\mathcal{E}^*)\otimes\wedge^{m+3}p_{\infty}^*\mathcal{F})\rightarrow\cdots.
\end{eqnarray*}
\end{ex}

\noindent\textbf{Pure modules.}
Following \cite{K}, consider sheaves of $\mathcal{A}[\mathcal{D}]$-modules $\mathcal{M}$ (codimension $r$), for which the natural morphism 
\begin{equation}
    \label{eqn: r purity for AD}
0\rightarrow \mathcal{M}\hookrightarrow \mathcal{E}xt_{\mathcal{A}[\mathcal{D}]}^r\big(\mathcal{E}xt_{\mathcal{A}[\mathcal{D}]}^r(\mathcal{M},\mathcal{A}[\mathcal{D}]),\mathcal{A}[\mathcal{D}])\big),
\end{equation}
is an injective morphism of left $\mathcal{A}\otimes \mathcal{D}$-modules.
\begin{prop}
Suppose that $\mathcal{M}$ satisfies (\ref{eqn: r purity for AD}) for some commutative $\D$-algebra $\mathcal{A}.$ In particular, suppose $\mathcal{A}$ is $\mathcal{O}_X.$ Then $\mathcal{M}$ is $r$-pure as a $\D$-module.
\end{prop}
\begin{proof}
For the initial commutative $\D$-algebra $\mathcal{A}=\mathcal{O}_X,$ one has the local Verdier duality identifying with the usual derived $\D$-module duality and $r$-purity in this context is the standard one for a left $\D$-module,
$0\rightarrow \mathcal{M}\hookrightarrow \mathcal{E}xt_{\mathcal{D}}^r\big(\mathcal{E}xt_{\mathcal{D}}^r(\mathcal{M},\mathcal{D}),\mathcal{D}\big).$
It is equivalently observed by going through the usual setting of obtaining pure quotients via exact sequences of $\D$-modules,
$$0\rightarrow \mathcal{T}_r(\mathcal{M})\hookrightarrow \mathcal{T}_{r-1}(\mathcal{M})\hookrightarrow \mathcal{E}xt_{\mathcal{D}}^r\big(\mathcal{E}xt_{\mathcal{D}}^r(\mathcal{M},\mathcal{D}),\mathcal{D}\big).$$
Here $\mathcal{M}_r\rightarrow \mathcal{M}$ submodule with $\mathcal{T}(\mathcal{M}_r)\rightarrow \mathcal{T}(\mathcal{M})$ and the image of this map is $\mathcal{T}_r(\mathcal{M}).$
Thus, we have
$0\subseteq \mathcal{T}_n(\mathcal{M})\subseteq \mathcal{T}_{n-1}(\mathcal{M})\subseteq \cdots\subseteq \mathcal{T}_0(\mathcal{M})\subseteq \mathcal{M},$
such that each successive quotient 
$\mathcal{T}_{r-1}(\mathcal{M})/\mathcal{T}_r(\mathcal{M})$ is a $\D$-module of dimension $n-r$ i.e. we have proven that 
$\mathrm{codim}\big(\mathrm{Char}(\mathcal{T}_{r-1}(\mathcal{M})/\mathcal{T}_r(\mathcal{M})\big)=r.$

\end{proof}
In particular, a $\D$-module $\mathcal{M}$ is $0$-pure if and only if it is torsion free and there exists
$0\rightarrow \mathcal{T}(\mathcal{M})\hookrightarrow \mathcal{M}\xrightarrow{\epsilon}\mathcal{H}om_{\mathcal{D}}\big(\mathcal{H}om_{\mathcal{D}}(\mathcal{M},\mathcal{D}),\mathcal{D}\big),$
with $\epsilon(m)(f):=f(m),$ the evaluation.
\begin{rmk}
    The property of $\D$-freeness implies $\D$-projectivity which in turn implies $\D$-torsion freeness.
    \end{rmk}

\subsection{Spencer-regularity}
A cohomological vanishing criteria for PDEs is now introduced which is compatible with Definition \ref{defn: D-geom m-involutive}. Given a $\D$-subscheme $Z\subset \mathrm{Spec}_{\D}(\A)$, defined by a $\D$-stable ideal $\I$, we will always now write $\mathrm{gr}(\mathcal{I})$ for the symbolic module, to emphasize it is determined via $\mathcal{I}.$ See the following remark.

\begin{rmk}
\label{gr(I)}
Consider a $\D$-subscheme $Z\subset \mathrm{Spec}_{\D}(\A)$ and consider the standard filtrations (\ref{defn: Standard filt}). Geometrically (c.f. \ref{eqn: Cosymbol iso}), morphisms $J_X^{\ell
+1}E\rightarrow J_X^{\ell}E$ are affine with associated vector bundle, 
$(p_{\ell}^0)^*Sym^{\ell}(T_X^*)\otimes(p_{\ell}^1)^*T_{E/X},$ with $p_{\ell}^k:J_X^kE\rightarrow J_x^{\ell}E,$ where $J_X^0E=E.$ Dualizing, gives (\ref{eqn: ChKahler1}).
Thus, set $\mathrm{gr}(\A)\simeq Sym_{\mathcal{O}_X}(\Theta_X)\otimes_{\mathcal{O}_X}\mathcal{E},$ with $\mathcal{E}$ the sheaf associated with $E.$ Then $\mathrm{gr}(\I):=\oplus_q F^q\mathcal{I}/F^{q-1}\mathcal{I}\subset \mathrm{gr}(\A).$ We have $\mathrm{gr}_q(\I)\subset Sym^q(\Theta_X)\otimes \mathcal{E}.$
\end{rmk}

Denote by $\mathcal{H}_{\delta}^{p,q}(\mathrm{gr}(\mathcal{I}))$ the $\delta$-Spencer cohomologies.
\begin{defn}
\label{defn: D-geometric generalized involutivity}
\normalfont 
Let $Z$ be formally integrable and involutive. Its \emph{degree of involution} is 
    $\mathrm{reg}_{Sp}(\mathcal{I}):=\underset{q_0\in \mathbb{N}}{\mathrm{inf}}\big\{H_{\delta}^{p,q}(\mathrm{gr}(\mathcal{I}))=0,\text{ for all } 0\leq p\leq \mathrm{dim}X, q\geq q_0\big\}.$
Module $\mathrm{gr}(\mathcal{I})$ is said to be \emph{$p_0$-acyclic at $q_0$} for $0\leq p_0\leq \mathrm{dim}X$ if:
 $H_{\delta}^{p,q}(\mathrm{gr}(\mathcal{I}))=0,\hspace{2mm} \forall q\geq q_0, 0\leq p\leq p_0.$
\end{defn}
We use the notation $\mathrm{reg}_{Sp}$ due to the following result.
\begin{prop}
\label{prop: Unifying}
    Suppose that $\mathcal{I}$ is formally integrable and $\D$-involutive. Then its degree of involution, given by Definition \ref{defn: D-geometric generalized involutivity} agrees with the Spencer-regularity (\ref{defn: D-Geometric Regularity}), and thus agrees with the Castelnuovo-Mumford regularity of its (linearized, in the non-linear setting) symbolic $\D$-module. 
\end{prop}
The proof makes use of tools not introduced above e.g. Pommaret basis. See \cite{Ma} for details.
\begin{proof}
The Koszul homology of $\mathcal{M}$ is equivalent to its minimal free resolution, if $\mathrm{reg}_{CM}(\mathcal{I})=m,$ meaning $Syz^i(\mathcal{I})$ is generated in degrees $\leq m+i,$ then 
$H_{p,q}(\mathcal{I})=0,\forall q>\mathrm{reg}_{CM}(\mathcal{I})=m.$
If one has a local involutive basis $\mathcal{H}_q$ for Pommaret division and $\mathcal{I}_{k+1}$ are the homogeneous degree $k+1$ elements, then
$\mathrm{dim}(\mathcal{I}_{k+1})\geq \sum_{\ell=1}^{\mathrm{dim}(X)}\ell\cdot \beta_k^{(\ell)}.$
It is an equality if and only if it gives rise to a Pommaret basis. The integers $\beta_k^{(\ell)}$ are the number of elements in the basis with leading term of class $\ell.$
Let $H_{p,q}$ denote the Koszul homology, then one has
$H_{p,q}(\mathcal{I})=0\forall r>\mathrm{reg}_{CM}(\mathcal{I})=m \Leftrightarrow H_{p,q}(\mathcal{M})=0,\forall q\geq \mathrm{reg}_{CM}(\mathcal{I}).$
Therefore, $\mathrm{reg}_{CM}=\mathrm{reg}_{Sp}(\mathcal{I}).$ 
If $\mathcal{H}_k$ is the basis of $\mathcal{I}_k$ whose generators have different leading terms, we may choose as representatives of a basis of $\mathcal{M}_k$ polynomials whose leading terms do not appear in the leading term set of $\mathcal{H}_k.$ Thus, if $\mathcal{H}_k$ contains $\beta_k^{(\ell)},$ elements with leading term of class $\ell$, the basis of $\mathcal{M}_k$ contains
$\alpha_k^{(\ell)}=m\cdot C_{k-1}^{k+n-\ell-1}-\beta_k^{(\ell)}$,
representatives.
\end{proof}
We also have the following obvious statement.
\begin{prop}
    If $Z$ is Spencer $m$-regular, then it is Spencer $m'$-regular for all $m'\geq m.$
\end{prop}

Giving a sequence of $\D$-ideal sheaves in the same ambient $\D$-algebra of jets gives a sequence of corresponding characteristic modules.

\begin{prop}
\label{prop: Involutivity sequences}
    Consider smooth morphisms of $\D$-schemes, $\mathcal{Y}'\rightarrow \mathcal{Y}\rightarrow \mathcal{Y}''$ defined by $\D$-ideal sheaves $\I',\I,\I'',$ respectively, which induce a short exact sequence of symbolic modules. Suppose each $\D$-scheme is Spencer regular. The following statements hold:
    \begin{enumerate}
    \item Suppose that $\mathrm{reg}_{Sp}(\I')=\mathrm{reg}_{Sp}(\I'')=m$.Then $\mathrm{reg}_{Sp}(\I)=m;$

    \item Suppose that $\mathrm{reg}_{Sp}(\I')=(m+1)$ and $\mathrm{reg}_{Sp}(\I)=m.$ Then $\mathrm{reg}_{Sp}(\I'')=(m+1);$
    
    \item Suppose that $\mathrm{reg}_{Sp}(\I)=m$ and $\mathrm{reg}_{SP}(\I'')=(m-1).$ Then $\mathrm{reg}_{Sp}(\I')=m.$
    \end{enumerate}
\end{prop}
\begin{proof}
This is proven using long-exact sequence and standard homological arguments. In some detail, consider the relative cotangent sequence associated to the morphism of $\D$-schemes $\EQ'\rightarrow\EQ\rightarrow \EQ''$ and the corresponding long-exact sequence in cohomology applied to the Spencer sequences of these cotangent sheaves e.g.
$$\alpha^*\Omega_{\EQ/\EQ^{''}}^1\rightarrow \Omega_{\EQ^{'}/\EQ^{''}}^1\rightarrow \Omega_{\EQ^{'}/\EQ}^1\rightarrow 0,$$
or dually, denoting the corresponding commutative $\mathcal{D}_X$-algebras by $\mathcal{A},\mathcal{A}',\mathcal{A}^{''},$ 
we have a sequence $\mathcal{A}^{'}\otimes_{\mathcal{O}_X}\mathcal{D}_X$-modules,
$$\mathcal{A}^{'}\otimes_{\mathcal{A}}\Omega_{\mathcal{A}/\mathcal{A}^{''}}^1\rightarrow \Omega_{\mathcal{A}'/\mathcal{A}''}^1\rightarrow\Omega_{\mathcal{A}'/\mathcal{A}}^1\rightarrow 0.$$
Using the exact sequences,
$p_{\infty}^*\Omega_{E/X}^1\rightarrow \Omega_{\EQ/X}^1\rightarrow \Omega_{\EQ/E}^1\rightarrow 0,$
for each structure map $p_{\infty}:\EQ\rightarrow X,p_{\infty}^{'}:\EQ^{'}\rightarrow X,p_{\infty}^{''}:\EQ^{''}\rightarrow X,$ together with the sequences (\ref{eqn: Filtered Kahler sequences}), one establishes the claim by passing to the corresponding long-exact sequence. Simply denoting the induced sequence of graded modules $\mathrm{gr}(\I')\rightarrow \mathrm{gr}(\I)\rightarrow \mathrm{gr}(\I''),$ a segment of the LES is 
$$\cdots\rightarrow \mathcal{H}_{Sp}^{p,q}\big(\mathrm{gr}(\I'')\big)\rightarrow \mathcal{H}_{Sp}^{p,q}\big(\mathrm{gr}(\I)\big)\rightarrow \mathcal{H}_{Sp}^{p,q}\big(\mathrm{gr}(\I')\big)\rightarrow \cdots.$$
For (i), since $\mathcal{H}_{Sp}^{p,q}\big(\mathrm{gr}(\I')\big)=\mathcal{H}_{Sp}^{p,q}\big(\mathrm{gr}(\I'')\big)=0,q\geq m$ and for all $p$, the claim is obvious. Claim (ii) is proven similarly, and (iii) reduces to a sequence of the form 
$$0\rightarrow \mathcal{H}_{Sp}^{p,q}\big(\mathrm{gr}(\I')\big)\rightarrow 0,$$
precisely at $q=m.$
\end{proof}

\section{Moduli of $\D_X$-ideal sheaves}
\label{sec: Moduli Space of D-Ideal Sheaves}
This section establishes the construction of the $\D$-geometric analog of the Hilbert and Quot functors that parameterize ideal sheaves in $\D$-geometry with Spencer-regular behaviour. We discuss the numerical characterizations of these sheaves.

\subsection{$\D$-bundles, projectivity and flatness}
Let $X$ be a $\D$-affine $k$-scheme and when necessary, we will assume it is quasi-projective. 
If $\EQ$ is a $\D$-scheme over $X$ the projectivity of an $\mathcal{O}_{\mathcal{Y}}[\mathcal{D}_X]$-module is local property for the $\D$-Zariski topology. 
We thus have a notion of $\mathcal{Y}$-locally
projective $\mathcal{O}_{\mathcal{Y}}[\mathcal{D}_X]$-module sheaves on $X$.

The property of being finitely generated is étale local and defines a notion of $\mathcal{Y}$-locally finitely generated and projective $\mathcal{O}_{\mathcal{Y}}[\mathcal{D}_X]$-modules. We call them \emph{vector $\D$-bundles over $\mathcal{Y}.$} As an example, any induced $\mathcal{O}_{\EQ}[\mathcal{D}_X]$-module with underlying locally free finitely generated $\mathcal{O}_X$-module is a vector $\D$-bundle on $\EQ.$ Every such object has a dual that is again a vector $\D$-bundle and the canonical pairing between them is non-degenerate. 

\begin{rmk}

\label{rmk: Ranks and Chern class}
   The non-commutative ring $\D$ has a skew field of fractions thus for locally projective $\D$-modules it makes sense to consider their ranks. Therefore a vector $\D$-bundle over a smooth $\D$-scheme $Z$ (e.g. jets) has a well-defined rank. By \cite[Prop 2.20]{KSY2}, there is a well-defined sequence of characteristic classes. In particular, we may fix the $\D$-geometric Chern class of a $\D$-bundle.
\end{rmk}
We tacitly make use of a $\D$-analog of the Raynaud-Gruson flattening theorem \cite[Lemma 2.3.8]{BD}.

Given a $\D$-algebraic space $\mathcal{Y}$ for each test $\D$-algebra $\mathcal{R}$, set $\mathcal{Y}_{\mathcal{R}}:=\mathcal{Y}\times_X\mathrm{Spec}(\mathcal{R}),$ and for each $\mathcal{M}\in \mathrm{Mod}_{\mathcal{D}}(\mathcal{Y}),$ set $\mathcal{M}_{\mathcal{R}}:=\mathcal{M}\otimes \mathcal{R}\in \mathrm{Mod}_{\mathcal{D}}(\mathcal{Y}_\mathcal{R}).$
Similarly, for $X\rightarrow S$ a quasi-projective and proper $S$-scheme with sheaf of relative differential operators $\mathcal{D}_{X/S},$ for any $s\in S$ denote by $i_s:X\times\{s\}\rightarrow X\times S$ the inclusion and set
$$i_{s}^*:=p_{X}^{-1}(\mathcal{O}_S/\mathfrak{m}_s)\otimes_{p_X^{-1}\mathcal{O}_S}(-):\mathbf{D}_{f.p}^b(\mathcal{D}_{X\times S/S})\rightarrow \mathbf{D}^b(\mathcal{D}_X),$$
where $p_X:X\times S\rightarrow S$ and $\mathfrak{m}_s$ is maximal ideal of functions vanishing at $s.$

\begin{prop}
    \label{prop: RelJets BC}
  Let $\pi:X\rightarrow S$ be a smooth and proper morphism of relative dimension $d_{X/S}$ with $\mathcal{E}$ a locally free finite rank $\mathcal{O}_X$-module. 
  Letting $\Delta\subset X\times_S X$ with projections $p_1,p_2$ one has an equivalence
  $$R^ip_{1*}p_2^*\mathcal{E}\simeq \pi^*R^i\pi_*\mathcal{E},\forall i\geq 0,$$ 
  and an isomorphism of $\mathcal{O}_X$-modules $p_{1*}p_2^*\mathcal{E}\simeq \pi^*\pi_*\mathcal{E}\otimes\mathcal{O}_X.$ 
\end{prop}
\begin{proof}
    Using for each $k\geq 0$ the exact sequence of sheaves of $\mathcal{O}_X$-modules,
    $$0\rightarrow \mathcal{I}_{\Delta}^{k}\rightarrow \mathcal{O}_{X\times_SX}\rightarrow \mathcal{O}_{X\times_S X}/\mathcal{I}_{\Delta}^{k+1}\rightarrow 0,$$
    tensoring with $\pi_2^*\mathcal{E}$, taking $p_{1*}$, one has the long-exact sequence in cohomology
    $$0\rightarrow p_{1*}(\mathcal{I}_{\Delta}^{k}\otimes p_2^*\mathcal{E})\rightarrow p_{1*}p_2^*\mathcal{E}\rightarrow p_{1*}(\mathcal{O}_{X\times_S X}/\mathcal{I}_{\Delta}^{k}\otimes p_2^*\mathcal{E})\rightarrow$$
    $$\rightarrow R^1p_{1*}(\mathcal{I}_{\Delta}^{k}\otimes p_{2}^*\mathcal{E})\rightarrow R^1p_{1*}p_2^*\mathcal{E}\rightarrow R^1p_{1*}(\mathcal{O}_{X\times_S X}/\mathcal{I}_{\Delta}^k\otimes p_2^*\mathcal{E})\rightarrow\cdots $$
    from which the flat-base change gives the result.
\end{proof}

\begin{cor}
    If $X$ is projective, set $\mathcal{O}_X(d)=i^*\mathcal{O}(1)^{\otimes d}$ where $i$ is the Plucker embedding of $X$ into projective space. Then, $J_{X/S}^k(\mathcal{O}_{X}(d))$ has the property that 
    $R^1p_{1*}(\mathcal{O}_{X\times X}/\mathcal{I}_{\Delta}^k\otimes p_2^*\mathcal{O}_X(d))=0.$
\end{cor}
Let us record the functoriality of sheaves of relative jets, for convenience. These results are well-known.

\begin{prop}
\label{prop: pb Jets}
Let $f:X\rightarrow Y$ be a smooth morphisms of $S$-schemes and $\mathcal{E}$ a finite rank locally free $\mathcal{O}_Y$-module. Then, there is a commutative diagram of locally free $\mathcal{O}_X$-modules, 
\[
\begin{tikzcd}
    0\arrow[r] & \mathrm{Sym}(f^*\Omega_{Y/S}^1)\otimes f^*\mathcal{E}\arrow[d]\arrow[r] & f^*J_{Y/S}^k(\mathcal{E})\arrow[d] \arrow[r] & f^*J_{Y/S}^{k-1}(\mathcal{E})\arrow[d]\arrow[r] & 0
    \\
0\arrow[r] & \mathrm{Sym}^k(\Omega_{X/S}^1)\otimes f^*\mathcal{E}\arrow[r] & J_{X/S}^k(f^*\mathcal{E})\arrow[r] & J_{X/S}^{k-1}(f^*\mathcal{E})\arrow[r] & 0
\end{tikzcd}
\]
If $V\subset Y$ is an open embedding, there is an equivalence $j_{V}^{-1}J_{Y/S}^k(\mathcal{E})\simeq J_{V/S}^k(j^*\mathcal{E})$ of $\mathcal{O}_V$-modules.
\end{prop}
The following collects the structural aspects of the sheaf of relative $k$-jets for a projective variety.
\begin{prop}
    Suppose that $X\subseteq \mathbb{P}_S^n,$ and let $\mathcal{E}$ be a finite rank locally free sheaf. Then there is a universal $k$-th order differential operator $\pi^*\pi_*\mathcal{E}\otimes\mathcal{O}_X\rightarrow J_{X/S}^k(\mathcal{E}),$ that is equivalent to the morphism,
$j_k:H^0(X,\mathcal{E})\otimes\mathcal{O}_X\rightarrow J_{X/S}^k(\mathcal{E}),$ which maps $s\in H^0(X,\mathcal{E})$ and a $k$-rational point $x_0\in X$ to the $k$-th order Taylor expansion of $s$ around $x_0.$
\end{prop}

\subsubsection*{Good filtrations for $\D_{\A}$-modules}
An analog of a good filtration (see e.g. \cite{K,Sch,Sa}) for $\mathcal{A}[\D]$-modules $\M$ of finite presentation is now given. It is compatible with both the order filtration on $\D_X$ and with the jet-filtration $\{F^k\A\}_{k\geq 0}$ on $\A,$ via the $\A$-module multiplication $\A\otimes \M\rightarrow \M.$ 
\begin{rmk}
It may be conveniently formulated via the standard order filtration on $\mathcal{A}[\D]$ given in Proposition \ref{prop: Char inclusions}. However, since $\D$-spaces are not  $\mathcal{O}_X$-finitely generated in general, and instead only $\D$-finitely generated or $\D$-finitely presented when $Z$ is formally integrable and $m$-regular for some $m\geq 0,$ we need a new criteria for finiteness of the filtered pieces $\mathrm{Fl}_j\M$ of a good filtration (in the $\D$-module setting, it is coherency).
\end{rmk}

\begin{defn}
\normalfont 
A \emph{good $\mathcal{A}[\D]$-filtration} on an $\mathcal{A}[\D]$-module $\M$ of finite-presentation, is a filtration on $\M$, denoted
$\mathrm{Fl}\M:=\{\mathrm{Fl}_m\mathcal{M}|m\in \mathbb{Z}\},$ by $\mathcal{A}[\D]$-submodules $\mathrm{Fl}_m\M$ satisfying:
\begin{equation}
\label{eqn: Goodfil}
    \begin{cases}
    \text{\'Etale-locally on }\EQ, \mathrm{Fl}_j\M=0,j \ll 0,
    \\
        \mathrm{Fl}_j\M \text{ is a vector }\D_X-\text{bundle},
        \\
        \mathrm{Fl}_k^{ord}\mathcal{O}_{\EQ}[\D_X]\cdot \mathrm{Fl}_j\M=\mathrm{Fl}_{k+j}\M,\hspace{1mm} j\gg0,k\geq 0.
    \end{cases}
\end{equation}
\end{defn}
The first condition in (\ref{eqn: Goodfil}) implies also $\mathrm{Fl}_j\M=0,j\ll 0,$ locally on $X.$ Moreover, for the initial object $\A=\mathcal{O}_X$ in $\mathrm{CAlg}_X(\D_X),$ a good $\mathcal{O}_X[\D_X]$-filtration is a good filtration as usual.
\begin{prop}
\label{prop: Associated graded}
Consider an $\mathcal{A}[\D]$-module $\M$ with a good $\mathcal{A}[\D]$-filtration. Then its associated graded module is an $\mathcal{O}_{T^*(X,\EQ)}$-module of finite-presentation.
\end{prop}
\begin{proof}
Given such an object $\mathrm{Fl}\M$, its associated graded module 
$\mathrm{Gr}^F(\M)$ is naturally a $\mathrm{Gr}^{F_{ord}}\mathcal{O}_{\EQ}[\D_X]$-module. Then, by Proposition \ref{prop: Char inclusions} it is an $\mathcal{O}_{T^*(X,\EQ)}$-module. Finite-presentation is clear from (\ref{eqn: Goodfil}).
\end{proof}
This result indicates by considering the derived functors of internal solutions (c.f. \cite{KSY}), namely $\mathbf{R}\mathcal{H}om_{\A[\D]}(-,\A[\D])$, applied to finitely-presented modules with good filtrations, their cohomologies are isomorphic to certain natural subquotients.
\begin{prop}
\label{prop: Approximations}
Let $Z$ be a $\D$-smooth scheme and consider a $\D$-subscheme $\mathrm{Spec}_{\mathcal{D}}(\mathcal{B})\hookrightarrow Z.$ Set $\mathcal{O}_Z=\mathcal{A}.$ Suppose that $\mathcal{M}^{\bullet}$ is a dg-$\mathcal{A}[\mathcal{D}]$-module that is quasi-isomorphic to a bounded complex of finitely generated $\mathcal{A}[\mathcal{D}_X]$-modules of finite rank. Then, there exists a good filtration such that the associated graded sheaves  
$\mathrm{Gr}^F\mathcal{E}xt_{\mathcal{O}_{\EQ}[\mathcal{D}]}^i\big(\mathcal{M}^{\bullet},\mathcal{O}_{\EQ}[\mathcal{D}]),$ are isomorphic to a sub-quotient of the sheaf
$$\mathcal{E}\mathrm{xt}_{\mathrm{Gr}^F\mathcal{O}[\mathcal{D}]}^i\big(\mathrm{Gr}^F\mathcal{M}^{\bullet},\mathcal{O}_{T^*(X,Z)}),$$
for every $i.$
\end{prop}
 
\begin{proof}
By assumptions write $\mathcal{M}^{\bullet}$ locally as a complex
$0\leftarrow \mathcal{M}\leftarrow \mathcal{V}_0\leftarrow \mathcal{V}_1\leftarrow \dots,$
with each $\mathcal{V}_i$ finitely generated and free $\mathcal{A}[\mathcal{D}]$-modules of finite rank. Locally, we may write $\mathcal{V}_i=\bigoplus_k\mathcal{A}[\mathcal{D}]\bullet \gamma_k^i$ on generators $\gamma=\{\gamma_k^i\}.$ 
Fix integers $\{c_{jk}\in \mathbb{Z}\}$ and define the desire good filtration by
$$\mathrm{Fl}(\mathcal{V}_i)=\{\mathrm{Fl}^m(\mathcal{V}_i)\},\hspace{1mm} with \hspace{2mm} \mathrm{Fl}^m(\mathcal{V}_i)=\bigoplus_k\mathrm{F}^{m-c_{jk}}(\mathcal{A}[\mathcal{D}])\bullet \gamma_k^i,$$
via the canonical filtration of Proposition \ref{prop: Char inclusions}.
On the one hand, by passing to graded objects we get, for every $m\in \mathbb{Z},$
$0\leftarrow \mathrm{Gr}_m^F(\mathcal{M})\leftarrow \mathrm{Gr}_m^F(\mathcal{V}_0)\leftarrow \mathrm{Gr}_m^F(\mathcal{V}_1)\leftarrow \cdots.$
On the other, dualizing via the $\mathcal{A}[\mathcal{D}]$-dual $\mathcal{H}om_{\mathcal{A}[\mathcal{D}]}(-,\mathcal{A}[\mathcal{D}]),$ gives a complex of $\mathcal{A}^r[\mathcal{D}_X^{op}]$-modules,
$$0\rightarrow \mathcal{H}om_{\mathcal{A}[\mathcal{D}]}(\mathcal{V}_0,\mathcal{A}[\mathcal{D}])\rightarrow  \mathcal{H}om_{\mathcal{A}[\mathcal{D}]}(\mathcal{V}_1,\mathcal{A}[\mathcal{D}])\rightarrow \cdots.$$
Since 
$\mathcal{E}xt_{\mathcal{A}[\mathcal{D}]}^k$ sheaves agree with the $k$-th cohomology of this complex, standard arguments give that 
$$\mathcal{E}xt_{\mathrm{Gr}^F\mathcal{A}[\mathcal{D}]}^k\big(\mathrm{Gr}^F(\mathcal{M}),\mathrm{Gr}^F\mathcal{A}[\mathcal{D}]\big)\simeq \mathcal{H}^k\big(\mathcal{H}om_{\mathrm{Gr}^F\mathcal{A}[\mathcal{D}]}(\mathrm{Gr}^F\mathcal{V}_{\bullet},\mathrm{Gr}^F\mathcal{A}[\mathcal{D}]\big),$$
for each $k$ and this shows in particular that $\mathrm{Gr}^F(\mathcal{H}^k(\mathcal{H}om_{\mathcal{A}[\mathcal{D}]}(\mathcal{V}_{\bullet},\mathcal{A}[\mathcal{D}])\big)$ is isomorphic to a sub-quotient of $\mathcal{H}^k\big(\mathrm{Gr}^F\mathcal{H}om_{\mathcal{A}[\mathcal{D}}(\mathcal{V}_{\bullet},\mathcal{A}[\mathcal{D}])\big).$
Via standard $\D$-module arguments and using Proposition \ref{prop: Associated graded}, we have that 
\begin{eqnarray*}
\mathrm{Char}\big(\mathcal{E}xt_{\mathcal{A}[\mathcal{D}]}^k(\mathcal{M},\mathcal{A}[\mathcal{D}])\big)&\subset &\mathrm{Supp}\big(\mathcal{E}xt_{\mathcal{O}_{T^*(X,Z)}}^k(\mu \mathrm{Gr}^F\mathcal{M}^{\bullet},\mathcal{O}_{T^*(X,Z)})\big)
\\
&\subset& \mathrm{Supp}(\mu\mathrm{Gr}\mathcal{M}^{\bullet}),
\end{eqnarray*}
again via the filtration of Proposition \ref{prop: Char inclusions}.
\end{proof}

Finally, let $\mathcal{A}_{X/S}^{\ell}=J_{X/S}(\mathcal{O}_C)$ for a commutative $\mathcal{O}_{X\times S}$-algebra $\mathcal{O}_C.$ Then, $i_s^*\mathcal{A}_{X/S}=i_s^*J_{X/S}(\mathcal{O}_C),$ is the pull-back sheaf of $\mathcal{D}_{X/S}$-algebras given by Proposition \ref{prop: pb Jets}. Assume $X/S$ is a non-characteristic fibration (c.f \ref{ssec: NC subspaces}). Then fiber-wise derived restriction defines a functor for each $s\in S$ on derived categories from $\mathcal{A}_{X/S}[\mathcal{D}_{X/S}]$-modules to $i_s^*\mathcal{A}_{X/S}[\mathcal{D}_X]$-modules.

\begin{rmk}
If $\mathcal{M}^{\bullet}$ is a finitely presented $\mathcal{A}[\mathcal{D}_{X/\times S/S}]$-module in amplitude $[p,q]$ then locally in a neighbourhood of $(x,s)\in X\times S$ there is a complex of free $\mathcal{A}[\mathcal{D}]$-modules of finite rank in degrees $[p-2n-n_S,q]$ quasi-isomorphic to $\mathcal{M}.$ Moreover, considering any two such modules, the sheaf of solutions satisfies $i_s^*\mathbb{R}\mathcal{S}\mathrm{ol}_{\mathcal{D}_{X\times S/S}}(\mathcal{M},\mathcal{N})\simeq \mathbb{R}\mathcal{S}\mathrm{ol}_{i_s^*\mathcal{D}_{X\times S/S}}(i_s^*\mathcal{M},i_s^*\mathcal{N}).$
\end{rmk}

\subsubsection*{Flatness}
Consider a complex $\mathcal{E}^{\bullet}$ of $\mathcal{O}_X$-modules and $\mathcal{M}^{\bullet}$ a complex of $\D_X$-modules. It naturally determines a complex of $\mathcal{O}_X$-modules, denoted $\mathrm{For}_{\mathcal{D}}(\mathcal{M}^{\bullet})$, by forgetting the $\D_X$-module structure. We follow \cite{BD}.
\begin{defn}
\label{defn: Flatness Definitions}
\normalfont
Complex $\mathcal{E}^{\bullet}$ (resp. $\mathcal{M}^{\bullet}$) is \emph{homotopically $\mathcal{O}_X$-flat} if for all acyclic $\mathcal{O}$-modules $\mathcal{F}^{\bullet}$, the complex $\mathcal{E}\otimes\mathcal{F}$ (resp. $\mathrm{For}_{\mathcal{D}}(\mathcal{M})\otimes \mathcal{F}$) is acyclic.
Similarly, $\mathcal{M}^{\bullet}$ is \emph{homotopically $\mathcal{D}_X$-flat} if for all acyclic complexes $\mathcal{L}^{\bullet}$ of $\mathcal{D}_X$-modules, $\mathcal{M}\otimes_{\mathcal{D}_X}\mathcal{L}$ is an acyclic complex of $k$-vector spaces.
Finally, a complex of $\mathcal{D}$-modules $\mathcal{M}^{\bullet}$ is \emph{homotopically $\D_X$-quasi-induced} if for all homotopically $\mathcal{O}_X$-flat $\D_X$-modules $\mathcal{L}^{\bullet}$ there is a quasi-isomorphism $\mathcal{M}^{\bullet}\otimes_{\mathcal{D}_X}^{\mathbb{L}}\mathcal{L}^{\bullet}\simeq \mathcal{M}\otimes_{\mathcal{D}}\mathcal{L}.$
\end{defn}
Any induced $\D_X$-module is quasi-induced. This can be written as the condition that for an $\mathcal{O}_X$-flat $\D_X$-module $\mathcal{M}$, we have $Tor_{>0}^{\mathcal{D}_X}(\mathcal{M},\mathcal{L})=0,$ or globally, as 
$$R\Gamma_{dR}(X,\mathcal{M}\otimes\mathcal{L})\simeq R\Gamma\big(X,h(\mathcal{M}\otimes\mathcal{L})\big).$$
We collect some properties of quasi-induced $\D$-modules.
\begin{prop}
The following statements hold:
\begin{enumerate}
\item The direct image of a quasi-induced
$\D_X$-module under an affine morphism is quasi-induced;
\item The exterior tensor product
of quasi-induced $\D_X$-modules is quasi-induced;
\item A holonomic $\D_X$-module is quasi-induced if and only if its support has dimension $0$.
\item Any tensor product of an $\mathcal{O}_X$-flat
$\D_X$-module and a quasi-induced $\D_X$-module is quasi-induced.
\end{enumerate}
\end{prop}

For a differential graded $\D$-algebra $\mathcal{A}$, put $U(\mathcal{A})$ the dg-$\mathcal{O}_X$-algebra which is quasi-coherent. 
\begin{defn}
    \label{definition: Flatness for D-Algebras}
    \normalfont 
    A morphism $f:\mathcal{A}\rightarrow \mathcal{B}$ of dg-$\D$-algebras is \emph{faithfully $\D$ flat} if $U(f):U(\mathcal{A})\rightarrow U(\mathcal{B})$ is faithfully flat.
\end{defn}
 If $f:\mathcal{A}\rightarrow \mathcal{B}$ is faithfully $\D$-flat then $\mathcal{B}$ is homotopically $\mathcal{A}$-flat if it is homotopically $\mathcal{A}$-flat as a dg-$\mathcal{A}[\mathcal{D}_X]$-module i.e. for all $\mathcal{M}^{\bullet}\in \DG(\mathcal{A}),$ the complex $\mathcal{M}\otimes_{\mathcal{A}}\mathcal{B}$ is acyclic. For non-dg-$\D$-algebras, e.g. those concentrated in degree zero, given a morphism $f:\mathcal{B}\rightarrow \mathcal{C}$, $\mathcal{C}$ is $\mathcal{B}$-flat as a $\D$-algebra if the morphism preserves $\D$-action and multiplication, and $\mathcal{C}$ is flat as a $\mathcal{B}$-module, via $f.$

\begin{prop}
\label{prop: Global HoFlat}
    Let $X$ be $D$-affine and and suppose that $Z$ is a smooth $\D$-scheme. Suppose that $\mathcal{M}^{\bullet}$ is a homotopically flat $\mathcal{O}_Z[\mathcal{D}_X]$-module. Then $\Gamma(Z,\mathcal{M})$ is homotopically flat over $\Gamma(Z,\mathcal{O}_{Z}[\mathcal{D}_X]).$   
\end{prop}
\begin{proof}
    Any such $\mathcal{M}$ has a finite left resolution since it is flat as a $\mathcal{O}_Z$-module (de Rham resolution). Thus, to prove the claim suppose
$0\rightarrow \mathcal{N}_1\rightarrow \mathcal{N}_0\rightarrow \mathcal{M}\rightarrow 0,$ is a short exact sequence with $\mathcal{N}_i$ globally flat. We must prove that if $\mathcal{M}$ is flat, it is globally flat. To this end, decompose $\mathcal{N}_i$ as a direct sum over finitely generated free $\mathcal{O}_Z[\mathcal{D}]$-modules. Then 
for all $\mathcal{L}$ a complex of left $\mathcal{O}_Z[\mathcal{D}_X]$-modules, the natural map
$\Gamma(Z,\mathcal{N}_1)\otimes_{\Gamma(Z,\mathcal{O}_{Z}[\mathcal{D}_X])}\Gamma(Z,\mathcal{L})\rightarrow \Gamma(Z,\mathcal{N}_0)\otimes_{\Gamma(Z,\mathcal{O}_{Z}[\mathcal{D}])}\Gamma(Z,\mathcal{L}),$
is injective and we have an isomorphism,
$\Gamma(Z,\mathcal{N}_i)\otimes_{\Gamma(Z,\mathcal{O}_{Z}[\mathcal{D}])}\Gamma(Z,\mathcal{L})\simeq \Gamma(Z,\mathcal{N}_i\otimes_{\mathcal{O}_Z[\mathcal{D}]}\mathcal{L}),$ thus since $\mathcal{M}$ is (homotopically) flat by assumption, the map $\mathcal{N}_1\otimes_{\mathcal{O}_Z[\mathcal{D}]}\mathcal{L}\hookrightarrow \mathcal{N}_0\otimes_{\mathcal{O}_Z[\mathcal{D}]}\mathcal{L})$ is an injective $\D$-module morphism. 
\end{proof}

\begin{prop}
    Suppose that $X$ is quasi-projective. Then there exists a homotopy $\D$-flat non-unital commutative $\D$-algebra resolution of $\mathcal{O}_X.$
\end{prop}
Moreover, for arbitrary $X$ one can take a resolution whose terms are isomorphic as graded commutative $\D$-algebras to free $\D$-algebras on $X$-locally projective graded $\D$-modules $\mathcal{M}$ in degrees $\leq 0$ i.e. $\mathrm{Sym}^{>0}(\mathcal{M}).$ 

When $X$ is quasi-projective $\mathcal{M}$ can be chosen to be a direct sum of induced $\D$-modules $\mathrm{ind}_{\mathcal{D}}^{\ell}(L)$ for a line bundle $L.$ In particular, each $L$ is a negative power of a given ample line bundle.

\subsection{$\mathcal{D}_X$-Geometric Hilbert Polynomials}
\label{ssec: D-Hilb Polynomials}
Given a linear system $\mathcal{M}$, its associated symbol generates an ideal in $\mathbb{C}[\xi_1,\ldots,\xi_n].$ This ideal gives rise to a variety in $T^*X$, denoted as usual $\mathrm{Char}(\mathcal{M})$ and one gets a (polynomial) function
$H_{\mathcal{M}}'(n)$ defined to be the dimension of the space of formal solutions. Via standard constructions of Gröbner bases, there exists a Hilbert function $H_{\mathcal{M}},$ and for $n\gg 0$, $H_{\mathcal{M}}(n)=H_{\mathcal{M}}'(n).$
We highlight this in two situations of interest.
\begin{notate}
Following \cite{CFK2}, given a (graded) associative algebra $A$ and a left $A$-module $M,$ one has $A$-Grassmannians
$Gr_{A}(k,M)\subset Gr(k,M),$
whose scheme of left ideals is $J(k,A):=Gr_{A}(k,A).$ We apply it to commutative algebras $A$. This describes two-sided ideals of $A$ whose quotients $A/I$ are again commutative.
\end{notate}

Fix coordinates $(x;\xi)\in T^*X$ and some integer codimension $C\in \mathbb{N}.$ Consider non-reduced sub-schemes of $T^*X$ cut-out by symbol ideals. We consider not only prime ideals $\mathcal{I}$ but $\mathcal{I}$-primary ideals $\mathcal{J}$ and consider the Hilbert scheme of points 
$Hilb^k\big(\mathbb{C}[\xi_1,\ldots,\xi_C\!]\big).$

\begin{prop}
\label{cor: Application 1}
Let $J_{C}(k,T^*X)$ be the set of codimension $C$ primary ideals in $\mathcal{O}_{T^*X}$ of multiplicity $k.$ 
Then there is an isomorphism of sets
$$Hilb^k\big(\mathbb{C}[\xi_1,\ldots,\xi_C]\big)\simeq J_{C}(k,T^*X),$$ as well as an equivalence of quasi-projective schemes,
\begin{equation}
    \label{eqn: Hilb and G equi}
Hilb^k\big(\mathbb{C}[\xi_1,\ldots,\xi_C]\big)\simeq Gr_{\mathbb{C}}^{\mathcal{D}}(k,\mathbb{C}[z_1,\ldots,z_C]),
\end{equation}
where $Gr_{\mathbb{C}}^{\mathcal{D}}(k,\mathbb{C}[z_1,\ldots,z_C])$ is the scheme of $k$-dimensional $\mathbb{C}$-vector subspaces of $\mathbb{C}[z_1,\ldots,z_C]$ closed under differentiation i.e. $\D$-subspaces.
\end{prop}
\begin{proof}
Firstly, by Grothendieck \cite{Gro3} we have that $Hilb^k(\mathbb{C}[\xi])$ is a quasi-projective scheme over $\mathbb{C}$. Its classical points are
ideals of co-length $k$ in the local ring $\mathbb{C}[\xi]$ (see also \cite{HL}). This ring is the completion of the polynomial ring in variables $\xi_i$ given by localization at the prime $\mathcal{I}.$ Recall the multiplicity $k$ of a primary ideal $\mathcal{J}$ over its
prime $\mathcal{I}=\sqrt{\mathcal{J}}$ is the length of a corresponding Artinian local ring. 

Let $\mathcal{I}$ denote a point of $J_{C}(k,T^*X)$ i.e. a primary ideal and note that by making formal replacements
$\partial_i\leftrightarrow \xi_i,$ with $\partial_{z_i}$ for $i=1,\ldots,C,$ (essentially via microlocalizing or Fourier-transformation) this ideal gives a point in 
$Hilb^k\big(\mathbb{C}[\partial_1,\ldots,\partial_{C}]\big).$ In one direction the differentially stable vector space is determined by the space of solutions in Proposition \ref{prop: Sol}: 
$$Sol(\mathcal{I}):=\{f\in \mathbb{C}[z_1,\ldots,z_C]|\mathbf{F}\bullet f=0,\forall \mathbf{F}\in \mathcal{I}\}.$$
Conversely, any $k$-dimensional $\mathbb{C}$-vector space $V$ stable under the action by $\mathcal{D}_{\mathbb{C}^C}$, determines a point in the punctual Hilbert scheme via $\mathcal{I}:=\mathrm{Ann}_{\mathbb{C}[\partial_1,\ldots,\partial_C]}(V),$
or equivalently $\mathrm{Ann}_{\mathbb{C}[\xi_1,\ldots,\xi_C]}(V).$
Thus, the point of $J_{C}(k,T^*X)$ determined by $\mathcal{I}$ is the ideal $\mathcal{J}:=\sqrt{\mathrm{Ann}_{\mathbb{C}[\xi_1,\ldots,\xi_C]}(V)}.$ The corresponding algebraic variety defined by $\mathcal{J}$ is the characteristic variety (\ref{annideal}). 
\end{proof}

These ideas also apply to the study of ideals in the Weyl algebra. That is, consider the moduli problem of studying finitely generated ideals in the $n$-th Weyl algebra $A_n:=\mathbb{C}[x_1,\ldots,x_n]\big<\partial_1,\ldots,\partial_n\big>.$
Equivalently, consider a finite set of generators $N\subset (A_n)^r$ given as an $A_n$-submodule, and study quotient modules
$M:=(A_n)^r/N.$ Also, we could study modules of the form
$$\mathcal{M}:=\mathcal{D}_{\mathbb{C}^n}\otimes_{A_n}\big(A_n^r/N\big),$$
for some generators $N=\sum_{i=1}^s A_n\cdot P_i$ for $P_i\in (A_n)^r,$ for $i=1,\ldots,s.$

\begin{prop}
\label{cor: Application 2}
    There is an equivalence between finitely generated left ideals in $A_n$ with a trivialization to $\mathcal{D}_{\mathbb{C}^n}$ and ideal sheaves
    $\mathcal{I}_{\nu}\subset \mathcal{O}_{T^*X},$
    with fixed numerical characteristics given via the polynomial
    $H:=\sum_{\nu=1}^rdim_{\mathbb{C}}\big(\mathbb{C}\{x,\xi\}/\mathcal{I}_{\nu}+\mathcal{J}^k\big),$
    for the maximal ideal $\mathcal{J}$ of $\mathbb{C}\{x,\xi\}.$
\end{prop}
\begin{proof}
Notice that 
$H(k)$ of (\ref{ssec: D-Hilb Polynomials}) agrees with the Hilbert polynomial of the microlocalization $\mu\big(gr^F(\mathcal{M})\big)$ for a `canonical' filtration. That is, using coherent filtrations 
$$H_{\mu(\mathrm{Gr}^F(\mathcal{M}))}(k)=\mathrm{dim}_{\mathbb{C}}\big(\mu \mathrm{Gr}^F\mathcal{M}/\mathcal{J}^k\cdot \mu \mathrm{Gr}^F\mathcal{M}\big).$$
The result follows from a standard argument via Gröbner bases. Roughly, let $\mathbf{G}_{\nu}:=\{P| lp(P)=\nu\},\forall\nu=1,\ldots,r,$
where $lp(P)=i$ is the leading point of $P$, defined when $lexp(P)=(\alpha,i).$ Since the characteristic variety is given by 
$$\mathrm{Char}(\mathcal{M})=\bigcup_{\nu=1}^r\big\{(x,\xi)\in T^*X|\sigma_{\nu}(P)(x,\xi)=0\forall P\in\mathbf{G}_{\nu}\big\}\hookrightarrow T^*X,$$
the result follows as the ideal $\mathcal{I}_{\nu}$ is generated by 
$\sigma(\mathbf{G}_{\nu}):=\{\sigma(P)_{\nu}|P\in\mathbf{G}_{\nu}\}.$
\end{proof}

In the $\D$-algebraic setting these notions must be adjusted using canonical filtrations and the characteristic module (\ref{eqn: Char Module of A}).

\subsubsection{Growth functions and Krull dimension}
Consider a $k$-th order PDE $Z=\{Z_k\subset J_X^kE\}.$ Following ideas of \cite{B}, we may define a numerical growth function,
\begin{equation}
    \label{eqn: Growth}
    G_{Z_k}:\mathbb{N}\rightarrow \mathbb{N}, G_{Z}(n):=\mathrm{dim}(\mathrm{Pr}_{\ell}(Z_k)).
    \end{equation}
Say that $Z$ is of \emph{polynomial growth} if there exists a numerical polynomial $H\in \mathbb{R}[t],$ such that for $r\gg 0, \mathrm{dim}(\mathrm{Pr}_r(Z_k))=H(r),$ for $r\geq k.$  

Hilbert polynomials may be defined for symbolic systems associated with infinite-prolongations $Z^{\infty}$ of formally integrable equations. Thus, 
 $\D_X$-schemes determined by formally $\D$-involutive ideal sheaves are of polynomial growth. 
Indeed, for $\D$-involutive ideal sheaves, fixing the $\D$-Hilbert polynomial is equivalent to fixing that of its characteristic (symbolic) $\D$-module.
Fix a differentially generated ideal $\mathcal{I}$ with filtration $F^k\mathcal{I}$. 

 It is natural to define the `Krull dimension $\I$', by
\begin{equation}\label{eqn: Krull}
\mathfrak{d}_{\mathcal{I}}:\mathbb{Z}_{+}\rightarrow \mathbb{Z}_+, k\mapsto \mathfrak{d}_{\mathcal{I}}(k):=dim\big(F^k\mathcal{A}/F^k\mathcal{I}\big).
\end{equation}
The following result is easily checked.
\begin{prop}
\label{prop: D-Dim Polynomial}
    Let $\mathcal{A}$ be the $\D$-smooth infinite jet $\mathcal{D}_X$-algebra in $\mathrm{rank}(E)$-variables with $d_X=\mathrm{dim}(X).$ Let $\mathcal{I}$ be a $\D$-ideal corresponding to a formally integrable involutive algebraic PDE. Then, there is a numerical polynomial $\mathbf{h}_{\mathcal{I}}(k)\in\mathbb{Q}[k],$ such that $\mathbf{h}_{\mathcal{I}}(k)=\mathrm{dim}(F^k\mathcal{A}/F^k\mathcal{I}),$ for sufficiently large $k$. Moreover, $\mathbf{h}_{\mathcal{I}}$ satisfies:
    \begin{itemize}
        \item[(i)] $0\leq \mathbf{h}_{\mathcal{I}}(k)\leq \mathrm{rank}(E)\cdot C_{d_X}^{k+d_X}.$
        \item[(ii)] If $\mathcal{I}\subseteq \mathcal{J}$, then $\mathbf{h}_{\mathcal{J}}\leq \mathbf{h}_{\mathcal{I}}.$

        \item[(iii)] Prime $\mathcal{D}_X$-ideals $\mathcal{I}\subseteq\mathcal{J}$ are equal if and only if $\mathbf{h}_{\mathcal{I}}=\mathbf{h}_{\mathcal{J}}.$ 
    \end{itemize}
    Furthermore, the degrees of this polynomial is bounded as $deg_{k}(\mathbf{h}_{\mathcal{I}})\leq \mathrm{dim}(X).$
\end{prop}
\begin{proof}
Follows immediately from observing that $\mathcal{I}\subseteq\mathcal{J}$ implies that $F^k\mathcal{I}\subseteq F^k\mathcal{J}$ for all $k\geq 0.$ 
This induces a surjective morphism (of $\mathcal{O}_X$-modules), $F^k\mathcal{B}\rightarrow F^k\mathcal{C},$ such that 
$\mathrm{dim}(F^k\mathcal{A}/F^k\mathcal{I})\geq \mathrm{dim}(F^k\mathcal{A}/F^k\mathcal{J}).$
In other words, the inclusion of ideals gives a sequence
$$0\rightarrow \mathcal{K}\rightarrow \mathcal{O}(J_X^{\infty}E)/\mathcal{I}\rightarrow \mathcal{O}(J_X^{\infty}E)/\mathcal{J}\rightarrow 0,$$
with $\mathcal{K}$ the kernel of the induced surjection from which 
$$0\rightarrow F^k\mathcal{K}\rightarrow \mathcal{O}(J_X^kE)/\mathcal{I}\cap \mathcal{O}(J_X^kE)\rightarrow \mathcal{O}(J_X^kE)/\mathcal{J}\cap \mathcal{O}(J_X^kE)\rightarrow 0,$$
which tells us 
$h_{\mathcal{I}}(k)=h_{\mathcal{J}}(k)+dim(F^k\mathcal{K})\geq h_{\mathcal{J}}(k),$ as required.
\end{proof}
The $\D$-Hilbert function counts the number of free functions in the Taylor expansion of a generic formal solution. For formally integrable and involutive systems a general solution depends on $d$ functions of $p$ variables with $p\in \mathbb{N}$ called the functional dimension of $\mathrm{Sol}_{\mathcal{I}}$ while $d$ is the functional rank (see e.g. \cite{KL},\cite{Sei}).
Compute the Cartan numbers using $\mathrm{Char}(\mathcal{B})$ via the Hilbert-Serre theorem e.g. $p=\mathrm{dim}(Char)+1.$ An explicit formula for the $\D$-Hilbert polynomial can be written down from the dimension microcharacteristic varieties given by Proposition \ref{prop: Char inclusions}.
\begin{prop}
\label{prop: Hilbprop}
Let $\mathcal{I}$ be a Spencer-regular $\D$-ideal and suppose it is of formal codimension $N$ and of pure order $k$. Then there is a numerical polynomial $P$ associated with $\mathcal{I}$ which agrees with the Hilbert function of its symbolic system $P_{\mathcal{B}}(z)=H_{\mathcal{M}}(z).$ If $p=\mathrm{deg}(P_{\mathcal{B}}(z))$ and $d=P_{\mathcal{B}}^{(p)}(z)$ then 
    $P_{\mathcal{B}}(z)=d\cdot z^p+(\ldots),$
omitting the lower order terms.
Generally, one may write
$$P_{\mathcal{B}}(z)=\sum_{i,j,k}(-1)^i\beta^{q,i}s_{j}^n\frac{(z-k-i)^{n-j}}{(n-j)!},$$
where $\beta^{i,j}$ are the Betti numbers determined by the resolution of the symbol module, $s_i^n:=\frac{(n-i)!}{n!}S_i(1,\ldots,n)$ where $S_j(k_1,\ldots,k_n)=\sum_{i_1<\ldots<i_j}k_{i_1}\cdots k_{i_j}$ is the $j$-th symmetric polynomial associated to the orders $k_i$ of the equations defining the system. 
\end{prop}


In order to calculate the Hilbert polynomial in practice one uses a resolution of the linearization of the symbol module. In this linearized case, the resolution of $\mathcal{M}_{\mathcal{I}}$ exists and is expressed as the $\mathbb{R}$-dual of the Koszul homology, denoted $H_{\delta}^{i,j}$ then $\beta^{i,j}=\mathrm{dim}(H_{\delta}^{i,j}).$

For formally integrable PDEs, $P_{\D}(\I;\ell)=\sum_{i\leq \ell}\mathrm{dim}(\mathcal{N}_i)$ is a polynomial for $\ell\gg 1,$ so $\mathrm{dim}(\mathcal{N}_z)=P_{\D}(\I;z)-P_{\D}(\I;z-1)$ for all larger integers $z=\ell,$ which extends to all $z\in\mathbb{C}.$
The following convenient formula computes the $\D$-Hilbert function,
\begin{equation}
\label{eqn: DHilbAgain}
P_{\D}(\I;z)=\sum_{p,q}(-1)^{p+q}\mathrm{dim}\mathcal{H}_{Sp}^{p,q}(\mathcal{N})\cdot C^{z+n-p-q-1}_{n-1}.
\end{equation}

\begin{prop}
\label{prop: Hilbert D-Polynomial}
    Let $\mathcal{B}$ be a $\mathcal{D}_X$-smooth non-linear PDE i.e. $H^0(\mathbb{T}_{\mathcal{B}})\simeq \Theta_{\mathcal{B}}.$ Then there exists a numerical polynomial 
    $H_{\Theta_{\mathcal{B}}}'$ defined by 
    $\mathrm{dim}\big(F^k\Theta_{\mathcal{B}}^{\ell}\big)\simeq \mathrm{dim}\big(\Theta_{F^k\mathcal{B}}^{\ell}\big),$
    as a $F^k\mathcal{B}$-vector space such that for $k$ large enough,
    $H_{\Theta_{\mathcal{B}}}'$ agrees with $\mathbf{h}_{\mathcal{I}}.$
\end{prop}
Rather than fixing a $\D$-numerical polynomial, for vector $\D$-bundles by the Remark \ref{rmk: Ranks and Chern class}, one can instead fix the $\D$-geometric Chern class.

\subsection{Spencer (semi)stability}
\label{ssec: Spencer stability}
Via Proposition \ref{prop: D-Dim Polynomial} and Proposition \ref{prop: Hilbprop} we define a \emph{reduced $\D$-Hilbert polynomial} by normalizing via the generic (functional) rank\footnote{This is the first Cartan character \cite{C}.} of the solution sheaf, i.e., the number of independent functional parameters in a formal solution, at each $x\in X.$  Denote it by $\mathrm{rank}(\mathcal{I}),$ and put
\begin{equation}
    \label{eqn: Reduced D-Hilb}
\overline{P}_{\mathcal{D}}(\mathcal{I},n):=\frac{P_{\mathcal{D}}(\mathcal{I},n)}{\mathrm{rank}(\mathcal{I})}.
\end{equation}
This suggests the following definition.
\begin{defn}
\normalfont
\label{defn: Spencer ss}
A $\D$-ideal sheaf $\mathcal{I}$ is \emph{Spencer semi-stable} (resp. \emph{Spencer stable}) if for every differentially-generated involutive $\D$ sub-ideal $\mathcal{J}\subset \mathcal{I}$, (resp. proper ideal),
we have that $\overline{P}_{\D}(\mathcal{J})\leq \overline{P}_{\mathcal{D}}(\mathcal{I})$, (resp. $\overline{P}_{\mathcal{D}}(\mathcal{J})<\overline{P}_{\mathcal{D}}(\mathcal{I})$).
\end{defn}
Representing the numerator as the Euler characteristic of the Spencer $\delta$-complex (\ref{eqn: Spencer delta symbol}), Spencer-semistability is an equality,
$$\frac{\sum_{i}(-1)^ih^i(\mathcal{J},n)}{\mathrm{rank}(\mathcal{J})}\leq \frac{\sum_{i}(-1)^ih^i(\mathcal{I},n)}{\mathrm{rank}(\mathcal{I})},\hspace{2mm}\text{ for } n \gg0.$$

\begin{prop}
    Suppose that $\I$ is $m$-involutive for some $m\geq 0.$ Then, $\mathrm{rank}(\mathcal{I})=dim\mathcal{H}^{0,m-1}\big(\mathcal{S}p(\mathrm{gr}\mathcal{I})\big)$. More generally, it is the leading coefficient of $P_{\mathcal{I}}(n)\cdot n^{p}$, where $p$ is determined by the dimension of characteristic variety.
\end{prop}
It is convenient to also define \emph{Spencer slopes}, as
\begin{equation}
\label{eqn: Spencer slope}
\mu_{Sp}(\mathcal{I}):=\frac{\mathrm{deg}\big(H^0(\mathcal{S}p(\mathcal{I})\big)}{\mathrm{rank}(\mathcal{I})}.
\end{equation}
There is an obvious notion of \emph{Spencer slope semistability} as $\mu_{Sp}(\mathcal{J})\leq \mu_{Sp}(\mathcal{I}).$ Slope stability is defined via strict inequality.
\begin{prop}
    Suppose that $\mathcal{I}$ is a differentially finitely generated $\D$-ideal. Suppose it is Spencer semi-stable (resp. Spencer stable). Then it is Spencer slope semi-stable (resp. Spencer slope stable).
\end{prop}
\begin{proof}
Note that the degree is equal to the leading coefficient in $\mathrm{lim}_{k\rightarrow \infty}P(\mathcal{I};k)/k^d,$ with $d$ the characteristic dimension. Then, note that (\ref{eqn: Spencer slope}) are equivalently written as $lim_{k\rightarrow \infty}P(\mathcal{I},k)/\mathrm{rank}(\mathcal{I})\cdot k^d.$ It is then a standard argument that one may consider $\overline{P}(\mathcal{J},k)\leq \overline{P}(\mathcal{I},k)$ for large $k$, and compare leading order terms, which are precisely the degrees. Then, the Spencer slope semi-stability (resp. stability) inequality gives $\mu_{Sp}(\mathcal{J})\leq \mu_{Sp}(\mathcal{I})$ (resp. $\mu_{Sp}(\mathcal{J})<\mu_{Sp}(\mathcal{I})$).
\end{proof}

\subsection{Differential algebra quotients by algebraic pseudogroups}
\label{ssec: Differential algebra quotients}
We describe the quotient $\D$-ideal, to be denoted $\mathcal{I}^{\mathcal{G}},$ by an algebraic group, and the $\mathcal{G}$-invariant symbolic system $\mathrm{gr}(\mathcal{I})^{\mathcal{G}}:=\mathrm{gr}(\mathcal{I}^{\mathcal{G}}).$ These constructions are used below to prove that the moduli space of Spencer semi-stable $\D$-ideal sheaves admits a presentation as a GIT-quotient.
Working in the (almost) algebraic category, we make use of a type of Noetherian property on quotient symbolic systems which leads to finite generation of the $\D$-algebra of functions constant along orbits.

The main result of this subsection is given by Proposition \ref{prop: GeometricQuotientEquation} which roughly states that, given an algebraic action on a formally integrable $\D$-scheme $Z$, then if $G$ acts algebraically and transitively on $X$, there exists an integer $\ell$ and a Zariski closed $G$-invariant proper subset $W_{\ell}\subset F^{\ell}Z,$ such that the action is regular on $(p_{\ell}^{\infty})^{-1}(W_{\ell})\subset Z.$
In other words, for any $k\geq \ell,$ the orbits on $F^kZ\backslash (p_{\ell}^k)^{-1}(W_{\ell})$ are closed, equidimensional and have the structure of algebraic varieties. This implies the existence of a quotient equation as a geometric quotient.

\subsubsection*{Recollections on geometric quotients}
Recall that a surjective open moprhism between algebraic varieties $f:X\rightarrow Y$ is a \emph{geometric quotient} if the fibers $f^{-1}(y)$ are $G$-orbits and for every open subset $U\subset Y$ the map $\mathcal{O}_Y(U)\rightarrow \mathcal{O}_X\big(f^{-1}(U)\big)^G$ is an isomorphism of algebras.

A geometric quotient is a categorical quotient. This means that any morphism $g:X\rightarrow Z$ of algebraic varieties constant along the orbits of $G$  can be factorized through $f$ i.e. there exists a morphism $h:Y\rightarrow Z$ such that $g=h\circ f.$ In particular, if a geometric quotient $X/G$ exists, it is unique in this categorical sense.

For differential equations, we study categorical quotients by \emph{pseudogroups}, following \cite{SiSt},\cite{Kum}.
\begin{defn}
\normalfont
A \emph{pseudogroup} is a collection
of local diffeomorphisms $G$, that contains unit, inverse,
and composition whenever defined. It is called a \emph{Lie pseudogroup} if
its elements are solutions to a system of differential equations.
\end{defn}

A \emph{Lie equation} of differential order $\leq r$ means an embedding $G^{(r)}\subset J_X^r$ that uniquely determines a groupoid $\{G^k\}_{k\geq r},$ by prolongation (see \cite{Li,Li2}). By the Cartan-Kuranishi theorem, analytic pseudogroups transitive on $X$ are always Lie pseudogroups. We always assume a type of formal-integrability of $G=\{G^k\}$ i.e.  that $G^k$ is a bundle over $G^{k-1}$ for all $k>0.$ Its' infinite prolongation is thus well-defined and is denoted by $G^{\infty}.$ There are well-defined tangents $TG^k,k>0,$ and joint relative tangent sheaves under projections $G^k\rightarrow G^{k-1}$ define a notion of symbolic system for $G^{\infty}.$

One may extend for each $\phi\in G$ via prolongation to act $\mathrm{Pr}_{\ell}(\phi):J^{\ell}\rightarrow J^{\ell}$, by setting $\mathrm{Pr}_{\ell}(\phi\circ \psi^{-1})=\mathrm{Pr}_{\ell}(\phi)\circ \big(\mathrm{Pr}_{\ell}\psi\big)^{-1}.$ For example, $\phi:\mathbb{C}^2\rightarrow \mathbb{C}^2$ action by $Sl_2(\mathbb{C})$ prolongs to $J^k$ by $\phi^{(k)}[f]_x^k:=[\phi(f)]_{\phi(x)}^k.$
It extends compatibly as an action on $J_X^{\infty}$.

The stabilizer of order $\leq k$ at $x\in X$ is denoted $G_{x}^k:=\{\phi\in G^k|\phi(x)=x\}.$ It acts on $k$-jets based at $x$. 

\begin{defn}
\normalfont
Let $Z_k\subset J_X^k(E)$ be a system of PDEs.
    A $G$ action is \emph{algebraic} if for the order $r$ of the
pseudogroup the stabilizer $G_{x}^r$ is an algebraic group acting algebraically on $J_X^r(E)_{x},x\in X.$
There exists an action $G_x^k$ on $Z_{k,x}$ and we call it \emph{algebraic} if for every $x\in X$, $Z_{k,x}\subset J_{X,x}^k(E)$ is an algebraic (non-singular) subvariety on which $G_{x}^k$ acts algebraically. 
\end{defn}
The pair of a $\D$-scheme $Z:=\{Z_k\}$ together with a formally integrable algebraic pseudogroup action $G=\{G^{(r)}\}$ is called a \emph{$G$-algebraic $\D$-scheme}, written as a pair $(Z,G)\subset J_X^{\infty}(E).$

We compute the quotient, which when combined with the co-filtered Douady-type construction for differential ideals given by Proposition \ref{prop: Douady} below, proves the $\D$-Hilbert scheme is a GIT quotient of the Spencer-semistable $\D$-ideal sheaves (Subsect. \ref{ssec: GIT}, below). Call a subset $W\subset J_X^kE$ Zariski closed if its fiber $W_x:=W\cap J_{X,x}^k(E),x\in X$ is Zariski closed. Similarly for $W\subset Z_k.$

\begin{prop}
\label{prop: Irreducible}
Assume that $G$ acts on $Z$ such that all $G$-orbits project to $X$. Then the algebraic variety $Z_x^k$ is irreducible for every $k$ and for all $x\in X.$
\end{prop}
\begin{proof}
By assumptions it follows at each finite order $Z_k\subset J_X^kE$ it is enough to restrict to a single fiber $Z_x^k$ with induced action by $G_x^k.$ By induction on $k$, with base case $k=\mathrm{Ord}(\I)$ (max order). This case is immediate and suppose it holds for $\mathrm{Ord}(\I)-1.$ Since $Z$ is formally integrable, $(p^k_{k-1})_x:Z_X^k\rightarrow Z_{x}^{k-1}$ is affine. Fixing polynomials $F_1,F_2$ in $k$-jets at $x\in X$, assume $(F_1\cdot F_2)(z_k)=0$ for all $z_k\in Z_x^k$ and that $F_1\neq 0.$ Then, for all $z_k$ with $F_1(z_k)\neq 0,$ the fiber $(p^{k}_{k-1})^{-1}(z_{k-1}),z_{k-1}=p^k_{k-1}(z_k)$ is affine and irreducible so that $(F_1F_2)|_{(p^k_{k-1})^{-1}(z_{k-1})}=0$ implies $F_2$ restricted to the fiber vanishes. But since $F_1(z_k)\neq 0$ for generic $z_k\in Z^k$, this implies $F_2=0$ on all fibers over a Zariski open set in $Z_x^{k-1}$. Since $Z_x^{k-1}$ is irreducible, $F_2$ is identically zero.
\end{proof}
Consider $G$ as in Proposition \ref{prop: Irreducible}. Such an algebraic action has a geometric quotient outside a $G$-stable Zariski closed subset. In particular, such a subset is nowehere dense, and removing a finite-codimension locus of singularity the quotient of an algebraic $\D$-scheme by an algebraic pseudo group is again an algebraic $\D$-scheme.
\begin{prop}
\label{prop: GeometricQuotientEquation}
    Let $(G,Z)$ be a $G$-algebraic formally integrable $\D$-space defined by a $\D$-ideal $\mathcal{I}.$ Then, there exists an integer $r\geq 0,$ and a Zariski-closed invariant proper subset $W_r\subset F^r(Z)$ such that the action is regular in 
    $$(p_r^{\infty})^{-1}\big(F^rZ\backslash W_r\big)\subset \mathrm{Spec}_{\D}(\mathcal{A}/\mathcal{I})=Z,$$
    with $p^{\infty}:Z\rightarrow X$ the structure map with $p^{\infty}_r:Z\rightarrow F^rZ,$ the induced projections $r>0.$
    In particular, $F^kZ\backslash (p_{r}^k)^{-1}(W_r)$ for every $k\geq r$ is algebraically filtered by closed and equidimensional orbits of $G^{(k)}.$ Thus, there exists a geometric quotient,
    $$\big(F^kZ\backslash (p_r^{k})^{-1}(W_r)\big)/ G^{(k)}.$$
\end{prop}
We now give the proof but note it makes use of a result (Proposition \ref{prop: KumLemma}) proven after.
\begin{proof}
Regular orbits have the maximal possible dimension and they
fiber a neighborhood after possibly restricting to an open dense subset. We may remove singular orbits inducing a geometric quotient.

To this end, note by our assumptions, we have for each $k$ a proper closed subset $W_k\subset Z^k\subset J_X^k(E)$ such that its compliment admits a geometric quotient, 
$Q_0^k\simeq (Z^k\backslash W_k)/G^k,$ since $W_k$ is Zariski-closed meaning $W_k\cap Z_x^k$ is Zariski-closed for every $x\in X.$ Then, we have
$$(Z^k\backslash W_k)/G^k\simeq \big(Z_x^k\backslash (W_k\cap Z_x^k)\big)/G_x^k.$$
Fix $\mathfrak{l}>\mathrm{Reg}_{\D}(\mathcal{I}).$ Since $G$ is transitive on $X$, orbits in $Z_k$ project to $X.$ In particular, the choice of $x\in X$ is irrelevant and we may restrict to the fiber 
$$Q_{x}^k:=[Z^k/G^k]_x^0=\widetilde{Z}_k\cap p_{k}^{-1}(x)/G_x^k.$$
By Proposition \ref{prop: KumLemma}, such a Zariski subset $\widetilde{Z}_k$ exists. We want to show fibers are affine and their restricted projections are affine morphisms of algebraic varieties. They are clearly the orbits of the $G_{x}^k$-action on affine fibers $Z^{k+1}\cap (p_k^{k+1})^{-1}(z_k)$ for $z_k\in Z_k$, with $p_k(z_k)=x\in X.$ Taking the quotient, which is given by the fibers of the canonical projection
\begin{equation}
\label{eqn: Q proj}
[Z^{k+1}/G^{k+1}]^0=\widetilde{Z}_{k+1}\cap p_{k+1}^{-1}(x)/G_{x}^{k+1}\rightarrow [Z^k/G^k]^0=\widetilde{Z}_k\cap p_k^{-1}(x)/G_x^k,
\end{equation}
we see they are affine. To show the induced quotient tower by the maps (\ref{eqn: Q proj}) are affine, we show the locus of singularity is of finite-codimension in $Z^{\infty}.$
To this end, note that for each $k\geq \ell>\mathrm{Reg}(Z)$, there are no singulariteis over regular points $\widetilde{Z}_{\ell}=Z^{\ell}\backslash W_{\ell}.$ Thus, the singular stratum of the $\D$-scheme is $(p_{\ell}^{\infty})^{-1}(W_{\ell})\subset Z^{\infty}.$ It is of finite codimension as reguired and this $W_{\ell},$ given by $Z^{\ell}\backslash \widetilde{Z}_{\ell}$ is a Zariski closed locus of singularity. Thus, for each $k\geq \ell$ singularities are contained in $(p_{\ell}^k)^{-1}(W_{\ell}),$ and consequently, for all $k\geq\ell >\mathrm{Reg}(Z),$ there is a tower of quotient equations,
$$\cdots\rightarrow [F^{k+r}Z/G^{k+r}]\rightarrow\cdots\rightarrow [F^{k+1}Z/G^{k+1}]\rightarrow [F^kZ/G^k]\rightarrow\cdots.$$
Each morphism is an affine map and the corresponding limit defines the algebraic $\D$-scheme $[Z/G]^{\infty}.$

\end{proof}
Thus given an involutive and formally integrable $\D$-ideal $\mathcal{I}$, with defining exact-sequence $0\rightarrow \mathcal{I}\rightarrow \mathcal{O}(J_X^{\infty}E)\rightarrow \mathcal{B}\rightarrow 0,$ with corresponding $\D$-scheme $Z:=\mathrm{Spec}_{\D}(\mathcal{B})=\mathrm{Spec}_{\D}\big(\mathcal{O}(J_X^{\infty}E)/\mathcal{I}\big),$ by Proposition \ref{prop: GeometricQuotientEquation}, we obtain a quotient equation (which is again almost-algebraic) and we denote the corresponding $\D$-ideal by $\mathcal{I}^G$, with corresponding exact sequence
\begin{equation}
    \label{eqn: QuotientIdeals}
    0\rightarrow \mathcal{I}^G\rightarrow \mathcal{O}(J_X^{\infty}E)^G\rightarrow \mathcal{B}^G\rightarrow 0.
\end{equation}
The corresponding $\D$-scheme defined by (\ref{eqn: QuotientIdeals}) is denoted by $Z^G:=\mathrm{Spec}_{\D}(\mathcal{B}^G)=\mathrm{Spec}_{\D}\big(\mathcal{O}(J_X^{\infty}(E))^G/\mathcal{I}^G).$
Since $G=\{G^k\}$ is formally integrable, it defines a $\D$-group scheme $\mathcal{G}$, in the sense of \cite{BD}. We will show in Subsect. \ref{Ssec: Spencer-regularity of the geometric quotient}, that one may then write 
$$Z^G=\mathrm{Spec}_{\D}(\mathcal{B}^G)\simeq \big[\mathrm{Spec}_{\D}(\mathcal{B})/\mathcal{G}]^{0},$$
by taking the regular (stable) locus.

\subsection{Regularity of the geometric quotient}
\label{Ssec: Spencer-regularity of the geometric quotient}
We prove the following central result which states the complexity of the ideal $\mathcal{I}^G$ corresponding to the geometric quotient system given by Proposition \ref{prop: GeometricQuotientEquation} is controlled. This is necessary to ensure finite dimensionality and boundedness.

Since the $G$-action is transitive on $X$, $Z_x^k/G_x^k$ is a rational quotient whose set of singularities is nowhere dense, $x\in X.$ Denote by 
\begin{equation}
    \label{eqn: RegularGeomQuot}
    Q^k:=(Z^k/G^k)_0\subset Z^k/G^k,
\end{equation}
the geometric quotient on an open subset of $Z_x^k$ where the subscript indicates taking the regular part. 
The $G$-orbit of a $k$-jet of a section $j_k(s)(x)\in Z^k$ for some $x\in X$ is, by definition $G^k\cdot j_k(s)\in Z_x^k,$ and the tangent sheaf to this orbit, at the point $z_k=j_k(s)(x)\in Z^k$ is denoted $R_k(z_k).$ Setting $(Z_k)^0\subset Z_k,k\geq 0,$ the set of regular orbits, then $R_k|_{Z_k^0}\subset TZ_k^0,$ for each $k.$ A point $z_k\in Z_k$ is said to be \emph{regular} if $z_k\in Z_k^0.$

For any $z\in Z=Z^{\infty}$ the associated $\D$-scheme, with $z_k=p^{\infty}_k(z),$ consider the sequence, 
\begin{equation}
    \label{eqn: Regular tangent sequence}
    0\rightarrow T_{p^k_{k-1}}\rightarrow TZ_k^0\xrightarrow{(dp^k_{k-1})|_0}TZ_{k-1}^0\rightarrow 0,
\end{equation}
obtained by restriction of the differentials to the regular set $Z_k^0.$
Restricting further, to $R_k,$ we obtain a kernel,
\begin{equation}
    \label{eqn: W symbols}
    0\rightarrow ker(Tp^k_{k-1}|_{Z_k^0})\rightarrow R_k\rightarrow R_{k-1},\hspace{1mm} k\geq 0.
\end{equation}
By restriction from the equation manifold to its sub-variety of regular points $Z_k^0,$ one may show that the family $\mathrm{gr}(R)_0:=\{ker(Tp^k_{k-1}|_{Z_k^0})\}_{k\geq \ell_0}$ defines a symbolic system to this induced equation. One may then apply the Poincar\'e $\delta$-lemma, Proposition \ref{prop: Delta-Poincare}, to obtain the following result.
\begin{prop}
\label{prop: KumLemma}
    There exists an integer $\ell_0\geq 0$ and a $G$-invariant Zariski open subset $\widetilde{U}_{\ell_0}\subset Z_{\ell_0}$ such that for every geometric point $z\in Z^{\infty},$ there exists $z_{\ell_0}\in \widetilde{U}_{\ell_0}$ with $p^{\infty}_{\ell_0}(z)=z_{\ell_0}.$ Moreover, there exists a Zariski open subset $\widetilde{Z}^{\infty}\subset Z^{\infty}$ of the $\D$-scheme such that 
    $\mathcal{H}_{Sp}^{p,q}(\mathrm{gr}(R)_0)=0$ for $p\geq \ell_0$ and for all $q\geq 0$.
\end{prop}
\begin{proof}
This follows by applying \cite[Lemma 22.5,23.1]{Kum}. Explicitly, we note the Zariski-open subset is given by $\widetilde{U}^{\infty}:=(p^{\infty}_{\ell_0})^{-1}(\widetilde{U}_{\ell_0}).$
\end{proof}
We then study the quotient equation.
\begin{prop}
\label{prop: Existence of geometric quotient involutive degree}
Consider $\{Z/\!/ G\}:=\{Z^k/G^k\}_{k\geq 0}$ with geometric quotient $\D$-scheme given by (\ref{eqn: RegularGeomQuot}). 
Consider the fiber sequence,
$$0\rightarrow \mathfrak{q}_k\rightarrow T_{[Z^k/G^k]}\rightarrow T_{[Z^{k-1}/G^{k-1}]}\rightarrow 0.$$
Then there exists $\mathfrak{l}\in \mathbb{N}$ and a Zariski open subset $\widetilde{Z}_{\mathfrak{\ell}}\subset Z^{\mathfrak{\ell}}$ such that for every generic point $z_k\in Z^k$ with $p^k_{\mathfrak{l}}(z_k)=z_{\mathfrak{l}}\in \widetilde{Z}_{\mathfrak{l}},$ then 
$H_{Sp}^{p,q}(\mathfrak{q})=0,$ for all $p,\geq \mathfrak{l}$ and $q\geq 0.$ In particular, for every $k\geq \mathfrak{l}$ the number $dim(\mathfrak{q}_k)$ grows polynomially.
\end{prop}
\begin{proof}
Suppose $\mathfrak{l}\geq \mathrm{InvDeg}(Z)$. Then, $\mathcal{H}_{\delta}^{p,q}(\big(\mathrm{gr}(\mathcal{I})\big)=0$ for $p\geq \mathfrak{l},q\geq 0.$
By dualizing the isomorphism (\ref{eqn: ChKahler1}) and using Proposition \ref{prop: Omega and Ch isom}, consider there is and induced diagram,
\begin{equation}
\label{eqn: Symbol analysis diagram}
\adjustbox{scale=.90}{
\begin{tikzcd}
& 0\arrow[d] & 0\arrow[d] &  0\arrow[d] 
\\
0\arrow[r] & ker(Tp^k_{k-1}|_{Z_k^0}) \arrow[d] \arrow[r]& R_k\subset TZ_k^0\arrow[d]\arrow[r] & \arrow[d]R_{k-1}\subset TZ_{k-1}^0\rightarrow 0
\\
0\arrow[r] & \mathrm{gr}_k(\I)\arrow[d]\arrow[r] & TZ_k\arrow[d]\arrow[r] & TZ_{k-1} \arrow[d] \rightarrow 0
\\
0\arrow[r] & \mathfrak{q}_k \arrow[d]\arrow[r]& T(Z^k/G^k)\arrow[d]\arrow[r]& T(Z^{k-1}/G^{k-1})\arrow[d] \rightarrow 0   
\\
& 0 & 0&   0
\end{tikzcd}}
\end{equation}
By functoriality of Spencer $\delta$-complexes, diagram (\ref{eqn: Symbol analysis diagram}) induces a short-exact sequence of complexes,
$0\rightarrow \mathcal{S}p^{\bullet}(ker(Tp^k_{k-1}|_{Z_k^0}))\rightarrow \mathcal{S}p^{\bullet}(\mathrm{gr}_k(\I))\rightarrow \mathcal{S}p^{\bullet}(\mathfrak{q}_k)\rightarrow 0,$
for each $k\geq 0.$ Standard homological arguments, as in Proposition \ref{prop: Involutivity sequences}, give that $\mathcal{H}_{\delta}^{i,j+1}(ker(Tp^k_{k-1}|_{Z_k^0}))\simeq \mathcal{H}_{\delta}^{i+1,j}(\mathfrak{q}).$
Thus if $\mathfrak{\ell}$ is as in Proposition \ref{prop: KumLemma},
then there exists $\widetilde{Z}_{\mathfrak{\ell}}$ such that $\mathcal{H}_{\delta}^{i,j}(\mathfrak{q})=0$ over $\widetilde{Z}_{\mathfrak{\ell}}.$ This immediately implies that the polynomial (e.g. bounded) growth of $dim(\mathfrak{q}_k).$
\end{proof}
A more detailed description of the $\D$-Hilbert polynomial for the geometric quotient equation i.e. the ideal $\mathcal{I}^G$, is given in Proposition \ref{prop: HilbPolyForGeomQuot} below.

\subsection{$\D$-Hilbert and Quot functors}
\label{ssec: D-Hilb and Quot}
Having at our disposal a good notion of numerical polynomials $P\in \mathbb{Q}[t]$ for characteristic modules associated to $\D$-involutive $\D$-ideal sheaves, we can now proceed to construct the moduli space of such objects with \emph{fixed} $\D$-Hilbert polynomials equal to $P.$

We arrive at the main Definition \ref{defn: D-Hilb Functor} for the $\D$-Hilbert functor (\ref{eqn: D-Hilbert}) by  first constructing the $\D$-Quot functor. It is well-known that in the classical setting they are equivalent constructions \cite{Gro3}.\footnote{This fact is no longer true in the derived setting, as observed in \cite{CFK2,BKS,BKSY2}. We discuss this further in \cite{KSh2}.}
\begin{cons}
    \label{cons: DHilb}
    \normalfont 
    Fix a $\D$-smooth scheme of jets $Z=\mathrm{Spec}_{\mathcal{D}}(\mathcal{A}),$ and a numerical polynomial $P\in \mathbb{Q}[t].$
    Consider the assignment
    \begin{equation}
        \label{eqn: Classical D-Quot}
\underline{Q}_{\mathcal{D}_X,Z}^{P}:\mathrm{CAlg}_{\mathcal{D}_X}\rightarrow \mathrm{Set},\hspace{2mm}\mathcal{B}\mapsto \underline{Q}_{\mathcal{D}_X,Z}^{P}(\mathcal{B}),
\end{equation}
where for each commutative left $\D$-algebra $\mathcal{B}$, we set $\underline{Q}_{\mathcal{D}_X,Z}^{P}(\mathcal{B})$ to be the set of surjective $\mathcal{D}_X$-linear morphisms, $$\varphi:\mathcal{O}_{Z}\otimes_{\mathcal{O}_X}\mathcal{B}\twoheadrightarrow \mathcal{Q},$$  to a module $\mathcal{Q}$ of finite  $(\mathcal{O}_{Z}\otimes\mathcal{B})[\mathcal{D}_X]$-presentation with proper support relative to $\mathrm{Spec}_{\mathcal{D}_X}(\mathcal{B})$, modulo the equivalence relation: $\varphi_1\simeq\varphi_2$ if there exists a $\mathcal{D}_X$-isomorphism $\alpha:\mathcal{Q}_1\simeq \mathcal{Q}_2$ such that $\alpha\circ\varphi_1=\varphi_2.$ That is, the natural diagram
\[
\begin{tikzcd}
    \mathcal{O}_{Z}\otimes_{\mathcal{O}_X}\mathcal{B}\arrow[r,two heads,"\varphi_1"] \arrow[dr,two heads, "\varphi_2"] & \mathcal{Q}_1\arrow[d,"\alpha"]
    \\
    & \mathcal{Q}_2,
\end{tikzcd}
\]
commutes.
From the discussions so far (\ref{eqn: Classical D-Quot})  should be defined on the sub-category spanned by $\D_X$-involutive $\D_X$-schemes and will be denoted the same:
$$\underline{Q}_{\mathcal{D}_X}^{P}:\mathrm{CAlg}_{X}^{\mathrm{inv}}(\mathcal{D}_X)\simeq Sch_X^{aff}(\mathcal{D}_X)_{inv}^{op}\rightarrow \mathrm{Set}.$$

Fixing a numerical polynomial $P\in \mathbb{Q}[t],$ we characterize the $\D$-ideals via their regularizing properties of their symbols, using Proposition \ref{prop: Hilbert D-Polynomial}. 
In this way, we define a moduli functor $\underline{Q}_{\mathcal{D}_X,Z}^P$whose objects parameterize quotient $\D$-modules of the form $\varphi:\mathcal{O}_{Z}\twoheadrightarrow\mathcal{O}_{Z}/\mathcal{I}_X,$ for some $\D_X$-ideal sheaf $\mathcal{I}$ corresponding to a $\D$-involutive affine $\D_X$-scheme as in Definition \ref{defn: D-geom m-involutive}, whose $\D$-Hilbert polynomial $P_{\mathcal{I}}$ is equal to $P$. 

The equivalence relation $\alpha$ asserts there is an isomorphism $\mathrm{ker}(\varphi_1)=\mathrm{ker}(\varphi_2)$ of $\D_X$-ideals.
\end{cons}

\begin{rmk}
By definition, the functor (\ref{eqn: Classical D-Quot}) coincides with the $\D$-Hilbert functor. However, it can be defined more generally for a fixed vector $\D$-bundle or locally $X$-projective $\mathcal{O}_Z[\mathcal{D}]$-module $\mathcal{V}$, denoted $\underline{Q}_{\mathcal{D}_X,\mathcal{V}}.$ The corresponding numerical characterization is suggested in Remark \ref{rmk: Ranks and Chern class} e.g. by fixing the $\D$-Chern classes.
\end{rmk}
We enunciate the situation for the tautological vector $\D$-bundle $\mathcal{V}$ over $Z$ given by $\mathcal{O}_Z.$ 
\begin{defn}
\label{defn: D-Hilb Functor}
\normalfont 
Let $X$ be a proper scheme and fix a smooth affine $\mathcal{D}$-scheme $Z=\mathrm{Spec}_{\mathcal{D}_X}(\mathcal{A}).$ The functor of \emph{$\D_X$-involutive $\mathcal{D}$-geometric PDEs with fixed $\D$ Hilbert polynomial} is the assignment
\begin{eqnarray}
    \label{eqn: D-Hilbert}
\mathcal{H}\underline{\mathrm{ilb}}_{\mathcal{D}_X/Z}^{P}:\mathrm{Sch}_{\mathcal{D}_X}^{\mathrm{inv}}&\rightarrow& \mathrm{Sets},\nonumber
    \\
    \mathcal{Y}&\mapsto & \mathcal{H}\mathrm{ilb}_{\mathcal{D}_X/Z}^{P}(\mathcal{Y}),
\end{eqnarray}
where $\mathcal{H}\mathrm{iln}_{\D_X/Z}^{P}(\EQ):=\{\mathcal{T}\subset \mathcal{Y}\times Z,\text{involutive, }\mathcal{D}\text{-flat over } \mathcal{Y}, P_{\mathcal{I}_{\mathcal{T}}}=P\}$.
Functor (\ref{eqn: D-Hilbert}) is simply to be called the $\D$-\emph{geometric Hilbert functor} and the conditions explicitly mean $\mathcal{O}_{\mathcal{T}}$ is $\D$-flat over $\mathcal{O}_{\mathcal{Y}}$ and given by a $\D$-involutive $\D$-ideal sheaf and its restriction to $\{y\}\times Z$ has $\D$-Hilbert polynomial equal to $P.$
\end{defn}
Functor (\ref{eqn: D-Hilbert}) is defined on the sub-category spanned by $\mathcal{D}_X$-schemes determined by formally $\mathcal{D}$-integrable equations which are additionally $\D$-involutive.

\begin{rmk}
 If $\mathcal{B}$ is not $\D$-formally integrable and if $\mathcal{B}'$ 
is obtained from
it by Kuranishi's prolongation-projection theorem, by our prescription of dimension and numerical polynomials (thus of regularity), the numbers $p,d$ in Proposition \ref{prop: Hilbprop} change (in fact, decrease).  Consequently, the dimension of the support of the linearization sheaf is not constant nor does it tend to a constant during prolongation. 
\end{rmk}
We now prove the representability of the $\D$-Hilbert functor for $\D$-ideal sheaves by a $\D$-scheme. We \emph{assume} there exists a universal $\mathcal{D}$-flat family of closed sub-schemes of $Z$, 
$$\mathcal{U}\subset Z\times \mathrm{Hilb}_{\mathcal{D}_X,Z}^{P(t)},$$ 
with $P_{\mathcal{U}}=P,$ such that: for every test $\mathcal{D}_X$-algebra of finite-type $\mathcal{R}^{\ell},$ with corresponding finite-type $\mathcal{D}$-scheme $\mathrm{Spec}_{\mathcal{D}}(\mathcal{R}_X^{\ell}),$ and for every flat family as above, there exists a unique morphism 
$$\mathrm{Spec}_{\mathcal{D}}(\mathcal{R}_X^{\ell})\rightarrow \mathrm{Hilb}_{\mathcal{D}_X,Z}^{P(t)},$$ 
such that $\mathrm{Spec}_{\mathcal{D}}(\mathcal{R}_X^{\ell})\times_{\mathrm{Hilb}_{\mathcal{D},Z}^{P}}\mathcal{U}\subset Z\times \mathrm{Spec}_{\mathcal{D}}(\mathcal{R}_X^{\ell}).$ In particular, there is a universal $\D$-algebra quotient
$\mathcal{B}_{univ}:=\mathcal{O}_{Z\times \mathcal{H}ilb_{\D_X}^{P,\Lambda}(Z)}/\mathcal{I}_{univ},$
with universal ideal $\mathcal{I}_{univ}\subset \mathcal{O}_{Z\times \mathcal{H}ilb_{\mathcal{D}_X}^{P,\Lambda}(Z)}.$
Since the numerical polynomials for a $\D$-ideal sheaf and its characteristic module contain the same information, one may show these two moduli spaces are isomorphic. If we consider projectivization of the characteristic variety e.g. in $\mathbb{P}T^*X,$ then the ($\D$-)Hilbert scheme is representable by a ($\D$-) scheme.

\subsection{Douady-type Grassmannians for differential ideals}
\label{ssec: Douady}
Due to the pro-algebraic nature of $J_X^{\infty}E$, the $\D$-Hilbert functor is to be viewed as a pro-ind object. To treat the infinite-dimensional (pro-finite-dimensionality), we adapt a construction of \cite{Dou}.

\begin{prop}
\label{prop: Douady}
Fix $Z:=J_X^{\infty}E.$ There is ind-scheme of finite-type 
    $\mathrm{Grass}_{\D}(Z),$ given by
$$\mathrm{Grass}_{\D}(Z):=\varprojlim_k \mathcal{G}r_X(J_X^kE)\simeq \varprojlim_k\big(\bigcup_{d\geq 0}Quot(J^kE;d)\big),$$
where $Quot(J^kE;d)$ is the moduli space parameterizing $d$-dimensional $\mathcal{O}_X$-coherent quotients $\mathcal{O}(J^kE)\twoheadrightarrow Q.$
\end{prop}
\begin{proof}
Note $\mathrm{dim}(J_X^kE)=n+m\cdot C_k^{n+k}.$ For each $k\geq 0$, we set $\mathcal{G}r_X(J_X^E):=\bigcup_{d\geq 0}Quot(J_X^kE;d).$ We want to show that specifying a point in $\mathrm{Grass}_{\D}(Z)$, built out of a system of $\mathcal{G}r_X(J_X^kE)$, it suffices to give compatible systems of coherent subspaces $R_k\subset J_X^kE$ ($\mathcal{O}_{R_k}$ is coherent), arising from kernels of quotient sheaves $\mathcal{O}(J_X^kE)\rightarrow Q_k,$ of rank $d_k,k\geq 0.$ Then, we generate a subspace $R\subset J_X^{\infty}E$ where $R_k=R\cap J_X^kE\subset J_X^kE$ with the compatibilities $R_k\simeq R_{k+1}\cap J_X^kE$ under the natural projections $J_X^{k+1}E\rightarrow J_X^kE.$ This compatible system then corresponds to a compatible system of points in Quot-schemes $[Q_k]\in Quot(J_X^kE;d_k)$.
To this end, note for each $k>0,$ there exists a natural morphism 
$\phi_k:\mathcal{G}r_X(J_X^{k+1}E)\rightarrow \mathcal{G}r_X(J_X^kE),$ defined as follows.
Consider a quotient i.e. a point $[Q_{k+1}]\in Quot(J_X^{k+1}E;d_{k+1})$ for some $d_{k+1},$ represented as a surjection $\mathcal{O}(J_X^{k+1}E)\rightarrow Q_{k+1}$, and since $p^{k+1}_k:J_X^{k+1}E\rightarrow J_X^kE$ is a surjective morphism between smooth algebraic varieties (vector bundles over $X$), then $\mathcal{O}(J_X^kE)\hookrightarrow \mathcal{O}(J_X^{k+1}E)$ is an injection of $\mathcal{O}_X$-modules, and we define the image of $[Q_{k+1}]$ under $\phi_k$ to be 
$$\phi_k[\mathcal{O}(J_X^kE)\twoheadrightarrow Q_{k+1}]:=[\mathcal{O}(J_X^kE)\rightarrow Q_{k}],$$
via $(p_k^{k+1})^*$, where $Q_k$ is define as follows. Since the composition $q_{k,k+1}:\mathcal{O}(J_X^kE)\subset\mathcal{O}(J_X^{k+1}E)\twoheadrightarrow Q_{k+1}$ need not be surjective, we consider 
$ker(q_{k,k+1})$ and define $Q_k:=Q_{k+1}/ker(q_{k,k+1}).$ Defined in this way, $\phi_k$ is induced from the  morphism
$$\psi_k:Quot(J_X^{k+1}E)\rightarrow Quot(J_X^kE), [Q_{k+1}]\mapsto [Q_k:=Q_{k+1}/ker(q_{k,k+1})].$$
It follows there exists a compatible system, whose limit is a (pro-ind) scheme of finite-type.
\end{proof}
The representability of the $\D$-Hilbert functor (\ref{eqn: D-Hilbert}) is established by proving there exists a finite-type ind-subscheme of the $\D$-geometric Douady-type Grassmannian in Proposition \ref{prop: Douady}. Roughly, it is obtained via the compatible system of kernels obtained from $\phi_k,\psi_k$, e.g. $\mathcal{I}:=\varprojlim ker[\mathcal{O}(J_X^kE)\twoheadrightarrow Q_k],$ which stabilizes due to regularity assumptions (\ref{defn: D-Geometric Regularity}).

\subsection{Algebraic quotients and GIT}
\label{ssec: GIT}
Fix a $G$-algebraic $\D$-space $(G,Z).$ In subsection \ref{ssec: Differential algebra quotients}, we assumed that $G$ is algebraic. It is possible to remove this assumption by separating orbits in low level jets. Higher order jets are automatic, as the prolonged actions are always algebraic. 
The growth of the $\D$-Hilbert function for the geometric quotient equation is bounded.

\begin{prop}
\label{prop: HilbPolyForGeomQuot}
    Let $(G,Z)$ be a $G$-algebraic $\D$-space defined by a formally integrable $\D$-ideal $\mathcal{I}$. Consider the quotient $\D$-scheme $Z^G:=\{[Z^k/G^k]\}$ as in (\ref{eqn: RegularGeomQuot}) and its
   associated numerical Hilbert function $h(Z^G;k):=dim(\mathfrak{q}_k)$. Then there exists a polynomial for large $n$ i.e. $\D$-Hilbert polynomial $P_{\D}(\mathcal{I}^{\mathcal{G}},n)$ such that $h(Z^G) \simeq P_{\D}(\mathcal{I}^G)$ for $n\gg 0.$
    
\end{prop}
\begin{proof}
Note that the growth is controlled by the value of $\mathrm{dim}(\mathfrak{q}_k)$ at generic points $z_k\in Z^k.$ It is given by 
$P_{\D}(\mathcal{I}^G,k):=dim(\mathfrak{q}_k)=dim([Z^k/G^k])-dim([Z^{k-1}/G^{k-1}]),$
where by Proposition \ref{prop: GeometricQuotientEquation}, $dim(Q^k)$ is the transcendence degree of the field of rational $G_x^k$-invariant functions on $Z_x^k$ i.e. $\mathcal{O}_{rat}(Z_x^k)^{G_x^k}.$
By Proposition \ref{prop: Existence of geometric quotient involutive degree}, there exists a numerical polynomial $P\in \mathbb{Q}[t]$ such that 
$P_{\D}(\mathcal{I}^G,k)=P(k),\hspace{1mm} k\gg1.$
\end{proof}
The following corollary is immediate from Proposition \ref{prop: Existence of geometric quotient involutive degree} and Proposition \ref{prop: HilbPolyForGeomQuot}.
\begin{cor}
The $\D$-Hilbert polynomial of the $\D$-ideal of the geometric quotient is 
$P_{\D}(\mathcal{I}^{\mathcal{G}},z)=P_{\D}(\mathcal{I},z)-P_{\D}(G/\mathcal{I},z),$
for a uniquely determined numerical $\D$-Hilbert polynomial $P_{\D}(G/\mathcal{I}).$
\end{cor}
\begin{proof}
Note there exists a uniquely defined numerical function,
$$h_{G/Z}(k):=dim(ker(Tp^k_{k-1}|_{Z_k^0})),k\gg 1.$$
By a similar argument as in Proposition \ref{prop: HilbPolyForGeomQuot}, it is given by a polynomial $P_{\D}(G/\mathcal{I},z)$. Then, by diagram (\ref{eqn: Symbol analysis diagram}), the result follows.
\end{proof}
To prove there exists a GIT-quotient of the moduli problem, with all boundedness arguments, 
we restrict to $\mathfrak{M}_{\D}^{Sp-ss,inv}\subset \mathfrak{M}_{\D}^{inv}$, consists of those $\D$-ideals, involutive and formally integrable, which are also Spencer semistable. 
Given an algebraic pseudogroup $G$ acting on points of $\mathfrak{M}_{\D}^{inv}(P,n,m)$, we now describe the induced quotient
$$\mathfrak{M}_{\D}^{Sp-ss,inv}(P,n,m)/\!/\mathcal{G}.$$

\section{Existence and finiteness}
\label{sec: Main Result Proof}
This section is devoted to the proof of the main result of the paper. 

\begin{thm}
\label{MainTheorem}
    Let $X$ be a smooth $D$-affine variety and consider a locally free $\mathcal{O}_X$-module $E$ of finite rank. Let $Z:=J_X^{\infty}E$ denote the algebraic infinite jets of $E$, endowed with its natural $\D_X$-algebra structure.
Consider the moduli functor (\ref{eqn: D-Hilbert}) classifying formally integrable, involutive $\D_X$-ideal sheaves in $J_X^{\infty}E.$ Then it is representable by an ind-scheme of finite-type,
    $$\mathbf{Hilb}_{\D_X}(Z)\simeq \varinjlim_{(d,k)}\mathrm{Hilb}^{(d,k)}(J_X^kE),$$
    where $\mathrm{Hilb}^{(d,k)}(J^kE)$ is the scheme of algebraic ideals $I$ in $J_X^kE$ with Hilbert polynomial $d$ and order $\leq k.$ Moreover, for each $m\geq 0$, the locus $\mathbf{Hilb}_{\D_X}^{inv,m}(J_X^{\infty}E)\subset \mathbf{Hilb}_{\D_X}(J_X^{\infty}E)$ where $\mathcal{I}$ is Spencer $m$-regular defines a Zariski-open finite-type subscheme.
\end{thm}

We split the proof of the Theorem \ref{MainTheorem} into several steps:
\begin{itemize}
\item[-] First, we discuss global generation and prove boundedness in Proposition \ref{prop: Boundedness 1}). 

\item[-] We prove a PDE-analog of the Lefschetz hyperplane theorem in Proposition \ref{prop: Inv Res}, ensuring consistency of Spencer-regularity along non-characteristic sub-manifolds.

\item[-] Introducing a dimension filtration (implicit from Douady's construction in Sect. \ref{ssec: Douady})) we prove representability in Proposition \ref{prop: HilbRep}.
\end{itemize}
A preliminary fact.

\begin{prop}
Fixing a (global) section $s_d\in \mathcal{O}_{T^*X}(d),$ induces a surjective map
$\mathcal{O}_{T^*X}(d)\otimes_{\mathcal{O}_{T^*X}}\mathcal{C}h_k\rightarrow \mathcal{C}h_{k+d},$ on global sections for all $k\geq \mathrm{Reg}_{\mathcal{D}}(\EQ).$ 
Moreover, $\mathcal{C}h$ is generated by its global sections.
\end{prop}
\begin{proof}
    The first part can be seen by noting that multiplication with some $\xi_j$ generates a morphism of quotient rings by Proposition \ref{prop: Mult by microlocal coordinates}. There exists $p\in \mathbb{Z},$ for which $\mathcal{C}h(\geq p)$ is generated by $\mathcal{C}h(p)$ and 
    $dim\big(\mathcal{C}h(r)\big)=\mathbf{h}^{\mathcal{Y}}(r),\forall r\gg p.$
    Since prolongations of involutive symbols are again involutive i.e.
    $\alpha_{k+1}^{(i)}=\sum_{\ell=i}^{n}\alpha_k^{(\ell)},$ gives by induction
$\alpha_{k+r}^{(\ell)}=\sum_{j=\ell}^{n}C_{r-1}^{r+j-\ell-1}\cdot \alpha_k^{(j)},$
and then we simply notice that 
$rank\big(\mathcal{C}h_{k+r}\big)=\sum_{\ell=1}^{n}C_{r}^{r+\ell-1}\cdot \alpha_k^{(\ell)},$
where the right-hand side agrees with the Hilbert polynomial for the non-linear PDE system $\EQ.$ The second statement follows from finite cogeneration of involutive symbolic systems and from the fact that $X$ is $\mathcal{D}$-affine, so the corresponding $\D$-module is indeed generated by its global sections.
\end{proof}


\subsection{Boundedness}
\label{ssec: Boundedness}
Consider an $\mathcal{O}_X$-submodule of $\mathcal{O}(J_X^kE)$, let $M_k^*$ denote its sections. Recall sections of $J_X^kE$ are dual to $\D_X^{\leq k}\otimes_{\mathbb{C}}E^*.$ Consider a compliment $M_k^{\perp}$ in this dual and put $\M_k:=\D_X^{\leq k}\otimes_{\mathbb{C}}E^*/M_k^{\perp},$ the quotient. One may consider prolongations, and set  $\M=\{\M_k\}$ and similarly for $M^{\perp}.$ Consider the smallest $\ell,m\in \mathbb{N}$ such that involutivity is imposed. For $N_{\ell}:=M_m^{\perp}\cap (\D_{X}^{\leq \ell}\otimes_{\mathbb{C}}E^*),$ and restricting to an open dense subset $U$ of $X$, if $q:=\ell-k$, then $\mathrm{pr}_q(M_k^*)$ is free and involutive. By \cite{Ma3}, the integers $m,\ell$ are uniformly bounded by some $\rho=\rho(k,rankE,n).$ See also the over-estimate \cite{Sw}.

Under Spencer regularity assumptions (\ref{defn: D-geom m-involutive}), this applies to families of $\D$-schemes. Namely, using the notion of $\mathcal{D}$-geometric regularity we have an analog of Mumford's boundedness result.
\begin{prop}
\label{prop: Boundedness 1}
Consider a family $\{\EQ_i\}_{i\in I}$ of Spencer regular and formally integrable algebraic non-linear PDEs imposed on sections of a $\mathrm{rank}(E)=m$ bundle $E\rightarrow X$ over a $\D$-affine $k$-scheme $X$ of dimension $n$. Then the family is bounded if and only if the set of $\mathcal{D}$ Hilbert polynomials $\{P_{\EQ_i}\}$ is finite and there is a uniform bound on the involutivity degrees i.e.
    $\mathrm{Reg}_{\mathcal{D}_X}(\EQ_i)\leq \rho,\forall i\in I.$
\end{prop}
This result essentially gives the appropriate `flattening stratification' \cite{Gro3}; it identifies sub-loci of the $\D$-Hilbert scheme which are better behaved, consisting of those $\D$-ideals which are Spencer $m$-regular for a fixed $m$. In particular, by (\ref{defn: D-geom m-involutive}) (ii), they are differentially generated from a fixed finite order.

\begin{proof}
For a given $\D$-scheme $Z^{\infty}$ in the family, fix $\ell\geq 0$ and consider the family of projections of $\mathrm{pr}_mZ_k$ with $\ell\leq m+k$ in $J_X^{\ell}E.$ Then, as $m\rightarrow \infty$ the family is decreasing and stationary. Let $m=k+p\geq \ell$ be the smallest integer such that the projection in $\mathcal{I}_{\ell}$ of $\mathrm{pr}_pZ_k\subset J_X^mE$ is equal to the closed subscheme $F^{\ell}Z^{\infty}.$ This always holds over a Zariski dense open subset $U\subset X.$ Malgrange proved that the integer $m$ is bounded by a function $\rho=\rho(k,\mathrm{dim}X).$
Thus, by assumption that the family is bounded, there exists a finite set $\{\rho_i\}_{i\in I}$ of bounds, obtained this way. Looking at the symbolic systems of each $\D$-scheme, we obtain a family of homogeneous ideals defining the characteristic modules, and let $\mathfrak{D}_i,i\in I$ be a bound for the elements of a Gröbner basis for the lexicographic order. Upon choosing generic coordinates, there exists by \cite[Theorem 3.4, Theorem 3.5]{Ma2}, an upper bound $\mathfrak{D}_i(n,d_i)$ for $\mathcal{I}$ generated by elements of degree $d_i,i\in I.$
Moreover, there exists a smallest $\ell_i,i\in I$ such that the symbolic module for PDE $Z_i^{\infty}$ is $\ell_i$-involutive. It is bounded by an integer $\gamma_i$ (see \cite{Sw}, or Appendix \ref{Appendix Sweeney}, and Theorem 3.4 \emph{loc.cit.}). Further, there exists a minimum $r_i$ for each $i\in I$ such that the $\ell_i$'th prolongation is contained in that of degree $r_i.$ This integer is bounded as well.

First, notice the general fact that if $\mathcal{N}$ is a symbol co-module for an involutive PDE, then for all $k\geq 0$ and for all $1\leq p\leq \mathrm{dim}(X),$ 
$$H_{\delta}^{p,q}(\mathcal{N})^*\simeq H_{p,q}(\mathcal{N}^*)\simeq H_{p-1,q+1}(\mathcal{N}^0),$$
where $\mathcal{N}^0$ is the annihilator.

Now, let $\mathcal{N}$ be a $\mathcal{C}:=\mathrm{Sym}(V^*)$-comodule and $\mathcal{M}$ be a $\mathrm{Sym}(V)$-module. Here, $V$ is the $n$-dimensional vector space which may be thought of as corresponding to the tangent sheaf $\Theta_{X,x}$ for each $x\in X.$
From Definitions \ref{defn: D-Geometric Regularity} and \ref{defn: D-geometric generalized involutivity}, we have involutivity at degree $k_0$ if it is $\mathrm{dim}(V)$-acyclic at degree $k_0.$
By duality, a graded $\mathrm{Sym}(V)$-module $\mathcal{M}$ is involutive at degree $k_0$ if 
$$H_{p,q}(\mathcal{M})=0,\forall k\geq k_0,0\leq p\leq \mathrm{dim}(V).$$

A graded ideal $\mathcal{I}$ is $k$-regular if its $i$-th syzygy module is generated by elements of degrees $\leq k+i.$ It follows that
$$\mathrm{reg}(\mathcal{I})=k\Rightarrow H_{p,q}(\mathcal{I})=0,q>k.$$
By our assumptions, each $\EQ_i$ is a $\mathcal{D}$-involutive PDE with corresponding $\mathcal{D}$-algebra of functions $\mathcal{B}_i=\mathcal{A}/\mathcal{I}_i$ with $\mathcal{A}$ our polynomial ring in derivatives of dependent variables for which $\mathcal{I}_i$ is $k_i$-regular:
$$H_{p,q}(\mathcal{B}_i)=0,\forall r\geq k_i,$$
and therefore by Proposition \ref{prop: Unifying},
$\mathrm{reg}(\mathcal{I}_i) \equiv \mathrm{deg}^{\mathrm{inv}}(\mathcal{B}_i).$
We may also apply Cartan's test \cite{C,C2}, which states there exists an integer $m_0\geq 0$ such that $H_{\delta}^{p,q}(\mathcal{N})=0,\forall q\geq m_0,0\leq p\leq \mathrm{dim}(X).$
By our assumptions on involutivity, we have that 
$\mathcal{M}$ is a finitely generated graded polynomial module, so by Hilbert's syzygy theorem there exists an integer $m_0\geq 0$ such that 
$H_{p,q}(\mathcal{M})=0,\forall q\geq m_0,0\leq p\leq \mathrm{dim}(X).$
Moreover, 
$H_{\delta}^{0,q}(\mathcal{N})=0,$ and the dimension of $H_{\delta}^{1,q-1}(\mathcal{N})$ equals the dimension of the quotient of the prolongation of $\mathcal{N}_{q-1}$ by $\mathcal{N}_q$ for every $q>0.$
\end{proof}
\subsubsection{Involutive restriction}
Given the $\D$-affine variety $X$ as before, the existence of non-characteristic sub-spaces follows from an analog of Noether normalization lemma applied to characteristic varieties defined by symbol ideals e.g. (\ref{eqn: Filtered Char ideal}). The set of all non microcharacteristic subspaces of $T^*X$ of fixed dimension define an open dense subset of the corresponding Grassmannian of $T^*X$. A subvariety $U\subset X$ is said to be non-microcharacteristic if it is when viewed, via the zero-section, as a sub-space of $T^*X.$
\begin{prop}
\label{prop: Inv Res}
Let $\mathcal{I}_X$ be a $\D$-involutive $\D$-ideal sheaf. For $j:H\subset T^*X$ a non-microcharacteristic analytic subset, denote by $\mathcal{I}_{H}$ the $\D$-ideal obtain by non-characteristic restriction. Then $\mathrm{Reg}_{\mathcal{D}}(\mathcal{I}_X)=m$ if and only if $\mathrm{Reg}_{\mathcal{D}}(\mathcal{I}_H)=m.$  
\end{prop}
\begin{proof}
We have the ideal sequence
\begin{equation}
    \label{eqn: Ideal sequence of H}
0\rightarrow \EuScript{J}_H\hookrightarrow \mathcal{O}_{T^*X}\rightarrow \mathcal{O}_H\rightarrow 0,
\end{equation}
which is a sequence of sheaves of filtered algebras by polynomial degree, which for homogeneous degree $k$ reads
$0\rightarrow \EuScript{J}_{H}(k)\hookrightarrow \mathcal{O}_{T^*X}(k)\rightarrow \mathcal{O}_{H}(k)\rightarrow 0.$
In particular, consider $k=1,$ and the dual sequence, considered stalk-wise, for a 
$k$-rational point $x$ of $X$ as
$$0\rightarrow \mathcal{V}_x^*\rightarrow \Theta_{X,x}^{\vee}\simeq\Omega_{X,x}^1\xrightarrow{\nu_{x}}\mathcal{W}^*\rightarrow 0,$$
where we use the notation: for $W\subset T_xX,$ we have $W^*$ is its dual, and $V^*=\mathrm{Ann}(W)\subset T_x^*X.$

Recalling $\mathcal{N}_k$ and $\mathcal{C}h_k^{\mathcal{Y}},$ as in (\ref{eqn: k-th symbol module}), and (\ref{eqn: k-characteristic module}), respectively. Viewed as a submodule of 
$\mathrm{Sym}^k(\Theta_{X,x}^{\vee})\otimes \mathcal{B},$ tensoring with (\ref{eqn: Ideal sequence of H}) gives $H$-restriction, $j^*\mathcal{C}h_{k}=\mathcal{C}h_k|_{H}\subset \mathrm{Sym}^k(\mathcal{W}_{x}^*)\otimes \mathcal{B}.$ It is the image of $\mathcal{C}h_k$ by the restriction $\mathrm{Sym}^k(\Omega_X^1)\otimes_{F^k\mathcal{B}}\mathcal{B}\twoheadrightarrow \mathrm{Sym}^k(\mathcal{W}_x^*)\otimes \mathcal{B},$ induced by $\nu_x,$ for $x\in X.$
Omitting reference to $x\in X$ from notation, we have on one hand a commutative diagram 

\[
\begin{tikzcd}
    \mathcal{C}h_{k+1}\arrow[d]\arrow[r] & \mathcal{C}h_k\otimes\Omega_{X}^1\arrow[d]
    \\
    \mathcal{C}h_{k+1}|_{H}\arrow[r] & \mathcal{C}h_k|_H\otimes W^*,
\end{tikzcd}
\]
as well as, for the $\ell$-th Spencer sequence, the isomorphism of first $\ell-\mathrm{I}+1$ terms:
\[
\begin{tikzcd}
    \mathcal{C}h_{\ell}\arrow[d]\arrow[r] & \mathcal{C}h_{\ell-1}\otimes W^*\arrow[d]\arrow[r] & \mathcal{C}h_{\ell-2}\otimes \wedge^2W^*\arrow[d]\arrow[r] & \cdots
    \\
    \mathcal{C}h_{\ell}|_{H}\otimes W^*\arrow[r] & \mathcal{C}h_{\ell-1}|_{H}\otimes W^*\arrow[r] & \mathcal{C}h_{\ell-2}|_{H}\otimes \wedge^2 W^*\arrow[r] & \cdots
\end{tikzcd}
\]
giving 
$$H^{*,*}(\mathcal{C}h)\simeq H^{*,*}(\mathcal{C}h|_{H}).$$
Here, $I$ is the degree of involutivity.

For $i<I$ it suffices to consider that, 
$\mathcal{C}h_i\simeq Sym^i(\Omega_X^1)\otimes\mathcal{B}_X,$ and $\mathcal{C}h_{i}|_{H}\simeq Sym^i(W^*)\otimes i^*\mathcal{B},$ where $i^*\mathcal{B}\simeq \mathcal{B}_H,$ we obtain that:
\[
\begin{tikzcd}
\arrow[r]
\mathcal{C}h_{I}\otimes\wedge^{q-1}W^*\arrow[d]\arrow[r] & Sym^{I-1}(\Omega_X^1)\otimes\mathcal{B}\otimes \wedge^qW^*\arrow[d]\arrow[r] & \cdots\\
    \arrow[r] \mathcal{C}h_{I}|_{H}\otimes\wedge^{q-1}W^*\arrow[r] & Sym^{I-1}W^*\otimes i^*\mathcal{B}\otimes \wedge^q W^*\arrow[r] & \cdots
\end{tikzcd}
\]
and the only contributions to cohomology are $H^{I,q-1}$ and $H^{I-1,q}.$ 
Calculating the Spencer cohomology is achieved by standard spectral sequence arguments. Namely, introduce a filtration by powers of $\Omega_X^1$ in $\wedge^q\Omega_X^1$, denoted $G_p^q:=\wedge^p\mathcal{V}^*\wedge \wedge^{q-p}\Omega_X^1,$ so that $G_{q}^q\subset G_{q-1}^q\subset \cdots\subset G_0^q:=\wedge^q\Omega_X^1.$
There is an induced filtration, given for each $\ell\geq 0,$ by
$$Fl_{\EQ}^{p,q}:=\mathcal{C}h_{\ell-p-q}^{\EQ}\otimes G_{p}^{p+q}\simeq \mathcal{C}h_{\ell-p-q}^{\EQ}\otimes \bigwedge^p\mathcal{V}^*\wedge\bigwedge^q\Omega_X^1.$$
It satisfies $Fl_{\EQ}^{p+1,q-1}\subset Fl_{\EQ}^{p,q},$ and its successive quotients, giving the $0$-th page are
$E_0^{p,q}:=Fl_{\EQ}^{p,q}/Fl_{\EQ}^{p+1,q-1}\simeq \mathcal{C}h_{\ell-p-q}^{\EQ}\otimes \wedge^p\mathcal{V}^*\otimes \wedge^q\mathcal{W}^*.$
The differential acts according to $W$ and one may show that 
$$E_1^{p,q}\simeq H^{\ell-p-q,q}(\mathcal{C}h)\otimes \wedge^p\mathcal{V}^*.$$
From these arguments, one may show that 
$$H^{0,q}(\mathcal{C}h)\simeq \bigoplus_{q_0\geq 0}H^{0,q_0}(\mathcal{C}h|_{H})\otimes \wedge^{q-q_0}(V^*).$$ 
To establish the full vanishing except for the position where involutivity occurs, one may proceed inductively on $\ell\geq k$ to show $H^{i,\ell-i}=0$ except at $i=k-1$ and if this holds for $\mathcal{C}h$ it holds for $\mathcal{C}h|_{H}.$ It is a long check so we omit the details in full but mention that every differential $d_{r}^{p,q}$ for $r>0$ is trivial except for the case of $d_{1}^{p,0}$ which happens when $\ell-p<k.$
\end{proof}
By \cite{G,G2}, we note the following.
\begin{rmk}
The obstructions for formal
integrability belong to the groups $H^{p,2}$, but after finitely many prolongations
the system becomes $m$-acyclic and $m$-involutive.
\end{rmk}
Proposition \ref{prop: Inv Res} states that the degree where this occurs is the same for characteristic modules and all their non-characteristic restrictions.

\begin{prop}
\label{prop: Involutive restriction is involutive}
For all $i \geq ord$ if $H^{i,j}(\mathcal{C}h)=0$ for all $j\geq n-m,$ then for every non-characteristic $H\subset T^*X$ whose dual is denote by $V$,
with restricted symbol $\mathcal{C}h|_{H},$ one has 
$H^{i,j}(\mathcal{C}h|_{H})=0,\forall j\geq n-m.$
\end{prop}
Denote for each non-negative integer $k$, and each classical solution $\varphi$ the restriction map
\begin{equation}
    \label{eqn: Restriction map}
    \mathrm{res}_{k}:H^0(T^*X,\mathcal{C}h_{k,\varphi})\rightarrow H^0(V,\mathcal{C}h_{k,\varphi}|_{V}).
    \end{equation}

\begin{prop}
    Suppose $H\subset T^*X$ is non-characteristic and for each classical solution $\varphi$ the restriction  $\mathcal{C}h_{k,\varphi}|_{V}$ is $m$-involutive. Then, if for some $m_0\geq m,$ map \emph{(\ref{eqn: Restriction map})} defining restriction $\mathrm{res}_{m_0}$ is surjective, then $\mathrm{res}_{p}$ is surjective for all $p\geq m_0.$
    \end{prop}
\begin{proof}
By assumption, the map  $\mathrm{res}_{k}:H^0(T^*X,\mathcal{C}h_{k,\varphi}))\rightarrow H^0(V,\mathcal{C}h_{k,\varphi}|_{V})$ is surjective and since $\mathcal{C}h_{k}|_{V}$ is $m$-regular, the natural morphism 
$$H^0(V,\mathcal{O}_V(1))\otimes H^0(V,\mathcal{C}h_{k}|_{V})\rightarrow H^0(V,\mathcal{C}h_{r+1}|_{V}),$$
is surjective as well. It follows that $\mathrm{res}_{r+1}$ corresponding to (\ref{eqn: Restriction map}) is surjective. The statement follows similarly.
\end{proof}

\subsection{Representability}
To conclude the proof of Theorem \ref{MainTheorem}, we construct a directed system defining an ind-scheme $\mathrm{Hilb}_{\D_X}^P(Z)=\varprojlim_k \mathrm{Hilb}_{\mathcal{D}_x}^{P,\leq q}(Z).$ It is filtered colimit of schemes of finite type, with each piece corresponding to data truncated at a finite jet level. This inverse system captures the behaviour of the entire formally integrable system as $q\rightarrow \infty.$

\begin{prop}
\label{prop: HilbRep}
The $\D$-Hilbert functor with fixed numerical polynomial $P\in \mathbb{Q}[t]$ that classifies formally integrable algebraic differential systems with fixed (regular) symbolic data is representable by an ind-scheme $\mathrm{Hilb}_{\D_X}(Z),$ equipped with a compatible $\D_X$-action. It is functorial for $S$-families of differential systems, for any non-characteristic fibration $f:X\rightarrow S$ i.e. the fibers $f^{-1}(s)$ are non-characteristic, for each $s\in S.$
\end{prop}
\begin{proof}
Consider the dimension filtration denoted $_{\leq d}\mathcal{H}ilb_{\D_X}(Z)\subset \mathcal{H}ilb_{\mathcal{D}_X}(Z)$ with $S$-points
$$\{\mathcal{I}_S|dim(Char(\mathcal{O}_{Z_s}/\mathcal{I}_s)\leq d, s\in S\}.$$
We consider the geometric points $s\in S.$ This is well-defined since $s\in S\mapsto dim(Char(\mathcal{B}_s))$ is upper semi-continuous.
The corresponding loci of characteristic dimension $d$, is the open strata given by the quotient pieces to the filtration,
\begin{equation}
\label{eqn: d-open loci}
\mathrm{Hilb}_{\D_X}^d(Z):=_{\leq d}\mathcal{H}ilb_{\D_X}(Z)/\hspace{1mm}_{\leq d-1}\mathcal{H}ilb_{\mathcal{D}_X}(Z).
\end{equation}
Set 
$$\mathcal{H}ilb_q^{\leq P}(Z):=\{\mathcal{O}_{Jet^q(E)}-\text{ideals } \hspace{1mm} I_q, \text{ algebraic and whose }$$
$$\text{differential closure }\overline{I}_q^{diff}\subset\mathcal{O}_Z \text{ has Hilbert polynomial }\leq P\}.$$
The ordering means $P_{I_q}(t)\leq P(t),$ for all integers $t.$
    The differential closure of $I_q\subset \mathcal{O}_{Jet^q(E)}$ spanned by equations $f_1,\ldots,f_r$ is spanned by $\{D^{\sigma}(f_A)|f_A\in I_q,A=1,\ldots,r |\sigma|\geq 0\}.$ We write $\big<\mathcal{D}_X\bullet I_q\big>\subset \mathcal{O}_Z$ for simplicity. It is the smallest left $\D_X$-ideal in $\mathcal{O}_Z$ such that
    $\overline{I_q}^{diff}\cap \mathcal{O}_{Jet^q(E)}\simeq I_q.$
This is just the colimit of all prolongations
$\mathrm{Pr}_r(I_q):=I_{q+r}.$ The sequence of prolongations corresponds to a filtered system of closed sub-schemes
$Y_q\subset Y_{q+1}\subset \cdots\subset Y^{\infty}.$
By formal integrability, there exists $r_0\geq 0$ such that for all $r\geq r_0,$ it holds that $\mathrm{Pr}_r(I_q)=I_{q+r}\simeq I_{q+r+1}\cap \mathcal{O}_{Jets^{q+r}(E)}.$
In particular, $I_{q+r_0}$ is closed under total derivatives and $\overline{I}_q^{diff}$ may be identified with 
$I_{q+r_0}\cdot \mathcal{O}_{Jets_X^{\infty}E}.$
There is a natural morphism 
$\mathcal{H}ilb_q^{\leq P}(Z)\hookrightarrow Quot_P(\mathcal{O}_{Jets^qE/X}),$ where the Quot-scheme is obtained as in Proposition \ref{prop: Douady}. This acts by $I_q\mapsto \big(\mathcal{O}_{Jets^q(E)}\twoheadrightarrow \mathcal{B}_q:=\mathcal{O}_{Jets^q(E)}/I_q\big).$ It is an isomorphism and by \cite{Gro3}, each Quot-functor is representable by scheme of finite-type. Thus 
$$\mathcal{H}ilb_{\D_X}^{\leq P}(Z)=\varinjlim_q\mathcal{H}ilb_q^{\leq P}(Z),$$
is a filtered colimit of schemes and therefore a finite-type ind-scheme. Since, prolongation imposes linear equations on the Quot scheme, it defines a closed subscheme of $\mathrm{Grass}_{\D}(Z)$, as in Proposition \ref{prop: Douady}. This corresponding limit identifies the desired $\D_X$-Hilbert scheme. It $\D$-structure is encoded in the differential closure of its points i.e. they correspond to $\D$-stable ideals.
\end{proof}

\subsection{Zariski tangent spaces}
\label{ssec: Zariski tangent spaces}
By Theorem \ref{MainTheorem}, we know the $\D$-Hilbert moduli functor is well-behaved in the $\D$-geometric setting. This suggests we can make sense of its Zariski tangent space. 
First, by extending results of \cite[Section 2]{BD},  natural differential complexes may be associated to $\D$-ideal sheaf sequences e.g. to points $[\mathcal{I}]$ of the $\D$-Hilbert scheme. Indeed, for such a $\D$-ideal, we have
 $$\ldots\rightarrow \Omega_{\mathcal{B}}^{(-2)}\rightarrow\mathcal{I}/\mathcal{I}^2\rightarrow \Omega_{\mathcal{A}}^1\otimes\mathcal{B}\rightarrow \Omega_{\mathcal{B}}^1\rightarrow 0.$$
 Let $\EQ\hookrightarrow Z$ be the corresponding embedding with $Z$ ambient $\D$-smooth $\D$-scheme.
The universal $\D$-linear de Rham differential is denoted $d_Z^v:\mathcal{O}_Z\rightarrow \Omega_Z^1,$ and consider the composition,
$$d_Z^v|_{\mathcal{I}}:\mathcal{I}\hookrightarrow \mathcal{O}_Z\rightarrow \Omega_Z^1,$$ 
as a morphism of $\mathcal{O}_Z[\mathcal{D}_X]$-modules. Since
$\Omega_{\mathcal{Y}}^1\leftarrow i_{\EQ}^*\Omega_{Z}^1\leftarrow \mathcal{I}/\mathcal{I}^2,$
we get an $\mathcal{O}_{\EQ}[\mathcal{D}_X]$-linear morphisms 
$d_{\EQ/Z}:\mathcal{I}/\mathcal{I}^2\rightarrow i_{\EQ}^*\Omega_{Z}^1,$ such that
$\mathrm{coker}(d_{\EQ/Z}^v)\simeq \Omega_{\EQ}^1.$
One may set
\begin{equation}
\label{eqn: Truncated Cotangent}
\mathbb{L}_{\EQ}^{\bullet}:=\mathcal{I}/\mathcal{I}^2\xrightarrow{d_{\EQ/}^v}i_{\EQ}^*\Omega_{Z}^1.
\end{equation}
The object (\ref{eqn: Truncated Cotangent}) is the \emph{truncated} cotangent complex. It is concentrated in degrees $[-1,0].$ 
\begin{rmk}
Approximating cotangent complexes via perfect complexes in a suitable derived category serves the basis of a perfect-obstruction theory for PDEs defined by $\D$-ideal sheaves. This is beyond the scope of the current work, and is only mentioned for completeness.\end{rmk}
We now sketch how to impose further geometric constraints on the $\D$-ideals $[\I]\in\mathcal{H}ilb_{\D}^P(Z)$ whose associated Spencer cohomologies as well as cohomologies of $\D$-module solution functors (e.g. de Rham cohomologies), are finite-dimensional. 

\subsubsection{Microlocal stratifications}
We restrict attention throughout to non-singular algebraic PDEs, which we interpret more precisely by imposing microlocal conditions on the characteristic varieties of the corresponding $\mathcal{D}_X$-ideal sheaves. Constraining the geometry of these varieties enables control over singularities and induces natural stratifications on the $\mathcal{D}_X$-Hilbert stack.

While a full treatment requires the $\infty$-categorical functoriality of the $\mathcal{D}$-cotangent complex and is deferred to the sequel \cite{KSh2}, we outline here the key microlocal input. In particular, this allows us to isolate loci in the $\mathcal{D}_X$-Hilbert moduli space corresponding to elliptic equations, using the geometric criterion for ellipticity via characteristic varieties as developed in \cite{KSY}.

Let $U_{\alpha}\subset \mathcal{C}har(\mathcal{I})$ be an irreducible component of the characteristic variety of $\mathcal{I}$, and put
$$mult_{U_{\alpha}}(\mathcal{C}h):=\sum_{U_{\alpha}\subset \mathcal{C}har(\mathcal{I})}mult_{S}(\mathcal{C}h)[U_{\alpha}].$$
In families, let $\Lambda\subset T^*(X/S)$. Assume it is involutive, but note the fibers are not necessarily involutive since $\Lambda_s:=\Lambda\cap T^*(X/S)_s\subset T^*(X_s),$
for $s\in S,$ is not always a proper intersection. There exists a relative analog $\mathcal{C}har^{rel}(\I)$, defined as in Subsect. \ref{sssec: D-Geometric Microcharacteristic Varieties}, by adapting constructions of \cite{SS,SS2}.

\begin{prop}
Assume $\Lambda$ is smooth over $S$. Then $\Lambda_s\subset T^*X_s$ is either empty or involutive.
\end{prop}
\begin{proof}
Let $Char^{rel}(\mathcal{I})\cap W$ be smooth over $U\subseteq S$ with $U$ open and $W$ an open neighbourhood of a point in $Char^{re}.$ By relative smoothness, 
$(Char^{rel}(\mathcal{I})\cap W)\cap T^*X_s$
is involutive in $T^*X_s$ for all $s\in S$ implying $dim\geq dimX-dimS.$ Then, one may use upper semi-continuity of fibers of $Char^{rel}.$
\end{proof}

Let $\Lambda \subset T^*X,$ be a closed conic subset and define $\mathcal{H}ilb_{\D_X}^{\Lambda}(Z),$ which parameterizes $\D$-ideal sheaves $\mathcal{I}$, with 
$Char_{\D}(\mathcal{I})\subset \Lambda.$
This defines a closed subfunctor. 
\begin{defn}
\label{defn: Non-singular Hilb}
\normalfont
Let $\Lambda\subset T^*X$ be a closed coisotropic conic subset, such that $\pi(\Lambda)\subset T^*X$ is smooth, with $\pi:T^*X\rightarrow X.$ 
The $\D$-Hilbert functor of $\D$-ideals $\mathcal{I}$ with $\D$-Hilbert polynomial $P\in \mathbb{Q}[t],$ which are \emph{non-singular relative to $\Lambda$}, is $\mathcal{H}ilb_{\mathcal{D}_X}^{P,\Lambda}(Z):=\mathcal{H}ilb_{\mathcal{D}_X}^{P}(Z)\cap \mathcal{H}ilb_{\mathcal{D}_X}^{\Lambda}(Z).$
\end{defn}
Consider (\ref{eqn: Truncated Cotangent}). In \cite{KSY}, the derived $\D$-characteristic variety was defined as the cohomological support of the tangent complex, denoted by $\mathbf{Char}_{\D}(\mathcal{I}).$

\begin{prop}
Consider a $\D$-ideal $\mathcal{I}$ and assume the cotangent complex of the associated $\D$-subscheme is of the form (\ref{eqn: Truncated Cotangent}). Let $\mathcal{C}\mathrm{har}_{\D}(\I)$ be the characteristic variety, given by Proposition \ref{prop: Char inclusions}. Then, 
$\mathcal{C}\mathrm{har}_{\D}(\I)\subseteq \mathbf{Char}_{\D}(\mathcal{I}).$
\end{prop}
\begin{proof}
Clearly from (\ref{eqn: Truncated Cotangent}) there exists a
canonical map $\mathbb{L}_{\EQ}^{\bullet}\rightarrow \mathcal{H}_{\mathcal{D}}^0(\mathbb{L}_{\EQ}^{\bullet})\simeq \Omega_{\EQ}^1.$ The result is immediate since
$$\mathbf{Char}_{\D}(\mathcal{I})\simeq \mathrm{supp}\big(\mathcal{H}_{\D}^0(\mathbb{L}_{\mathcal{Y}})\big)\cup \bigcup_{i\neq 0}\mathrm{supp}\big(\mathcal{H}_{\D}^i(\mathbb{L}_{\mathcal{Y}})\big).$$

\end{proof}
These results serve to allow microlocal control over singularities, while also providing a means to consider only certain classes of $\D$-ideals whose corresponding solution functors have finite-dimensional cohomologies. For more general non-linearities, one should work with $s$-microcharacteristics (see Reminder \ref{rem: s-microchars}) with appropriate modifications to Definition \ref{defn: Non-singular Hilb}.
Microlocal constraints can imply finite-dimensionality of deformation-obstruction spaces of the corresponding $\D$-modules given by the Zariski tangent sheaves, as we now compute in the second main result. 

\begin{thm}
    \label{MainTheorem2}
Consider Theorem \ref{MainTheorem}. The Zariski tangent space at a point $[\mathcal{I}]$ controlling first-order deformations of $\mathcal{I}$ as a $\D$-ideal, is given by
    $$T_{[\mathcal{I}]}\mathbf{Hilb}_{\D_X}(J_X^{\infty}E)\simeq Hom_{\D_X}(\mathcal{I},\mathcal{O}(J_X^{\infty}E)/\mathcal{I}).$$
The space of obstructed deformations at $[\mathcal{I}]$ is,
$$\mathcal{O}bs_{[\mathcal{I}]}=Ext_{\mathcal{O}(J_X^{\infty}E)/\mathcal{I}\otimes_{\mathcal{O}_X}\D_X}^1(\mathcal{I}/\mathcal{I}^2,\mathcal{O}(J_X^{\infty}E)/\mathcal{I}\otimes_{\mathcal{O}_X}\D_X),$$
which is naturally endowed with a good $\mathcal{O}(J_X^{\infty}E)[\D]$-filtration whose associated graded algebra is isomorphic to $\bigoplus_{q< \mathrm{Reg}_{\D}(\mathcal{I})}\mathcal{H}_{Sp}^{1,q}\big(\mathrm{gr}(\mathcal{I})\big).$
\end{thm}
\begin{proof}
Fixing $[\mathcal{I}]$, viewed geometrically as a closed $\D$-subscheme $\mathrm{Spec}_{\D}(\A/\mathcal{I})\hookrightarrow \mathrm{Spec}_{\D}(\A),$ consider the defining exact sequence $\IAB$, with $\mathcal{B}:=\A/\mathcal{I}.$ Consider the conormal sequence in $\mathcal{A}/\mathcal{I}[\D]$-modules (essentially defining (\ref{eqn: Truncated Cotangent})):
\begin{equation}
    \label{eqn: ConormalD}
\mathcal{I}/\mathcal{I}^2\rightarrow \Omega_{\mathcal{A}}^1\otimes_{\mathcal{A}}\mathcal{A}/\mathcal{I}\rightarrow \Omega_{\mathcal{A}/\mathcal{I}}^1\rightarrow 0,
\end{equation}
where $\Omega_{\mathcal{A}}^1=\mathcal{H}om_{\A[\mathcal{D}]}(\Theta_{\mathcal{A}},\mathcal{A}\otimes_{\mathcal{O}_X}\D_X)$ is a left $\mathcal{A}[\mathcal{D}]$-module and $\Theta_{\A}=\mathcal{D}er_{\D}(\mathcal{A},\mathcal{A})$ is naturally a right-module.
Apply the functor $\mathcal{H}om_{\mathcal{A}/\mathcal{I}\otimes\D_X}(-,\mathcal{A}/\mathcal{I}\otimes_{\mathcal{O}_X}\D_X)$ to (\ref{eqn: ConormalD}) gives
\begin{eqnarray}
\label{eqn: LES Conormal}
    0\rightarrow \mathcal{H}om_{\mathcal{A}/\mathcal{I}\otimes\D_X}(\Omega_{\mathcal{A}/\mathcal{I}}^1,\mathcal{A}/\mathcal{I}[\D])&\rightarrow& \mathcal{H}om_{\mathcal{A}/\mathcal{I}\otimes\D_X}(\Omega_{\mathcal{A}}^1\otimes_{\mathcal{A}}\mathcal{A}/\mathcal{I},\mathcal{A}/\mathcal{I}[\D]) \nonumber
    \\  &\rightarrow&\mathcal{H}om_{\mathcal{A}/\mathcal{I}\otimes\D_X}(\mathcal{I}/\mathcal{I}^2,\mathcal{A}/\mathcal{I}[\D])\rightarrow\cdots
    \end{eqnarray}
The vector fields on $\mathrm{Spec}_{\D}(\mathcal{A}/\mathcal{I}),$ are by definition $\mathcal{H}om_{\mathcal{A}/\mathcal{I}\otimes\D_X}(\Omega_{\mathcal{A}/\mathcal{I}}^1,\mathcal{A}/\mathcal{I}[\D])$, since $\mathcal{A}$ is smooth (as a jet $\D$-algebra), then $\Omega_{\A}^1$ is a projective $\A[\mathcal{D}]$-module and consequently,
$\mathcal{H}om_{\mathcal{A}/\mathcal{I}\otimes\D_X}(\Omega_{\mathcal{A}}^1\otimes_{\mathcal{A}}\mathcal{A}/\mathcal{I},\mathcal{A}/\mathcal{I}[\D])\simeq \Theta_{\A}\otimes_{\mathcal{A}}\mathcal{A}/\mathcal{I}$ by adjunction.

The long-exact sequence (\ref{eqn: LES Conormal}) is thus written,
\begin{eqnarray*}
0\rightarrow \Theta_{\mathcal{A}/\mathcal{I}}\rightarrow \Theta_{\mathcal{A}}\otimes_{\mathcal{A}}\mathcal{A}/\mathcal{I}&\rightarrow& \mathcal{H}om_{\mathcal{A}/\mathcal{I}\otimes\D_X}(\mathcal{I}/\mathcal{I}^2,\mathcal{A}/\mathcal{I}[\D])\rightarrow 
\\
&\rightarrow& Ext_{\mathcal{A}/\mathcal{I}[\D]}^1(\Omega_{\mathcal{A}/\mathcal{I}}^1,\mathcal{A}/\mathcal{I}[\D])\rightarrow \cdots.
\end{eqnarray*}
Then the second map sends a vector field $v$ to the deformation $\mathcal{I}_v\mapsto \mathcal{I}+\epsilon\cdot v(\mathcal{I}),$ of the ideal. Namely, first-order deformations of $\mathcal{I}$ over $\mathbb{C}[\epsilon]/(\epsilon)^2,$ correspond to elements of an associated ideal $\mathcal{I}_{\epsilon}\subset \mathcal{O}(J_X^{\infty}E)[\epsilon]/(\epsilon^2),$ such that $\mathcal{I}_{\epsilon}\otimes_{\mathbb{C}[\epsilon]}\mathbb{C}\simeq \mathcal{I}.$ Explicitly, such deformations are $\D$-homomorphisms $\phi:\mathcal{I}\rightarrow \mathcal{A}/\mathcal{I}$ so $\mathcal{I}_{\epsilon}\simeq \{F+\epsilon\cdot G|F\in \mathcal{I},G=\phi(F) \text{mod }\mathcal{I}\}.$

Its kernel consists of trivial deformations, therefore the $Ext^1$-group contains the obstructions to lifting a first-order deformation i.e. tangent vector, corresponding to an element of $\mathcal{H}om_{\mathcal{A}/\mathcal{I}\otimes\D_X}(\mathcal{I}/\mathcal{I}^2,\mathcal{A}/\mathcal{I}[\D]).$ Thus, we conclude that for $[\mathcal{I}]\in \mathcal{H}ilb_{\D}^P(\mathcal{A}),$
$$T_{[\mathcal{I}]}\mathcal{H}ilb_{\D}^P(\A)\simeq \mathcal{H}om_{\mathcal{A}/\mathcal{I}\otimes\D_X}(\mathcal{I}/\mathcal{I}^2,\mathcal{A}/\mathcal{I}[\D]).$$
Sequence (\ref{eqn: ConormalD}) is compatible with the filtrations, since $\mathcal{I}/\mathcal{I}^2$ inherits a filtration with components $\bigoplus_{k\geq 0}F^k\mathcal{I}+\mathcal{I}^2/F^{k+1}\mathcal{I}+\mathcal{I}^2,$ and then noting that for each $k\geq 0,$ the $k$-th geometric symbol is isomorphic to $F^k\mathcal{I}/F^{k+1}\mathcal{I}+F^k\mathcal{I}\cap F^1\mathcal{A}\cdot F^{k-1}\mathcal{A}.$ 
Immediately from the long-exact sequence, the statement concerning the space of obstructions is clear. We now establish its relation with the Spencer cohomologies and Spencer regularity of the point $[\mathcal{I}].$
Namely, since $F^k(\mathcal{I}/\mathcal{I}^2)=(F^k\mathcal{I}+\mathcal{I}^2)/\mathcal{I}^2,$ there exists a filtration 
$$F^{-q}Ext_{\mathcal{A}/\mathcal{I}[\D_X]}^1(\mathcal{I}/\mathcal{I}^2,\mathcal{A}/\mathcal{I}[\D]):=\{\psi|\psi(F^{\ell}(\mathcal{I}/\mathcal{I}^2)\subset F^{\ell-q}(\mathcal{A}/\mathcal{I}[\D_X]),\ell\geq 0\}.$$
An element $\psi\in F^{-q}\mathcal{O}bs_{[\mathcal{I}]}$ restricts as $\psi|_{F^q}:F^q(\mathcal{I}/\mathcal{I}^2)\rightarrow \mathcal{A}/\mathcal{I}[\D_X].$
Consider 
$$\mathrm{gr}_{-q}\mathcal{O}bs_{[\mathcal{I}]}:=F^{-q}Ext_{\mathcal{A}/\mathcal{I}[\D_X]}^1(\mathcal{I}/\mathcal{I}^2,\mathcal{A}/\mathcal{I}[\D])/F^{-q-1}Ext_{\mathcal{A}/\mathcal{I}[\D_X]}^1(\mathcal{I}/\mathcal{I}^2,\mathcal{A}/\mathcal{I}[\D]).$$
Then, by restriction, $\psi\in F_{-q}$ vanishes on $F^{q+1}(\mathcal{I}/\mathcal{I}^2)$ as $\psi \notin F_{-q-1}.$ Thus, descends to the quotient, denoted by
$\psi':F^q(\mathcal{I}/\mathcal{I}^2)/F^{q+1}(\mathcal{I}/\mathcal{I}^2)\rightarrow \mathcal{A}/\mathcal{I}[\D].$
Note the isomorphism
$F^q(\mathcal{I}/\mathcal{I}^2)/F^{q+1}(\mathcal{I}/\mathcal{I}^2)\simeq F^q\mathcal{I}/F^{q+1}\mathcal{I}+F^q\mathcal{I}\cdot F^1\A\simeq \mathrm{gr}_q(\mathcal{I}),$
of $\mathcal{O}_X$-coherent submodules of $\mathrm{gr}^q(\A)\simeq Sym^q(T^*X)\otimes E^*,$ as in (\ref{gr(I)}). 
Using a spectral sequence argument and the Spencer $\delta$-complex resolution for $\mathcal{I}/\mathcal{I}^2,$ one has 
$$Ext_{\A/\mathcal{I}\otimes\D_X}^i(\mathcal{I}/\mathcal{I}^2,\mathcal{A}/\mathcal{I}[\D])\simeq H^i\big(\mathcal{H}om_{\mathcal{A}/\mathcal{I}[\D]}(Sp(\mathcal{I}),\mathcal{A}/\mathcal{I})\big).$$
Thus, there is an isomorphism of $\mathcal{O}_X$-modules,
$\mathrm{gr}_{-q}\mathcal{O}bs_{[\mathcal{I}]}\simeq \mathcal{H}_{Sp}^{1,q}(\mathrm{gr}(\mathcal{I})), g\geq 0.$
By Proposition \ref{prop: Approximations}, the associated graded is isomorphic to a subquotient, therefore upon taking a direct sum, due to Spencer regularity bound for $\mathcal{I},$ there is an isomorphism 
$$\bigoplus_{q<\mathrm{Reg}_{\D}(\mathcal{I})}\mathrm{gr}_{q}Ext^1\simeq \bigoplus_{q<\mathrm{Reg}(\mathcal{I})}\mathcal{H}_{\mathrm{Sp}}^{1,q}(\mathrm{gr}(\mathcal{I})).$$

\end{proof}
The perfectness and finite-dimensionality (e.g. of Spencer cohomologies) is tied to properties of the differential operator defining the $\D$-ideal \emph{c.f.} (\ref{SpRmk}). We use microlocal stratifications of Definition \ref{defn: Non-singular Hilb}, to isolate those PDEs whose sheaf of solutions, e.g $R\mathcal{H}om_{\D_X},$ has finite-dimensional cohomologies.
In a general setting, consider a morphism $f:Z\rightarrow X$ of varieties and suppose $\Sigma_Z\subset Z,\Sigma_X\subset X$ are closed sub-manifolds. Denote $f_{\Sigma}$ the restriction and assume it is a closed embedding. There is a corresponding diagram,
\[
\begin{tikzcd}
    T^*Z & \arrow[l,"j_d"] Z\times_XT^*X\arrow[d] \arrow[r,"j_{\pi}"] & T^*X
\\
T_{\Sigma_Z}^*Z\arrow[u] \arrow[d,"\pi_{\Sigma_Z}"] & \arrow[l,"j_d"] \Sigma_Z\times_{\Sigma_X}T_{\Sigma_X}^*X\arrow[d,"\pi"] \arrow[r,"j_{\pi}"] & T_{\Sigma_X}^*X\arrow[u]\arrow[d,"\pi_{\Sigma_X}"]
\\
\Sigma_Z \arrow[r,"id"] & \Sigma_X \arrow[r,"j"] & \Sigma_X
\end{tikzcd}
\]

As a special case, take $(\Sigma\subset X),$ and consider a $\mathcal{D}_X$-algebra $\mathcal{B},$ defined by $[\mathcal{I}]\in \mathcal{H}ilb_{\D}^P(Z).$

\begin{defn}
\label{eqn: Elliptic Moduli}
\normalfont
The $\D$-Hilbert functor of $\D$-ideals \emph{elliptic with respect to} $\Sigma$, is defined by $\mathcal{H}ilb_{\D}^{P,\Sigma \times_X T_X^*X}(Z).$
\end{defn}
Definition \ref{eqn: Elliptic Moduli} generalizes the notion of ellipticity given in \cite[Sect. 5.3.1]{KSY}. Indeed, it states that (by pull-back along solutions), 
$$\mathbf{Char}_{\D}(\mathcal{I})\cap T_{\Sigma}^*X\subset \Sigma\times_X T_X^*X.$$ In particular, if $\Sigma$ is a real manifold $M$ and $X$ a complexification of $M$, then Definition \ref{eqn: Elliptic Moduli} recovers the usual notion of ellipticity: the linearized $\mathcal{D}_X$-module is elliptic i.e. $Char_{\mathcal{D},\varphi}(\mathcal{Y})\cap T_M^*X\subset M\times_X T_X^*X.$

\section{Spencer-stability is Gieseker stability}
In this final section we given an application of the ideas developed in this paper pertaining to a classification problem of PDE-theoretic nature. We follow \cite{Do,UY} (motivated by the examples given in Subsect. \ref{ssec: Overview and statement of results}), by introducing Spencer polystability, discuss a conjectural correspondence for moduli of polystable differential ideals. We prove the conjecture in a particular case.

\subsection{Correspondences for moduli of Spencer-stable sheaves}
Let $X$ be a compact Kähler manifold, $E$ a holomorphic vector bundle and $\mathcal{I}\subset \mathcal{O}(J_X^{\infty}E)$ a formally integrable, involutive $\D_X$-ideal sheaf encoding a (non-linear) PDE system. We propose to investigate the correspondence between Spencer stability of $\mathcal{I}$ and existence of a canonical (possibly singular) solutions to the PDE defined by $\mathcal{I}$.

In Theorem \ref{thm: DSUY implies DUY} we prove this correspondence in the case when $\mathcal{I}$ is the $\D$-ideal associated to the equation of flat-connections (constructed in the proof, given in Subsect. \ref{sssec: Proof of DSUY}). 
\begin{defn}
    \label{defn: Sp-polystable}
    \normalfont
    A $\D$-ideal $\mathcal{I}\subset \mathcal{O}(J_X^{\infty}E)$ is \emph{Spencer-polystable}
if we may write
$\mathcal{I}\simeq\mathcal{I}_1\oplus \cdots\oplus \mathcal{I}_k,$
with each $\mathcal{I}_i$ Spencer-stable with $\mu_{Sp}(\mathcal{I}_i)=\mu_{Sp}(\mathcal{I}).$
\end{defn}
The conjectural correspondence suggests to search for a canonical solution $\varphi\in H^0(X,\mathcal{A}/\mathcal{I})$ such that its graph $\Gamma_{\varphi}\subset J_X^{\infty}(E)$ minimizes a PDE-theoretic energy functional.
To this end, fix $\varphi\in \mathrm{Sol}_{\D}(\mathcal{I}),$ and a cycle in singular homology $H_{*c}(X)$. Then, there are natural assignments,
$H_{*,c}(X)\times \mathcal{H}^{*}(Z)\rightarrow \Omega_H^p,(K,\omega)\mapsto \lambda_{K,\omega},$
where $\lambda_{K,\omega}$ is defined for each open subset $U$ of $X$ by its values on $U$-parameterized sections $s_U\in \Gamma_K(X,E)(U)\subseteq Hom(X\times U,E),$ with support in compact subsets $K$, by 
$\lambda_{K,\omega}(s_U):=\int_{\sigma}(j_{\infty}(\varphi(x,u))^*\omega\in \Omega_U^p.$ Here $H$ is an appropriate sub-space of the space of solutions, via Proposition \ref{prop: Sol}. Integrating singular solutions this way, one may try to define
\emph{Spencer-type energy},
$$\mathcal{E}(\varphi,K):=\int |\!|\overline{\mathcal{H}}_{Sp}^{*}(\varphi)|\!|^2dVol,$$
where the $L^2$-norm on the right comes from a certain non-holomorphic Spencer defect e.g defined by a sub-complex of the Spencer complex, induced by the usual splitting, $T_X\otimes\mathbb{C}=T_X^{1,0}\oplus T_X^{0,1}$ with $\overline{T}_X^{1,0}=T_X^{0,1},$ whose cohomologies detect deformations of analytic solutions $\overline{\mathcal{H}}^{1,0}.$

That is, if $\varphi$ is not a holomorphic solution section, its deviation from being holomorphic is measured by some class
$\overline{\partial}_{Sp}(\varphi)$ in the cohomology of a natural sub-complex of the Spencer complex. If this class vanishes, $\varphi$ is a holomorphic. We interpret the class $[\overline{\partial}_{Sp}(\varphi)]$ as the (analytic) obstruction to solving the PDE defining $\mathcal{I}$ at $\varphi.$
\\

\noindent\textbf{Conjecture.}\emph{ Let $(X,\omega)$ be a compact Kähler manifold, $E$ a holomorphic vector bundle and $\mathcal{I}\subset \mathcal{O}(J_X^{\infty}E)$ a formally integrable, $\D$-involutive $\D$-ideal sheaf. Then, $\mathcal{I}$ is Spencer-polystable if and only if there exists a (possibly singular) solution $\varphi\in H^0(X,\mathcal{B})$ minimizing a Spencer-type functional e.g. $\mathcal{E}(\varphi).$}
\\

In the remainder of this subsection we prove a result which provides evidence for this conjecture.

\subsection{The equation of flat-connections}
Consider variables $(x_1,\ldots,x_n)\in X$ with $dim_X=n$ and $u=(u^1,\ldots,u^m)(x)$ in a rank $m$ bundle $E$ (fiber-wise coordinates). Fix a collection of $nm$-functions $g_i^{\alpha}=g_i^{\alpha}(x_1,\ldots,x_n,u^1,\ldots,u^m),$ which determine a connection $\nabla$ e.g. by $\nabla(\partial_i)=\partial_i+\sum_{\alpha=1}^mg_i^{\alpha}\partial_{u^{\alpha}},i=1,\ldots,n.$

Consider the system of $m n(n-1)/2$ non-linear PDEs in $m n$-unknowns,
\begin{equation}
    \label{eqn: Flat-connection equation}
    \frac{\partial g_j^{\alpha}}{\partial x_j}+\sum_{\beta=1}^mg_i^{\beta}\frac{\partial g_j^{\alpha}}{\partial u^{\beta}}=\frac{\partial g_i^{\alpha}}{\partial x_j}+\sum_{\beta=1}^mg_j^{\beta}\frac{\partial g_i^{\alpha}}{\partial u^{\beta}},\hspace{1mm} 1\leq i<j\leq n,\hspace{1mm}\alpha=1,\ldots,m.
\end{equation}

System (\ref{eqn: Flat-connection equation}) is the coordinate description expressing flatness of $\nabla$. It defines a $\D_X$-ideal, denoted $\mathcal{I}_{\nabla}.$ It may be given an invariant meaning using the language of $\D$-geometry in (\ref{eqn: Flat connection functions}) below, so that by Proposition \ref{prop: Sol}, any solution gives a flat connection $\nabla$. 
We verify the conjecture by showing that if $\mathcal{I}_{\nabla}$ is Spencer-polystable, then there exists a Hermitian-Yang-Mills metric on $E$ and conversely, if $E$ admits a HYM metric $h$, then $\mathcal{I}_{\nabla}$ is Spencer-polystable. 

The proof implicitly invokes Simpson's non-abelian Hodge theory and the classical DUY-correspondence. In particular, the key observation is that the Spencer complex for $\mathcal{I}_{\nabla}$ encodes the full deformation theory of $\nabla$ (see e.g. \cite{KP}) while symbol stability i.e. Spencer stability, matches the stability condition of Higgs bundles (Proposition \ref{prop: Stability compare}). 

\begin{thm}
\label{thm: DSUY implies DUY}
Let $(X,\omega)$ be a compact Kähler manifold, $(E,\nabla)$ a holomorphic vector bundle with flat connection i.e. a flat bundle. Consider the $\D$-ideal sheaf sequence 
$$\mathcal{I}_{\nabla}:=\{F_{\nabla}\}\hookrightarrow \mathcal{O}(J_E^{\infty}J_X^1E)\rightarrow \mathcal{B}_{\nabla},$$
where $\mathcal{B}_{\nabla}$ is the $\D$-algebra of functions associated with the equation of flat-connections.
Then, the ideal $\mathcal{I}_{\nabla}$ is Spencer-polystable if and only if the holomorphic bundle $(E,\nabla)$ admits a Hermitian-Yang-Mills metric.
\end{thm}

\subsubsection{Proof of Theorem \ref{thm: DSUY implies DUY}}
\label{sssec: Proof of DSUY}
We first construct the $\D$-scheme of flat-connections and describe the associated symbolic structure and its cohomological properties under Spencer-stability assumptions.
\begin{proof}
Given $(E,\nabla)$, the curvature of the connection $F_{\nabla}$ defines a differentially generated, formally-integrable $\D_X$-ideal sheaf $\mathcal{I}_{\nabla}.$ It is canonically obtained from (\ref{eqn: Flat-connection equation}) via prolongation.

Assume that the coefficients $\{\nabla_i^{\alpha}\}$ determining $\nabla$ (there are $dim(X)\cdot rank(E)$-many of them) are expressed generically via
\begin{equation}
    \label{eqn: Flat connection functions}
\{F_k\equiv F_k(x,u,\nabla_i^{\alpha})|i=1,\ldots,n,\alpha=1,\ldots,m, k=1,\ldots,r\}.
\end{equation}
We associate with (\ref{eqn: Flat connection functions})
a $\D$-scheme $\EQ_{F},$ with the property that points of
$\mathrm{Sol}_{\mathcal{D}_X}(\EQ_F)$ are connections in $E,$ using Proposition \ref{prop: Sol}. We sketch the proof. 
To this end, let $[\![-,-]\!]$ be the Fr\"olicher-Nijenhuis bracket on vector valued derivations, and consider the finite-rank vector bundle $J_X^1(E),$ viewed as a bundle over $E.$ Consider $J_E^1\big(J_X^1(E)\big).$ Via the well-known one-to-one correspondence between connection $1$-forms $Der(E,\Omega_E^1)$ satisfying $[\![U,U]\!]=0$ and flat-connections in $E,$ sections of $J_X^1(E)$ are first-order differential operators with scalar-type symbol. Let  
$\theta_1:=[\nabla]_{\theta}^1$ denote an element of $J_E^1(J_X^1(E)).$
This defines a function 
$$F(U):=[\![U,U]\!]:\mathcal{O}(J_{E}^1J_X^1(E))\rightarrow \mathbb{C},$$
via evaluation.
In particular, when $U=U_{\nabla}$ then 
$F(U_{\nabla})=0,$
is the equation of flat connections $\nabla$ in $E\rightarrow X.$
This is given a coordinate independent invariant meaning by setting
\begin{equation}
\label{eqn: FlatConns}
\EQ_{E}^{flat}:=\{\theta_1\in J_E^1(J_{X}^1E)| F(U)(\theta_1)=0\}.
\end{equation}
There is an induced action of differential operators $\partial_{x_i}$ and $\partial_{u^{\alpha}}$ on this bundle. In other words, the natural $\D$-action prolongs as a differentially stable system with $\D$-ideal $\mathcal{I}_{\nabla}.$
\begin{defn}
\normalfont
The \emph{$\D$-scheme of flat-connections,} is the closed $\D$-subscheme associated with the infinite-prolongation of (\ref{eqn: FlatConns}), given by
\begin{equation}
\label{eqn: D-Scheme of Flat Connections}
Z_{(E,\nabla)}:=\mathrm{Spec}_{\mathcal{D}}(\mathcal{O}_{\EQ_E^{flat,(\infty)}})\hookrightarrow J_X^{\infty}\big(J_X^1E\big).
\end{equation}
\end{defn}
By Definition \ref{defn: Algebraic solution}, an algebraic section 
$s:E\rightarrow J_X^1E$ 
is a flat-connection (solution) in $E$ if and only if $s(E)\subset \EQ_E^{flat}.$
This equation of flat connections is imposed in the bundle 
$J_E^1(J_X^1(E)).$ By Proposition \ref{prop: Sol}, we obtain the following.
\begin{prop}
There is an isomorphism of $\D$-spaces of sections
$\mathrm{Sol}_{E}(\EQ_E^{flat,\infty})\simeq \mathrm{Sect}(E,J_X^1E)\times_{\mathrm{Sect}(E,J_E^{\infty}J_X^1E)}\mathrm{Sect}(E,\EQ_{E}^{flat,\infty}).$
\end{prop}
This result is intuitively clear as solutions to $J_E^{\infty}(J_X^1E)$ (the `empty' equation) are sections of $J_X^1E$, which is dual to the Atiyah algebroid of $E,$ thus are connections in the bundle $E$. Then, solutions to the sub-$\D$-scheme of flat connections picks out the sub-space of sections of $J_X^1E,$ thus connections in $E$, which are additionally flat.

Since the $\D$-scheme of flat connections (\ref{eqn: D-Scheme of Flat Connections}) arises from a \emph{non-linear} system (\ref{eqn: Flat-connection equation}), we linearize and use the Spencer complexes in the $\D_{\A}$-module setting.

Namely, linearizing at a solution $\nabla$, we obtain the operator of (universal) linearization,
$$L_{\nabla}:\Omega^1(X;EndE)\rightarrow \Omega^2(X;EndE),\hspace{1mm}\eta\mapsto d_{\nabla}\eta.$$
Seen in another way, the symbolic modules $\mathrm{gr}(\mathcal{I}_{\nabla})$ associated with this $\D$-ideal correspond to the compatibility complex for linearized flatness equation $d_{\nabla}\phi=0,$ for $\phi\in \Omega^1(EndE).$ Looking at $\mathrm{gr}^*(\I_{\nabla})$, we have a compatibility complex for $L_{\nabla}$ isomorphic to
$$\mathcal{S}p(\mathcal{I}_{\nabla})\simeq 0\rightarrow End(E)\rightarrow \Omega^1(End(E))\rightarrow \Omega^2(End(E))\rightarrow \cdots \Omega^{n}(X;End(E))\rightarrow 0.$$
Then the Spencer slope (\ref{eqn: Spencer slope}) is 
$$\mu_{Sp}(\mathcal{I}_{\nabla}):=deg\big(H^0(\mathcal{S}p(\mathcal{I}_{\nabla})\big)/\mathrm{rank}(\mathcal{I}_{\nabla}).$$
Note
$H^0(\mathcal{S}p(\mathcal{I}_{\nabla}))\simeq ker(d_{\nabla}:End(E)\rightarrow \Omega^1EndE),$
and $\mathrm{rank}(\mathcal{I}_{\nabla})=rank(EndE)=rank(E)^2.$ 

Using non-abelian Hodge theory and the classical DUY-correspondence, we compare the stability conditions.
\begin{prop}
\label{prop: Stability compare}
    Spencer stability of $\mathcal{I}_{\nabla}$ is equivalent to Gieseker polystability of the Higgs bundle $(E,\theta),$ where $\theta$ is induced via the harmonic metric from $\nabla.$ 
\end{prop}
\begin{proof}
    Suppose $\mathcal{J}\subset \mathcal{I}$ is a non-zero sub $\D$-ideal. Assume it corresponds to a sub-bundle $F\subset E$ which is moreover $\nabla$-invariant i.e. $\nabla^F:=\nabla|_{F}$ is a sub-connection. Sub-connections and sub-$\D$-ideals are in correspondence this way, so we write $\mathcal{J}_{\nabla_F}.$ Consider $\mathrm{gr}_*(\mathcal{J}_{\nabla_F}),$ and reduced $\D$-Hilbert polynomials, i.e.
$$\frac{P_{\D}(\mathcal{J}_{\nabla_F},n)}{\mathrm{rank}(\mathcal{J}_{\nabla_F})}=\frac{deg H_{Sp}^0(\mathcal{S}p(\mathcal{J}_{\nabla_F}))}{rank(F)^2}=\frac{deg(ker(d_{\nabla_F}^0))}{rank(F)^2},$$
since $\mathrm{rank}(\mathcal{J}_{\nabla_F})=rank(F)^2.$ One may also verify $deg(ker_{\nabla})$ is determined by $rank(E)$ and $deg(E)$ by Hirzebruch-Riemann-Roch theorem.
Spencer-stability implies 
$$deg(ker(d_{\nabla_F}))/rank(F)^2\leq deg(ker(d_{\nabla}))/rank(E)^2,$$
which for stable sheaves is trivially satisfied, but for $\nabla$-invariant sub-bundles reduces to 
$\mu(F)\leq \mu(E),\mu(F)=deg(F)/rank(F),$ as usual.

Now we prove that $\mathcal{I}_{\nabla}$ is Spencer polystable if and only if there exists a harmonic metric $h$ solving $F_{\nabla}+[\theta,\theta^*]=0.$ 
To this end, assume $\mathcal{I}_{\nabla}$ is Spencer polystable. Then there exists a harmonic metric since $(E,\nabla)$ decomposes into stable summands of equal slope. 

This follows since jet-functors are additive, so $J_X^{\infty}(E)\simeq \bigoplus_i J_X^{\infty}(E_i),$ and under this decomposition the direct sum connection obviously block-diagonally. Moreover any $s\in \Gamma(X,E)$ maybe written $s=\sum_i s_i, s_i\in \Gamma(X,E_i).$ Then, their Taylor expansions $j_{\infty}(s)$ are of the form $\sum_i j_{\infty}(s_i),$ locally.

We use the following result, stated without proof.
\begin{prop}
Given a flat bundle $(E,\nabla)$, defined over some open neighbourhood $U$ of $X$, if $s\in \Gamma(U_0,E|_{U_0})$ is a flat section to the restriction $(E,\nabla)|_{U_0},$ where $U_0\subset U,$ then there exists a unique flat section $\widetilde{s}\in \Gamma(U,E)$ extending $s.$
\end{prop}

These facts imply $\mathcal{I}\simeq \mathcal{I}_1\oplus\cdots \oplus\mathcal{I}_k$ with $\mu_{Sp}(\mathcal{I}_i)=\mu_{Sp}(\mathcal{I})$. Then, since each $\mathcal{I}_i,i=1,\ldots,k$ corresponds to a stable flat bundle $(F_i,\nabla_{F_i})$ with $\mu(F_i)=\mu(E).$
By the classic DUY correspondence, each summand $(F_i,\nabla_{F_i}),i=1,\ldots,k$ admits a Hermitian-Yang-Mills metric $\beta_i,$ which glues to $h$. Conversely, suppose that $h$ is harmonic, this implies Spencer polystability since $h$ splits the flat bundle $(E,\nabla)$ into stable summands $(E_i,\nabla_i)$ of equal slope. Thus, $\mathcal{I}_{\nabla}$ decomposes into Spencer stable $\D$-ideals $\mathcal{I}_{\nabla_i},i=1,\ldots,k$. It is straightforward to check $\mu_{Sp}(\mathcal{I}_{\nabla_i})=\mu_{Sp}(\mathcal{I}).$ This concludes the proof of Proposition \ref{prop: Stability compare}.
\end{proof}
The proof of Theorem \ref{thm: DSUY implies DUY} is then concluded by adapting the well-known deformation-theoretic argument for unubstructedness of deformations of flat connections in holomorphic budles. 

\end{proof}

\appendix

\section{Proof of Proposition \ref{prop: Redux}}
\label{proof of redux}
\begin{proof}
  
We consider scalar PDEs $F=(F_1,\ldots,F_N)$ of orders $ord(F_A),$ for $A=1,\ldots,N,$ for simplicity, noting generalization to non-scalar PDEs is straightforward but more involved.

Each $F_A$ determines a corresponding map $J_X^k(E)\rightarrow J_{X}^{k-k_A}(X\times \mathbb{C}),$ and set 
$V:=\bigoplus_s V_s,$ the vector bundle with $dim(V_s)$ defined as the number of $A\in \{1,\ldots,N\}$ such that $ord(A)=s,$ and with $\sum dim(V_s)=N.$
Then $\mathsf{F}=(F_1,\ldots,F_N)$ determines a map
$$\mathsf{F}^{\infty}:J_X^{\infty}E\rightarrow J_{X}^{\infty}(V).$$
Set $\mathcal{V}_X:=\mathcal{O}(J_{X}^{\infty}V).$ 
 Then, pre-composition of $\mathsf{F}^{\infty}:\mathcal{A}_X\rightarrow \mathcal{V}_X,$ with other non-linear PDEs, we obtain a right-exact sequence
 \begin{equation}
     \label{eqn: Seq1}
     \mathrm{diff}(V,X\times\mathbb{C})\xrightarrow{-\mathsf{F}^{\infty}}\mathrm{diff}(E,X\times\mathbb{C})\rightarrow \mathcal{B}_X\rightarrow 0.
 \end{equation}
Letting $\mathcal{I}_X$ be the defining $\D$-ideal of $\mathcal{B}_X,$ then 
$$F^i\mathcal{I}_X\simeq Im(-\mathsf{F}^{\infty})\subset F^i\mathrm{diff}(E,X\times\mathbb{C}),$$
where $F^i\mathcal{I}$ is the usual filtration (\ref{eqn: FiniteIdeal}).
Consequently, 
\begin{equation}
    \label{eqn: B-cokernel}
F^i\mathcal{B}_X\simeq F^i(Coker(-\mathsf{F}^{\infty})\big)\simeq F^i\mathrm{diff}(E,X\times\mathbb{C})/F^i\mathcal{B}_X.
\end{equation}

Each term of (\ref{eqn: Seq1}) is a module over $\mathcal{A}[\D]$ and are moreover filtered:
$$\mathrm{Fil}^i\mathcal{A}[\D]\simeq \cup_{i,j}F^i\A_X\otimes \D_X^{\leq j},$$
which extends in the obvious way to 
$\mathrm{Fil}^i\mathcal{A}[\D](E,\mathcal{O})\simeq \bigcup_{i,j}F^i\A\otimes F^j\D_X(E,\mathcal{O}),$
and similarly for $\mathrm{Fil}^j\mathcal{B}_X[\D_X].$
Identifying $\mathrm{diff}(E,\mathcal{O}$ with $\mathcal{A}_X$ in the standard way,
we get from (\ref{eqn: B-cokernel}) an exact-sequence
$$0\rightarrow F^i\mathcal{I}\otimes F^i\D_X(E,\mathcal{O})\rightarrow \mathrm{Fil}^i\mathcal{A}[\D](E,\mathcal{O})\rightarrow \mathrm{Fil}^i\mathcal{B}_X[\mathcal{D}_X](E,\mathcal{O})\rightarrow 0.$$
Arguing in the exact same manner for $V$ and for $\mathcal{V}_X,$ we obtain a well defined filtered $\mathcal{B}_X^{\ell}$-module $\mathcal{M}_{\mathsf{F}^{\infty}},$ defined by the sequence
\begin{equation}
    \label{eqn: Seq2}
    \mathrm{Fil}^i\mathcal{B}[\mathcal{D}](V,\mathcal{O})\xrightarrow{lin_{\mathsf{F}}}\mathrm{Fil}^{i+k}\mathcal{B}_X[\mathcal{D}_X](E,\mathcal{O})\rightarrow F^{s+k}\mathcal{M}_{\mathsf{F}}\rightarrow 0.
\end{equation}
There is a canonical morphism $\mathrm{diff}(E,V)\rightarrow \A_X^{\ell}\otimes \D_X(E,V),$ sending $\mathsf{F}$ to $lin_{\mathsf{F}}.$
Consider the restriction of $\ell$ to the $i$-th filtered piece, we set 
$$Im(lin_{\mathsf{F}}^i)=lin_{\mathsf{F}}\big(\mathrm{Fil}^i\mathcal{B}_X[\D_X](V,\mathcal{O})\big).$$
This is a $F^i\mathcal{B}_X$-module and via base-chang along the canonical morphism $F^i\mathcal{B}_X\hookrightarrow F^{i+k}\mathcal{B}_X,$ we it generates a $F^{i+k}\mathcal{B}_X$-submodule in $\mathrm{Fil}^{i+k}\mathcal{B}_X[\mathcal{D}_X](E,\mathcal{O}).$
Via this base-change, by considering (\ref{eqn: Seq2}) we have that
\begin{equation}
\label{eqn: FM_F}
F^i\M_{\mathsf{F}}\simeq \mathrm{Fil}^i\mathcal{B}_X[\mathcal{D}_X](E,\mathcal{O})/F^i\mathcal{B}_X\cdot \mathrm{Im}(lin_{\mathsf{F}}^i).
\end{equation}
Using the canonical order filtration on differential operators, (\ref{eqn: Seq1}) is canonically filtered and subsequently taking quotient of two successive levels of the filtration, we obtain a sequence
$$F^{i-1}\M_{\mathsf{F}}\rightarrow F^i\mathcal{M}_{\mathsf{F}}\rightarrow \mathrm{Gr}_i^F(\M_{\mathsf{F}})\rightarrow 0,$$
as required.
\end{proof}

\section{Compatibility complexes and differential syzygy}
\label{sec: Compatibility complex}
We include here the construction of a compatibility complex for $\D_{\A}$-modules and an important proposition concerning its exactness.
\begin{cons}
\label{cons: Compatibility}
\normalfont
    For every $k_1\in \mathbb{N}\backslash\{0\},$ consider $\varphi_{P}^{k+k_1}:\overline{Jets}^{k+k_1}(\M_0)\rightarrow \overline{Jets}^{k_1}(\M_1),$ c.f. (\ref{eqn: InfiniteAhom}). We construct a compatibility complex for $P$ in a canonical way.
    If $k_1=0$ one may assume without loss of generality $\varphi_P^k$ is surjective. For $k_1>0,$ set 
    $\M_2:=\overline{J}^{k_1}(\M_1)/Im(\varphi_P^{k+k_1})\equiv coker(\varphi_P^{k+k_1}).$
    Denote by $q_1:\overline{J}^{k_1}(\M_1)\rightarrow \M_2,$ the quotient map. It is an element of 
    $$\mathcal{H}om(\overline{J}^{k_1}(\M_1),\M_2),$$ 
    thus corresponds to some $P_1\in \mathcal{A}[\D](\M_1,\M_2)$ of order $\leq k_1.$
    This representability result implies we may write $P_1=q_1\circ j_{k_1}.$ Now, 
    by construction,
    $$\M_0\xrightarrow{P}\M_1\xrightarrow{P_2}\M_2,$$
    is a complex of $\mathcal{A}[\D]$-modules, since
    $$P_1\circ P=(q\circ j_{k_1})\circ P=(q_1\circ \varphi_P^{k_1})\circ j_{k+k_1}=0.$$
One continues in this fashion: for $k_2$, we construct the quotient module $\M_3:=\overline{J}^{k_2}(\M_2)/Im(\varphi_{P_1}^{k+k_1+k_2})$, the quotient map $q_2$ and an operator $P_2$ so that $P_2\circ P_1=0.$ Eventually, we obtain a complex of $\mathcal{A}[\D]$-modules of the form:
$$0\rightarrow \M_0\xrightarrow{P}\M_1\xrightarrow{P_2}\M_2\xrightarrow{P_2}\cdots\rightarrow \M_i\xrightarrow{P_i}\M_{i+1}\rightarrow\cdots.$$
The question of its formal exactness is seen by considering $\mathsf{Sp}_{\mathcal{D}}(\M_0)^{\bullet}$, and arguing via (\ref{eqn: k-horizontal delta jet}).
\end{cons}
\begin{prop}
\label{prop: Exists CC formally exact}
    Suppose that $P\in \mathrm{Fil}^k\mathcal{A}[\D](\M_0,\M_1)$ is involutive in the above sense. Then there exists a compatibility complex described by \autoref{cons: Compatibility}, which is moreover formally exact for all integers $\{k_1,k_2,k_3,\ldots \in \mathbb{Z}_+\}.$
\end{prop}
\begin{proof}
    We need to show that for every $\ell\geq 1,$ the sequences
    $$\Sym^{k_{i-1}+k_i+\ell}(p_{\infty}^*\Omega_X^1)\otimes_\A\mathcal{M}_{i-1}\rightarrow \Sym^{k_i+\ell}(p_{\infty}^*\Omega_X^1)\otimes_{\A}\mathcal{M}_i\rightarrow \Sym^{\ell}(p_{\infty}^*\Omega_X^1)\otimes_\A\mathcal{M}_{i+1},$$
    are exact. Proceed by induction on $i,\ell$ and whose inductive step employs a spectral sequence argument applied to the bicomplex, where we denote $\overline{S}^k:=\Sym^{k}(p_{\infty}^*\Omega_X^1),$ for simplicity:
    \begin{equation}
\label{eqn: DiagramN}
\adjustbox{scale=.85}{
\begin{tikzcd}
 0\arrow[d] & 0\arrow[d] & 0\arrow[d] & 
\\
 \overline{\mathcal{N}}^{K+\ell}\arrow[d] \arrow[r]& p_{\infty}^*\Omega_X^1\otimes_{\A}\overline{\mathcal{N}}^{K+\ell-1}\arrow[d,]\arrow[r] & p_{\infty}^*\Omega_X^2\otimes_{\A}\overline{\mathcal{N}}^{K+\ell-2}\arrow[d]\arrow[d] \arrow[r] & \cdots 
\\
  \overline{S}^{K+\ell}\otimes_{\A}\M_0\arrow[d] \arrow[r]& p_{\infty}^*\Omega_X^1\otimes \overline{S}^{K+\ell-1}\otimes \M_0 \arrow[d]\arrow[r] & p_{\infty}^*\Omega_X^2\otimes \overline{S}^{K+\ell-2}\otimes \M_0\arrow[d]\arrow[r] & \cdots 
\\
 \vdots \arrow[d] & \vdots \arrow[d] & \vdots \arrow[d] & \vdots 
 \\
\overline{S}^{k_i+\ell}\otimes_{\A}\M_i\arrow[d] \arrow[r]& p_{\infty}^*\Omega_X^1\otimes \overline{S}^{k_i+\ell-1}\otimes \M_i \arrow[d]\arrow[r] & p_{\infty}^*\Omega_X^2\otimes \overline{S}^{k_i+\ell-2}\otimes \M_i \arrow[d]\arrow[r] & \cdots  
\\
\overline{S}^{\ell}\otimes_{\A}\M_{i+1}\arrow[d] \arrow[r]& p_{\infty}^*\Omega_X^1\otimes \overline{S}^{\ell-1}\otimes \M_{i+1} \arrow[d]\arrow[r] & p_{\infty}^*\Omega_X^2\otimes \overline{S}^{\ell-2}\otimes \M_{i+1}\arrow[d]\arrow[r] & \cdots
\\
0& 0 & 0&  
\end{tikzcd}}
\end{equation}
We have $K:=ord(P)+k_1+k_2+\cdots+k_i$.
\end{proof}

\section{Uniform bounds}
\label{Appendix Sweeney}
It is possible to give an explicit upper-bound for the Spencer-regularity of the family of $\D$-ideal sheaves as in Proposition \ref{prop: Boundedness 1}.
The characteristic $\D$-module defining the $\D$-characteristic variety (pulled back along a solution) as in proof of Proposition \ref{prop: D-Char is coisotropic}, is a submodule of the free $\mathcal{O}_{T^*X}\simeq\mathrm{Sym}^*(\Theta_X)$ module of finite rank and finitely generated by homogeneous elements, one can write down a bound. It is typically an overestimate \cite{Sw} depending only on values: $n,m,k$ with $k$ the order of the operator (corresponding to the degree of homogeneous polynomial). We now express it via recursion relations\footnote{It is given here for a single operator, but the generalization to systems (with possibly multiple orders) is straightforward but adds extra technicalities so we choose to omit it.}.
\begin{prop}
\label{prop: Sweeney Bound}
    There exists integers $\rho_1=\rho_1\big(n,m,\mathrm{Ord}(P)\big),$ which can be chosen such that
    \begin{equation}
        \label{eqn: Bounds}
        \begin{cases}
            \rho_1(0,m,1)=0;
            \\
            \rho_1(n,m,1)=m\cdot C_{n-1}^{\mathfrak{a}+n}+ \mathfrak{a}+1,\hspace{1mm}\text{ if } \mathfrak{a}=\rho_1(n-1,m,1);
            \\
            \rho_1(n,m,k_0)=\rho_1(n,\mathfrak{b},1),\hspace{1mm}\text{ if } \mathfrak{b}=\sum_{0}^{k_0}C_{n-1}^{\ell+n-1}\cdot m
        \end{cases}
    \end{equation}
If the operator $P$ is involutive then for all $\rho\geq \rho_0$ the symbol sequence it determines is exact.    
\end{prop}
In particular, for a family of $\D_X$-ideals $\{\mathcal{I}_j\}$ with orders $k_j$ (although we will just take pure orders) we write $\rho_1(0,m,1)=0$ and have that
$$\rho\big(n,m,\mathrm{Ord}_{\mathcal{D}_X}(\mathcal{I}_j)\big)=\rho\big(n,m\cdot  C_{n}^{k_j+n-1},1\big),$$
and 
$$
\rho(n,m,1)=m\cdot C_{n-1}^{\rho(n-1,m,1)+n}+\rho(n-1,m,1)+1.$$
\begin{prop}
Let $\mathcal{N}$ be the symbolic comodule such that $\mathcal{C}h_{k,1}=\mathcal{C}h_{k+1},$ for all $k\geq m_0.$ Then $\mathrm{Reg}_{\mathcal{D}_X}(\EQ)=m_0$ if and only if there exists a basis $\{\xi_1,\ldots,\xi_n\}$ in $T^*X,$ for which
$\mathrm{dim}(\mathcal{C}h_{k+1})=\sum_{\ell=1}^{n}\ell\cdot \alpha_k^{(\ell)},$
holds for $k=m_0.$
\end{prop}

\end{document}